\documentclass{article}
\usepackage{amssymb,amsmath,mathtools,amsthm,amsfonts,amsrefs,amscd,textcomp}
\usepackage[utf8]{inputenc}

\usepackage{subfigure}

\usepackage{paralist}
\usepackage{graphics} 
\usepackage[demo]{graphicx}
\usepackage[colorlinks=true]{hyperref}
\hypersetup{urlcolor=blue, citecolor=red}

\usepackage{nicefrac}       
\usepackage{microtype}      
\usepackage{latexsym,fancybox}
\usepackage{here}
\usepackage{tabularx}
\usepackage{longtable}
\usepackage{amsrefs}
\usepackage{lscape}
\usepackage{setspace}
\usepackage{optidef}
\usepackage{xcolor} 
\usepackage{caption}
\usepackage{geometry}
\usepackage{layout}
\usepackage[misc]{ifsym}
\usepackage[pagewise]{lineno}

\theoremstyle{plain}
\newtheorem{theorem}{Theorem}[section]
\newtheorem{lemma}{Lemma}[section]

\newtheorem{asm}{Assumption}[section]

\newtheorem{remark}{Remark}

\usepackage{algorithm}
\usepackage{algorithmic}

\newcommand{\R}{\mathbb{R}}
\newcommand{\g}{\nabla}
\newcommand{\fulg}[1]{\nabla F({#1})}

\newcommand{\isgd}[1]{\nabla f_{i_t}({#1})}

\newcommand{\E}{\mathbb{E}}
\newcommand{\iamtk}{\Tilde{A}^{k-1}_{i_t}} 
\newcommand{\am}{\Tilde{A}^{k-1}} 
\newcommand{\svc}{s^k}
\newcommand{\svct}{(s^k)^T}
\newcommand{\yvc}{y^k}
\newcommand{\xsk}{\Tilde{w}_k}

\newcommand{\tig}{\Tilde{g}}

\date{}
\title{A Stochastic Variance Reduced Gradient using Barzilai-Borwein Techniques as Second Order Information}

\begin{document}

\maketitle
\centerline{Hardik Tankaria\footnote{\Letter\ ~\url{hardik7@amp.i.kyoto-u.ac.jp}}}
\medskip
{\footnotesize
 \centerline{Department of Applied Mathematics and Physics}
   \centerline{ Graduate 
School of Informatics, Kyoto University}
   \centerline{Yoshida-Honmachi, Sakyo-ku, Kyoto 606-8501, Japan}
} 

\medskip

\centerline{Nobuo Yamashita\footnote{\Letter\ ~\url{nobuo@i.kyoto-u.ac.jp}}}
\medskip
{\footnotesize
 \centerline{Department of Applied Mathematics and Physics}
   \centerline{ Graduate 
School of Informatics, Kyoto University}
   \centerline{Yoshida-Honmachi, Sakyo-ku, Kyoto 606-8501, Japan}
}

\bigskip


\begin{abstract}
In this paper, we consider to improve the stochastic variance reduce gradient (SVRG) method via incorporating the curvature information of the objective function. We propose to reduce the variance of stochastic gradients using the computationally efficient  Barzilai-Borwein (BB) method by incorporating it into the SVRG. We also incorporate a BB-step size as its variant. We prove its linear convergence theorem that works not only for the proposed method but also for the other existing variants of SVRG with second-order information. 
We conduct the numerical experiments on the benchmark datasets and show that the proposed method with constant step size performs better than the existing variance reduced methods for some test problems.
\end{abstract}
\noindent
\section{Introduction}
  One of important tasks of supervised machine learning is to solve a certain optimization problems efficiently.
One of such optimization problems is the empirical risk minimization problem \cite{shalev2014understanding}. In this paper, we focus on the problem of
minimizing the empirical risk defined as
\begin{equation}\label{objectivefun}
   \min_{w \in \R^d} F(w) = \frac{1}{n} \sum_{i=1}^{n} f_i(w;a_i,b_i), 
\end{equation}
where $f_i(w;a,b):\R^d \to \R$ is the composition of a loss function $\ell$ and a prediction function $h$ parametrized by $w$, that is, $ f_i(w;a_i,b_i) = \ell(h(w;a_i);b_i)$.
For simplicity, we denote $ f_i(w;a_i,b_i) $ as $ f_i(w)$.
For example, when we consider the least squares regression of $n$ training samples $(a_i,b_i)_{i=1}^{n}$, where $a_i \in \R^d $ and $b_i \in \R$, the loss function is given as $f_i(w) = (w^Ta_i - b_i)^2 $. Moreover, in the case of the
$\ell_2$-regularized logistic regression we have $f_i(w) =  \log(1+\exp{(-b_ia_i^Tw)}) + \frac{\lambda}{2}\|w\|_2^2$, 
where $\lambda$ is a $\ell_2$-regularized parameter, $a_i \in \R^d$, and $b_i \in \{-1,1\}$ are the samples associated with 
a binary classification problem.

The growing amount of the data calls for a common challenge that arises in minimizing the objective function $F$ with a large $n$.
We focus on minimizing the average of twice continuously 
differentiable and strongly convex function, i.e., each $f_i(w)$ is strongly convex and twice continuously differentiable. In large-scale machine learning problems, solving such models is 
computationally expensive because of the large training data. In order to solve such large-scale optimization
problems, we need efficient optimization techniques. 

The classical full gradient method~\cite{grad_des_cauchy1847methode} 
computes the full gradient as the name indicates, and uses iterations of the form as follows,
\begin{equation}\label{GD}
    w_{k+1} = w_k - \eta_k \g F(w_k) = w_k  - \frac{\eta_k}{n} \sum_{i=1}^{n} \g f_i(w_k),
\end{equation}
where $\eta_k$ is a stepsize. The convergence rate of the full gradient is sublinear for convex and linear for strongly convex functions. However, it is computationally expensive as it computes
$n$ gradients $\nabla f_i$ at each iteration. Alternatively, the popular way to solve these large-scale optimization problems
is using the stochastic gradient methods which use either a single or small subset (mini-batch) of training data at every 
iteration. The stochastic gradient descent (SGD)~\cite{SGD_Robbins_1951} chooses 
a single training sample $i \in \{1,2, \ldots, n \}$ uniformly at random and update $w_k$ as follows,
\begin{equation}
    w_{k+1} = w_k  - {\eta_k} \g f_i(w_k),
\end{equation}
For global convergence, SGD requires a 
diminishing step size $\eta_k$. Therefore its convergence becomes slow because of the variance of stochastic gradient. 
We can expect that reducing the variance of stochastic gradient allows a constant step size and hence one can attain a faster convergence rate.

In the last decade, many researchers have been involved in an attempt to reduce the variance of the stochastic gradient
with constant step size~\cite{SVRG_NIPS2013, Xiao_2014,sag_roux2012stochastic,saga_defazio2014}. Stochastic average gradient (SAG)~\cite{Schmidt2017minimizing, sag_roux2012stochastic} proposed to use biased updates and achieves a linear convergence rate. Stochastic average gradient accelerated (SAGA)~\cite{saga_defazio2014} was proposed to use an unbiased estimate by improving SAG. Also, there has been some research on the dual minimization ~\cite{Shalev2013stochastic, Zhang2013linear}.
  
{The stochastic variance reduced} gradient~\cite{SVRG_NIPS2013} (SVRG) method explicitly reduce the 
variance through an unbiased estimate of gradient: 
\begin{flalign}\label{SVRG}
 {SVRG} && w_{k+1} & = w_k - \eta_k v_k,&
\end{flalign}
where $ v_k = \nabla f_i(w_k) - \nabla f_i(\Tilde{w}) + \nabla F(\Tilde{w})$ and $\Tilde{w}$ is the resultant vector at the end of previous epoch. Johnsan and Zhang~\cite{SVRG_NIPS2013} showed that variance of $v_k$ is diminishing as $w_k$ is approaching to a solution, and hence sequence converges to linearly with a constant step size. However, the standard SVRG does not exploit the information on the Hessian of $F.$ 

Much attempt has been devoted to stochastic variance reduced gradient methods with the second-order information.  {It is essential to know that second-order information can be applied in two different ways. Natural way is to use second-order approximation of objective function such as stochastic quasi-Newton method~\cite{1schraudolph2007stochastic, 2byrd2016stochastic, 3moritz2016linearly} and SVRG-BB~\cite{SVRGBB_NIPS2016}, etc.
\begin{flalign}\label{eq:QuasiNewton}
 {} && w_{k+1} & = w_k - \eta_k B_k^{-1}g_i(w_k),&
\end{flalign}
where $g_i(w_k)$ is any stochastic gradient which use first-order information such that $\E[g_i(w_k)] = \nabla F(w_k)$ and $B_k$ is an approximation of the Hessian  $[\nabla^2 f(w)]$.

An alternative way to use second-order information is to further reduce the variance via variance reduction technique of the stochastic gradient as
\begin{flalign}\label{eq:SVRG-2}
 {} && w_{k+1} & = w_{k} - \eta_k (v_k +
  {(\g^2 F(\Tilde{w}) - \g^2 f_{i_t}(\Tilde{w})) (w_{k} - \Tilde{w})}),&
\end{flalign}
where $v_k = \g f_{i_t}({w_{k}}) - \g f_{i_t}({\Tilde{w}}) + F(\Tilde{w})$.

Gower et. al., \cite{GOWER18a} proposed a variant of the SVRG which uses the second-order information to further control the variance of the stochastic gradient which is called the SVRG-2 method~\eqref{eq:SVRG-2}. It computes the Hessian with $O(nd^2)$ cost and the theoretical properties of SVRG-2 show that it provides better variance reductions and needs only a small numbers of epochs to converge. Additionally, they \cite{GOWER18a} proposed several Hessian approximations such as low-rank and diagonal Hessian. Note that when dimension $d$ is small, using SVRG-2~\cite{GOWER18a} with true Hessian may be affordable. However, SVRG-2 can be time  consuming when $d$ is large. Moreover, the linear convergence of diagonal and low-rank approximation is not shown in \cite{GOWER18a}.

In order to address these issues, we propose the Barzilai-Borwein approximation~\cite{BB_BARZILAI_1988} to further control the variance using SVRG.} 

{Table~\ref{tab:Different_Methods} shows the clarification among the various variants of SVRG which are first and second order methods with variance reduction through first and second order information.}
\begin{table}[!h]
    \centering
    \caption{various variance reduction methods}
    \begin{tabularx}{1\textwidth} { 
   | >{\raggedright\arraybackslash}X 
   || >{\centering\arraybackslash}X 
   | >{\centering\arraybackslash}X | }
  \hline
   & First-order model & Second-order model \\
  \hline
  \hline
  Variance reduction via first-order information & SVRG~\cite{SVRG_NIPS2013} & SVRG-BB~\cite{SVRGBB_NIPS2016} \qquad  \\
 \hline
 Variance reduction via second-order information  & SVRG-2 and SVRG-2D~\cite{GOWER18a}, SVRG-2BB (Proposed)  & SVRG-2BBS (Proposed)  \\
 \hline
\end{tabularx}
    \label{tab:Different_Methods}
    \vspace{-0.3cm}
\end{table}
\newline
We list our contributions as follows:
\begin{itemize}

\item We propose to use the Barzilai-Borwein(BB) method as a Hessian approximation to further reduce the variance of the stochastic gradient. We call this method SVRG-2BB.
\item We incorporate the Barzilai-Borwein step size for stochastic variance reduced gradient with the SVRG-2BB as another variant. We call the variant SVRG-2BBS.
\item We provide a simpler proof for the linear convergence results of the proposed methods for strongly convex functions. We prove the linear convergence for not only proposed methods, but also for SVRG-2 and SVRG-2D~\cite{GOWER18a}. 

\end{itemize}

This paper organized as follows: In section 2, we briefly discuss SVRG~\cite{SVRG_NIPS2013} and SVRG-2~\cite{GOWER18a} and its observations. In section 3, we first introduce the Barzilai-Borwein method (BB-method)~\cite{BB_BARZILAI_1988}, and then we propose SVRG-2 with BB-method as a Hessian approximation and BB-step size. In section 4, we provide the convergence analysis with simpler proof. In section 5, we conduct the numerical experiments for the proposed methods on $\ell_2$-regularized logistic regression and $\ell_2$-squared SVM. We compare the proposed methods with the existing SVRG and its variants. Finally we make some concluding remarks in section 6.



\section{The existing methods: SVRG and SVRG-2}
 In this section we introduce the SVRG~\cite{SVRG_NIPS2013} and the SVRG-2~\cite{GOWER18a}.

Generally, the stochastic gradient methods attempt to estimates 
the true gradient accurately. Since the expectation of the stochastic gradient $\nabla f_i(w)$ is the true gradient $\nabla F(w),$ i.e., $\E[\nabla f_i(w)]=\frac{1}{n} \sum_{i=1}^{n} \nabla f_i(w) =  \nabla F(w), i\in \{1,\ldots,n\}$. Thus, in order to get the best estimate of the stochastic gradient, it is natural to apply variance reduction techniques.

First we give a common framework of SVRG and SVRG-2.
\begin{algorithm}[!h]
 \caption{\textbf{\rm}} 
\begin{algorithmic}[1]
\label{alg:Algorithm_1}
  \item[\textbf{Parameters:}] update frequency $m$ and step size $\eta$, $\Tilde{g}$ is gradient of $F$ at $\Tilde{w}$
  \item[\textbf{Initialize}] $\Tilde{w}_0$
  \item[\textbf{1. for}] $k = 1, 2, \ldots$
  \item[\textbf{2. }] $\Tilde{w} = \Tilde{w}_{k-1}$
  \item[\textbf{3. }] $\Tilde{g} = \frac{1}{n} \sum_{i=1}^{n} \nabla f_i({\Tilde{w}})$
  \item[\textbf{4. }] $\Tilde{A}$ is given 
  \item[\textbf{5. }] $w_0 = \Tilde{w} $
  \item[\textbf{6. \quad for}] $t = 1,\ldots, m$
  \item[\textbf{7. \qquad }] Randomly pick $i_t \in \{1,\ldots,n\}$
  \item[\textbf{8. \qquad }] $\Tilde{A}_{i_t}$ such that $\mathbb{E}[\Tilde{A}_{i_t}] = \Tilde{A} $ is given 
  \item[\textbf{9. \qquad }] \textbf{$w_{t} = w_{t-1} - \eta\ (\isgd{w_{t-1}} - \isgd{\Tilde{w}} + \Tilde{g} + {(\Tilde{A} - \Tilde{A}_{i_t}) (w_{t-1} - \Tilde{w})})$}
  \item[\textbf{10. \quad end for}]
  \item[\textbf{11.\ option {1}}] $ \Tilde{w}_{k} = w_m$
  \item[\textbf{\ \ \quad option {2}}] $\Tilde{w}_{k} = w_{t}$ for randomly chosen $t \in \{0, \ldots,m-1\}$
  \item[\textbf{12. end for}]
  \item[\textbf{Output }] $\Tilde{w}_{k} = w_m$
\end{algorithmic}
\end{algorithm} 

The search direction of Algorithm 1 uses $ \g f_{i_t}({\Tilde{w}})$ at $\Tilde{w}$ for a current epoch $k$, and $\Tilde{w}$ does not update till the next epoch $k$. 
 If $\Tilde{w}$ does not change much, then it leads to slow convergence.
Therefore it is important to use the matrices $\Tilde{A} $ and ${A}_{i_t} $ which is computationally affordable and satisfies following properties:
\begin{itemize}
    \item[\rm {1}.] Unbiased estimate of $\Tilde{A}_{i_t}$, that is
        $$
    \E[\Tilde{A}_{i_t}] = \Tilde{A}.
        $$
    \item[\rm {2}.] Approximation of $\Tilde{A}_{i_t}$ and $\Tilde{A}$ such that,
        $$
        \Tilde{A}_{i_t} \approx \nabla^2 f_{i_t}(\Tilde{w})  \text{\ and \ } \Tilde{A} \approx \nabla^2 F(\Tilde{w}).
        $$
\end{itemize}
From property 1, we have $\mathbb{E}(\isgd{w_{t-1}} - \isgd{\Tilde{w}} + \Tilde{g} + {(\Tilde{A} - \Tilde{A}_{i_t}) (w_{t-1} - \Tilde{w})}) = \nabla F(w_{t-1}) $, and hence we can expect that Algorithm 1 finds a solution. The property 2 is to reduce the variance of $w_t$. 

Algorithm 1 includes the well-known SVRG and its variants as follows:
\begin{alignat}{2}
& \textbf{SVRG~\cite{SVRG_NIPS2013}}:\  && \Tilde{A} = \Tilde{A}_{i_t} = 0, \label{alg_svrg} \\
& \textbf{SVRG-2~\cite{GOWER18a}}   :\  && \Tilde{A} = \nabla^2 f(\Tilde{w}), \Tilde{A}_{i_t} = \nabla^2 f_{i_t} (\Tilde{w}) \text{\ with option  {1}}, \label{alg_svrg2nd}\\
& \textbf{SVRG-2D~\cite{GOWER18a}}   :\  && \Tilde{A} = \text{diagonal}(\nabla^2 f(\Tilde{w})), \Tilde{A}_{i_t} = \text{diagonal}(\nabla^2 f_{i_t} (\Tilde{w})) \ \text{with option {1}}, \label{alg_svrg2_diag}
\end{alignat}
where a $\text{diagonal}(G)$ denotes a diagonal matrix whose diagonal entries are $G$, that is
\begin{equation}\label{eq:diagonal_def}
        G_{i,j} = \begin{cases}
    g_{i,j} & \mbox{if the $i = j$}\\
    0 & \text{otherwise}.
    \end{cases}
  \end{equation}
 The SVRG proposed~\cite{SVRG_NIPS2013} computes variance reduced gradient where gradient of $F$ at $\Tilde{w}$, i.e., $\Tilde{g}$ does not update till the end of epoch. Hence its reduction of the variance is not much, which causes slow convergence. Moreover, it needs to use smaller step size to maintain the linear convergence.

   The SVRG-2 proposed by Gower et. al.~\cite{GOWER18a} exploits the difference of Hessian 
$(\nabla^2 F(\Tilde{w}) - \nabla^2 F_i(\Tilde{w})) $  by multiplying
with $(w_t - \Tilde{w})$, which uses the current parameter $w_t$ and updates at every iteration. In practice when the 
Hessian is sparse or $d$ is small, calculating $\g^2 F(\Tilde{w})$ and computing $\g^2 F(\Tilde{w}) (w_{t-1}-\Tilde{w})$ is 
computationally tractable. However, when $d$ is large, it is not the good idea to 
use original Hessian since it costs $O(nd^2)$.
The SVRG-2D proposed by Gower et. al.~\cite{GOWER18a} computes the diagonal Hessian, which does not perform well as seen in section 5.


\section{SVRG-2 with the Barzilai-Borwein approximation}

In this section, we first introduce the standard BB-method~\cite{BB_BARZILAI_1988}, then we propose to use the BB-method as $\Tilde{A}$ and $\Tilde{A}_{i_t}$ in Algorithm~\ref{alg:Algorithm_1} to further reduce the variance. Finally, we propose to use BB-step size along with the proposed BB variant of SVRG-2.

\subsection{Barzilai-Borwein method}

In this section we recall the Barzilai and Borwein method~\cite{BB_BARZILAI_1988} for solving unconstrained optimization deterministically.

Barzilai and Borwein~\cite{BB_BARZILAI_1988}  proposed a two-point step size gradient
method~\cite{BB_BARZILAI_1988} which is also known as the BB-method. 
This method is popular to solve unconstrained optimization problems which is
motivated from the quasi-Newton methods. 

Consider the unconstrained optimization problem:
\begin{equation*}
    \min_w f(w),
\end{equation*}
where $f$ is a differentiable function. 

The quasi-Newton method~\cite{Num_Opt_Nocedal_1999} solve
minimization problem by using the approximate Hessian which needs to satisfy the 
secant equation. The quasi-Newton method takes the iteration of the form:
$$
w_{k+1} = w_k -\eta_k B_k^{-1}\g f(w_k),
$$
where $ B_k$ is an approximation of the Hessian matrix of $f$ at the $w_k$. The 
approximate Hessian 
$B_k$ needs to satisfy the following secant equation~\cite{Num_Opt_Nocedal_1999}:
$$
B_k s^k = y^k,
$$
where $ s^k = w_{k} - w_{k-1}$ and $y^k = \g f(w_k) - \g f(w_{k-1})$ for $k\geq 1$. The well-known BFGS method requires $O(d^2)$ time and space complexities for computing $B_k$.

In practice when $d$ is large, it is time consuming to compute the $B_k$. The BB-method restricts $B_k$ as $B_k = \frac{1}{\alpha_k}\mathrm{I},$ with $\alpha > 0$ and gets $\alpha_k$ by the secant equation in the least square sense, i.e.,
\begin{equation}
\min_{\alpha} \left\| \frac{1}{\alpha}s^k - y^k\right\|^2,
\label{bb1-problem}
\end{equation}
Note that when $(s^k)^\top y^k > 0$ the solution of ~\eqref{bb1-problem}, 
$$\alpha_k = \frac{\|s^k\|^2}{(s^k)^Ty^k}.$$
Then it is clear that, 
\begin{equation}
    B_k = \frac{(s^k)^Ty^k}{{\|s^k\|^2}} \mathrm{I}.
    \label{bb-solution}
\end{equation}
An alternative way to get $\alpha_k$ can be given as follows.
$$
\min_{\alpha} \left\| s^k - {\alpha}y^k\right\|^2,
$$
and the solution is given by,

\begin{equation}
    \label{eq:alternate_bb}
    \alpha_k = \frac{{(s^k)}^Ty^k}{\|y^k\|^2}, \quad B_k = \frac{\|y^k\|^2}{(s^k)^Ty^k} \mathrm{I}.
\end{equation}

\begin{remark}
Note that when $f$ is strongly convex, we have $(s^k)^\top y^k > 0.$
\end{remark}
\subsection{SVRG-2 with the Barzilai-Borwein method (SVRG-2BB)}

In this subsection, we propose Algorithm 1 with the BB-methods~\eqref{bb-solution}. The previous subsection introduced two different approximations of $B_k $ of \eqref{bb-solution} and \eqref{eq:alternate_bb}. We adopts the
~\eqref{bb-solution} for the rest of the paper because~\eqref{bb-solution} is 
convenient for convergence analysis. 

Now we propose to use the BB-method for a Hessian approximation along with Algorithm 1. We call it the SVRG-2BB method.

\vspace{0.5cm}
\noindent\fbox{
\begin{minipage}{\textwidth}
    \textbf{SVRG-2BB:} Algorithm 1 with ~\eqref{bb-solution} for the approximate Hessian of $f_i(w)$, that is
    \begin{equation} \label{eq:svrg2bb_A}
    \Tilde{A}^k = \frac{\svct \yvc}{\|\svc\|^2} \mathrm{I} = \frac{1}{n} \sum_{i=1}^{n}
    \frac{\svct (\g f_{i}(\xsk) - \g f_{i}(\Tilde{w}_{k-1}))}{\|\svc\|^2},
    \end{equation}

    \begin{equation}
        \label{alg_svrgapp}
    \Tilde{A}^{k}_{i_t} = \frac{\svct (\g f_{i_t}(\xsk) - \g f_{i_t}(\Tilde{w}_{m-1}))}{\|\svc\|^2},
    \end{equation}
    where $s^k = \xsk - \Tilde{w}_{k-1}$ and $y^k = \g F(\xsk) - \g F(\Tilde{w}_{k-1})$. Hence, SVRG-2BB is Algorithm \ref{alg:Algorithm_1} with \eqref{eq:svrg2bb_A} in step 4 and \eqref{alg_svrgapp} in step 8.
\end{minipage}}
\vspace{0.4cm}

It is easy to see that $\E[\Tilde{A}^{k}_{i_t} ] = \Tilde{A}^k$ for the above approximations derived from the BB-method. Since the approximate Hessian of BB-method is $\alpha_k \mathrm{I}$, we can treated 
$\Tilde{A}^k$ and $\Tilde{A}_{i_t}^k$ as scalars. Then, it is clear that 
$(\Tilde{A} - \Tilde{A}_{i_t}) (w_{t-1} - \Tilde{w})$ is the same direction of 
the $(w_{t-1} - \Tilde{w}) $, and $v_t$ given as 
\begin{equation}
    v_t = \isgd{w_{t-1}} - \isgd{\Tilde{w}} + \Tilde{g} - 
  \iamtk (w_{t-1} -   \Tilde{w}) + \am (w_{t-1} - \Tilde{w}),
    \label{vk}
\end{equation}
where computational costs is $O(n)$ when $\nabla f_{i_t}(w)$ is provided.

\begin{remark}
It is important to note that $(s^k)^Ty^k = (w_{k} - w_{k-1})^T(\nabla f(w_{k}) - \nabla f(w_{k-1})) > 0 $ is not necessarily true when the objective function is not convex. In order to handle this issue, we give a simple remedy to use SVRG-2BB as follows:
\begin{equation*}
    \Bar{A}^k = \max\left\{ \frac{(s^k)^T(y^k)}{\|s^k\|^2}, \delta \right\},
\end{equation*}
where $\delta > 0$.
\end{remark}

\subsection{SVRG-2BB with a Barzilai-Borwein step size (SVRG-2BBS)}

In this subsection we propose SVRG-2BB with the BB-step size. The typical acceleration technique of the SVRG is to use a quasi-Newton method as the quadratic model of the objective function, and generates an iterate as
\begin{equation}\label{quasi_Newton_method}
    w_{k+1} = w_k - \eta_k B_{k}^{-1} (\nabla f_i(w_k) - \nabla f_i(\Tilde{w}) + \nabla F(\Tilde{w})),
\end{equation}
where $B_{k}$ is either Hessian or approximate Hessian of the objective
function~\eqref{objectivefun} at previous epoch.
One of such methods is SVRG-BB~\cite{SVRGBB_NIPS2016} which employs the second-order via the information
from BB-method and $\eta_k = 1/m$. 
SVRG-BB updates $B_{k}^{-1}$ as follows, 
$$B_{k}^{-1} := \frac{\|s_{k}\|^2}{(s_{k}^Ty_{k})}\mathrm{I}, $$
where $s^k = \Tilde{w}_{k-1} - \Tilde{w}_{k-2}$, $y^k = \nabla F(\Tilde{w}_{k-1}) - \nabla F(\Tilde{w}_{k-2})$ and $m$ is the length of an epoch.
Note that SVRG-BB is regarded as Algorithm~\ref{alg:Algorithm_1} with
$$\eta_k = \frac{ B_{k}^{-1}}{m}=\frac{1}{m}\cdot\frac{\|s_{k}\|^2}{(s_{k}^Ty_{k})}.$$

Similarly, we exploit the idea of using BB-step size with SVRG-2BB so that we can make use of the gradient information and parameters to further reduce the variance, i.e., to make the most use of gradient information from the BB-method. We call it SVRG-2BBS.

\vspace{0.4cm}
\noindent\fbox{
\begin{minipage}{\textwidth}
SVRG-2BBS: SVRG-2BB with variable stepsize $\eta_k^{t}$ defined by
\begin{equation}
    \eta_{k}^{t} = \frac{\xi_t}{m_1}\frac{\|\Tilde{w}_{k-1} - \Tilde{w}_{k-2}\|^2}{((\Tilde{w}_{k-1} - \Tilde{w}_{k-2})^T (\Tilde{g}_{k-1} - \Tilde{g}_{k-2}))}, \label{eq:svrgbb2s1}
\end{equation}
where $\xi_t > 0$.
\end{minipage}}
\vspace{0.4cm}

 Note that SVRG-BB~\cite{SVRGBB_NIPS2016} computes the BB-stepsize per epoch. Whereas in SVRG-2BBS, we control and update the BB-step size by multiplying an appropriate constant $\xi_t$. Note also that the denominator of the BB-step size in \eqref{eq:svrgbb2s1} usually becomes very large since the number of samples $m=2n$ is large. To overcome the difficulty, we modify $\eta_k$ by multiplying $\xi_t$ and by replacing constant $m$ to $m_1$.
 
 \begin{remark}
 It is noteworthy to mention that the update frequency in the inner \textbf{for} loop of SVRG-2BBS is $m = 2n$ (i.e., similar to SVRG). We replaced $m$ with $m_1$ only in the stepsize $\eta^t_k$ of SVRG-2BBS and not in the update frequency.
 \end{remark}
\section{Convergence analysis}
  In this section, we show the linear convergence of Algorithm~\ref{alg:Algorithm_1} when $F$ is strongly convex.

We first analyze the difference in the variances of the SVRG and Algorithm~\ref{alg:Algorithm_1}. We make the following assumptions.

\begin{asm} \label{strongly convex}
Each $f_i$ is $ \mu_i$-strongly convex, $i.e.,$ there exists $\mu_i > 0$ such that
\begin{equation*}
    f_i(w) \geq f_i(z) + \g f_i(z)^T(w-z) + \frac{\mu_i}{2} \|z-w\|^2~\forall w,z\in \R^d.
\end{equation*}
\end{asm}
\noindent Moreover the gradient of each $f_i$ is $L$-Lipschitz continuous, $i.e.,$
\begin{equation*}
    \|\g f_i(w) - \g f_i(z)\| \leq L \|w-z\| ~\forall w,z \in \R^d.
\end{equation*}

\noindent Under this assumption, it is clear that $\g F(w)$ is also $L$-Lipschitz continuous:
\begin{equation}\label{smoothness}
    \|\g F(w) - \g F(z)\| \leq L \|w-z\| ~\forall w,z \in \R^d.
\end{equation}

Next lemmas are useful for our analysis. We omit its proof since it is directly follows from ~\eqref{smoothness}.
\begin{lemma}
If $f_i$ is $L$-Lipschitz continuous, then
\begin{equation*}\label{syl-Lipschitz_inequality}
     (w-z)^T (\g f_i(w) - \g f_i(z)) \leq L \| w - z \|^2,
\end{equation*}
\end{lemma}
and hence $s^Ty \leq L \|s\|^2$ for $s = w - z$ and $y = \g f_i(w) - \g f_i(z)$. 

The next lemma gives the result of the Lipschitz continuous Hessian.
\begin{lemma}~{\rm{\cite{Nesterov_2004}}}\label{lips-Hess_lemma}
If each $f_i : \R^d \to \R$ be twice continuous differentiable with $\Tilde{L}_t$-Lipschitz 
continuous Hessian, then for any $w,y \in \R^d$ we have,
\begin{equation*}
    \|{\isgd{w} - \isgd{z} -\g^2 f_{i_t}({w}) (w-{z}) }\| \leq 
        \frac{\Tilde{L}_t}{2}\|w - z \|^2.
\end{equation*}
\end{lemma}
\noindent Next two Lemmas give results from the strongly convexity.
\begin{lemma}\label{lem:Grad_Hess_ineq}~{\rm \cite{Nesterov_2004}}
If each $f_i$ is $\mu$-strongly convex, then for any $w,y \in \R^d$ we have
$$
\| f_i(w) - f_i(z) \| \geq \mu_i \|w - z \|.
$$
\end{lemma}
\begin{lemma}~{\rm\cite{Nesterov_2004}}
If $F$ is $\mu$-strongly convex function, $w_*$ is a minimum ${F(w)}$,
$$
F(w) - F(w_*) \geq \frac{\mu}{2} \| w - w_* \|^2
$$
for all $w \in \R^d.$
\label{lemma-str-con-fe}
\end{lemma}
Now we compare the variance of the SVRG-2 method with that of the original SVRG. 
 Let $v^{svrg}_{t} = \isgd{w_{t-1}} - \isgd{\Tilde{w}} + \nabla F(\Tilde{w}) $
for the SVRG. Recall that $\mathbb{E}[v_t^{svrg}] =  \g F(w_{t-1})$. Moreover we have 

\begin{align}
        \E[\|v^{svrg}_t - \fulg{w_{t-1}} \|^2] & = \E[\|\isgd{w_{t-1}} - \isgd{\Tilde{w}} + \nabla F(\Tilde{w}) - \fulg{w_{t-1}}\|^2] \nonumber\\
        & = \E[\| \isgd{w_{t-1}} - \isgd{\Tilde{w}} - \E[\isgd{w_{t-1}} - \isgd{\Tilde{w}}]\|^2]\nonumber \\
        & \leq \E[\|\isgd{w_{t-1}} - \isgd{\Tilde{w}}\|^2 ] \nonumber\\
        & = O(\|w_{t-1} - \Tilde{w}\|^2), \label{eq:svrg_variance}
\end{align}
where the inequality follows from the fact that 
$\E[\|\xi - \E[\xi]\|^2]=\E[\| \xi\|^2] - \|\E[\xi]\|^2\leq \E[\|\xi\|^2]$.
 Let $v^{2nd}_{t} = \isgd{w_{t-1}} - \isgd{\Tilde{w}} + \nabla F(\Tilde{w}) +{(\g^2 F(\Tilde{w}) - \g^2 
f_{i_t}(\Tilde{w})) (w_{t-1} - \Tilde{w})}$ for the SVRG-2 method. We see that $\E[v_t^{2nd}] = \nabla F(w_{t-1})$. Moreover we have
\begin{align}
    &\E[\|v^{2nd}_t - \fulg{w_{t-1}} \|^2]  \nonumber\\
    & =  \E[\|{(\isgd{w_{t-1}} - \isgd{\Tilde{w}} + \nabla F(\Tilde{w}) )} -\g^2 f_{i_t}(\Tilde{w}) (w_{t-1}-\Tilde{w}) \nonumber \\
    & \qquad \qquad \qquad \qquad \qquad \qquad  \qquad +\g^2 
F(\Tilde{w}) (w_{t-1} -\Tilde{w}) - \fulg{w_{t-1}}\|^2 ]\nonumber\\
    & = \E[\|{(\isgd{w_{t-1}} - \isgd{\Tilde{w}} -\g^2 f_{i_t}(\Tilde{w}) (w_{t-1}-\Tilde{w}) )} \nonumber\\
    & \qquad \qquad \qquad \qquad \qquad - \E[\| {
    \isgd{w_{t-1}} - \isgd{\Tilde{w}} -\g^2 f_{i_t}(\Tilde{w}) (w_{t-1}-\Tilde{w})}\|]\|^2] \nonumber \\
    & \leq \E[\|{(\isgd{w_{t-1}} - \isgd{\Tilde{w}} -\g^2 f_{i_t}(\Tilde{w}) (w_{t-1}-\Tilde{w}) )}\|^2 ]\nonumber\\
    & = O(\|w_{t-1} - \Tilde{w}\|^4), \label{eq:2nd_variance}
\end{align}
where the first inequality follows from the fact that $\E[\|\xi - \E[\xi]\|^2]=\E[\| \xi\|^2] - \|\E[\xi]\|^2\leq \E[\|\xi\|^2]$, and the last inequality follows from Lemma~\ref{lips-Hess_lemma} with $z = w_{t-1}$ and $w = \Tilde{w}$.

%
%
%
%
%
%

The inequality \eqref{eq:svrg_variance} and \eqref{eq:2nd_variance} suggests that the variance of $v_t^{2nd}$ is smaller than that of $v_t^{svrg}$ when $w_{t-1}$ is close to $\Tilde{w}$.
To reduce the computational cost, Algorithm~\ref{alg:Algorithm_1} exploits $\Tilde{A}^{k-1}$ and $\Tilde{A}^{k-1}_{i_t}$ as $\nabla F(w_{k-1})$ and $\nabla F_{i_t}(w_{k-1})$. However, we do not have the variance bound~\eqref{eq:svrg_variance} for arbitrary $\Tilde{A}^{k-1}$ and $\Tilde{A}^{k-1}_{i_t}$.

Thus we suppose the following assumption on $\Tilde{A}^{k-1}$ to prove the linear convergence of Algorithm~\ref{alg:Algorithm_1}.

\begin{asm}\label{alphaassumtion}
There exists a positive constant $\alpha$ such that for all $t$ and $k$,
\begin{equation}
    \E\left[ \|{(\isgd{w_{t-1}} - \isgd{\Tilde{w}} -\iamtk (w_{t-1}-\Tilde{w}) )}\|^2 \right] \leq \alpha 
    \E \left[ \|\isgd{w_{t-1}} - \isgd{\Tilde{w}} \|^2 \right]
\end{equation}
\end{asm}
\noindent Note that, due to~\eqref{eq:2nd_variance}, SVRG-2 satisfies Assumption~\ref{alphaassumtion} when $w_{t-1}$ is close to $\Tilde{w}$.

The next lemma gives the bound of the gradient difference, and it is useful for further analysis.
\begin{lemma}\label{svrg_lipschitz}
~{\rm\cite{SVRG_NIPS2013}}~ Suppose that Assumption \ref{strongly convex} holds. Let $ w_* = \min_w f_i(w) - f(w_*)  - \g f_i(w_*)^T (w - w_*)$. Then
\begin{equation}
    \frac{1}{n} \sum_{i=1}^{n} \|\isgd{w}- \isgd{w_*}\|^2 \leq 2L [F(w) - F(w_*)].
\end{equation}
\end{lemma}
\noindent Under Assumptions \ref{strongly convex} and \ref{alphaassumtion}, we provide an upper bound on the variance.
\begin{lemma}\label{lem:variance_bnd}
Suppose that Assumptions \ref{strongly convex} and \ref{alphaassumtion} hold. Let $v_t$ be
\begin{equation}
    v_t = (\isgd{w_{t-1}} - \isgd{\Tilde{w}} + \tig) {+(\am - \iamtk) (w_{t-1}-\Tilde{w})}
    \label{searchdirection}
\end{equation}
Then variance  
$\E[\|v_t - \fulg{w_{t-1}} \|^2]$ is 
bounded by $4\alpha L (F(w_{t-1}) -  F(w_*)  + F(\Tilde{w}) -  F(w_*))$.
\end{lemma}
\begin{proof}
Recall that $v_t$ is an unbiased estimate of 
$\fulg{w_{t-1}}$, i.e., $\E[v_t] = \fulg{w_{t-1}}$.
Then its variance can be calculated as
\begin{align*}
    &\E[\|v_t - \fulg{w_{t-1}} \|^2] \\
    & =  \E[\|{(\isgd{w_{t-1}} - \isgd{\Tilde{w}} + \tig)} -\iamtk(w_{t-1}-\Tilde{w})+\am (w_{t-1} -\Tilde{w}) - 
    \fulg{w_{t-1}}\|^2] \nonumber\\
    & = \E[\|{(\isgd{w_{t-1}} - \isgd{\Tilde{w}} -\iamtk (w_{t-1}-\Tilde{w}) )} \nonumber\\
    & \qquad \qquad \qquad \qquad \qquad \qquad \qquad \qquad - \E[ {
    \isgd{w_{t-1}} - \isgd{\Tilde{w}} -\iamtk (w_{t-1}-\Tilde{w})}]\|^2]\\
    & \leq \E[\|{(\isgd{w_{t-1}} - \isgd{\Tilde{w}} -\iamtk (w_{t-1}-\Tilde{w}) )}\|^2].
\end{align*}
Under the Assumption~\ref{alphaassumtion}, it is easy to see that we have
\begin{align}
       \E[\|v_t - \fulg{w_{t-1}} \|^2] & \leq \alpha \E[\|\isgd{w_{t-1}} - \isgd{\Tilde{w}}\|^2] \nonumber\\
       & \leq \alpha (2\E[\|\isgd{w_{t-1}} - \isgd{w_*} \|^2] + 2\E[\| \isgd{\Tilde{w}} - \isgd{w_*}\|^2])\nonumber\\
       & \leq 4\alpha L (F(w_{t-1}) -  F(w_*)  + F(\Tilde{w}) -  F(w_*)).\label{variance}
\end{align}
where the last inequality follows from the Lemma~\ref{svrg_lipschitz}.
\end{proof} 

Now we proceed to prove the main result. The proof is similar to ~\cite{SVRG_NIPS2013}. For the completeness, we give the proof.
\begin{theorem}\label{thm:fun_linear}
Consider Algorithm~\ref{alg:Algorithm_1} with option-$\rm {2}$. Suppose that Assumptions~\ref{strongly convex} and \ref{alphaassumtion} hold. Let $w_*$ be the optimal solution of the problem~\eqref{objectivefun}, and $m$ is
sufficiently large so that 
\begin{equation*}
    \beta = \left[ \frac{1}{\mu \eta (1 - \eta L (2\alpha + 1))m} + 
    \frac{2 L \eta \alpha}{1 - \eta L (2\alpha + 1)}
    \right] < 1,
\end{equation*}
then we have linear convergence in expectation:
$$\E[F(\xsk) - F(w_*)] \leq \beta^k~\E[F(\Tilde{w}_{0}) - F(w_*)].$$
\end{theorem}
\begin{proof}
Let $v_t = (\isgd{w_{t-1}} - \isgd{\Tilde{w}} + \tig) {+(\am - \iamtk) (w_{t-1}-\Tilde{w})}$ be the search direction in the $k$-th epoch of Algorithm~\ref{alg:Algorithm_1}. Then,
\begin{align}
    \E[\|v_t\|^2] & = \E[\|v_t - \fulg{w_{t-1}} +\fulg{w_{t-1}} - \fulg{w_*}\|^2] \nonumber \\
    & = \E[\|v_t - \fulg{w_{t-1}}\|^2] + \E[\|\fulg{w_{t-1}} - \fulg{w_*}\|^2] \nonumber \\
    & \leq 4\alpha L (F(w_{t-1}) -  F(w_*)  + F(\Tilde{w}) -  F(w_*)) + 2L [F(w_{t-1}) - F(w_*)] \\
    & = 2L (2\alpha + 1) [F(w_{t-1}) - F(w_*)] + 4 \alpha L  [F(\Tilde{w}) - F(w_*)] \nonumber\\
    & = 2L [ (2\alpha + 1) (F(w_{t-1}) - F(w_*)) + 2\alpha (F(\Tilde{w}) - F(w_*)) ], \label{vbound}
\end{align}
where the first inequality follows from the Lemma~\ref{lem:variance_bnd} and Lemma~\ref{svrg_lipschitz}.


Now we bound the distance of $w_t$ to $w_*$.
\begin{align}
   & \E[\|w_{t} -   w_*\|^2] \nonumber\\
   & = \E[\|w_{t-1} - \eta v_t - w_*\|^2] \nonumber\\
    & = \|w_{t-1} - w_*\|^2 - 2 \eta~\E((w_{t-1} - w_*)^T v_t) + \eta^2~\E[\|v_t\|^2] 
    \nonumber\\
    & = \|w_{t-1} - w_*\|^2 - 2 \eta (w_{t-1} - w_*)^T \fulg{w_{t-1}} + \eta^2 
    \E[\|v_t\|^2] \nonumber\\
    & \leq \|w_{t-1} - w_*\|^2  -2\eta [F(w_{t-1}) - F(w_*)] \nonumber\\
    & \qquad \qquad \qquad  + 2 \eta^2 L [(2\alpha 
    + 1)~(F(w_{t-1}) - F(w_*)) + 2\alpha~(F(\Tilde{w}) - F(w_*))]\nonumber \\
    & = \|w_{t-1} - w_*\|^2 - 2 \eta (1 - \eta L (2\alpha + 1))~[F(w_{t-1}) - 
    F(w_*)] \nonumber\\
    & \qquad \qquad \qquad \qquad \qquad \qquad \qquad \qquad \qquad \quad
    + 4 \alpha \eta^2 L ~[F(\Tilde{w}) - F(w_*)], \label{linear_x}
\end{align} 
where the first inequality is  obtained by the convexity of $F$.
By summing the previous  inequality over $t = 0,1,\ldots,m-1$ and taking expectation with all the history using
option $\rm {2}$ at epoch $k$, we have 
\begin{align*}
    \E[\|w_m - w_*\|^2] + 2 \eta (1 - \eta L (2\alpha + 1)) & m~\E[F(\xsk) - F(w_*)] \\
    &\leq \E[\|w_0 - w_*\|^2] + 4 \alpha \eta^2 L m~\E[F(\Tilde{w}) - F(w_*)]
\end{align*}
and
\begin{align*}
    & \E[\|\Tilde{w} - w_*\|^2] + 4 L m \eta^2 \alpha~\E[F(\Tilde{w}) - F(w_*)] \\
    &  \leq \frac{2}{\mu} \E[F(\Tilde{w}) - F(w_*)] +  4 \alpha \eta^2 L m~\E[F(\Tilde{w}) - F(w_*)] \\
    &  = 2 \left( \frac{1}{\mu} + 2 L m \eta^2 \alpha \right) ~\E[F(\Tilde{w}) - F(w_*)],
\end{align*}
where the second inequality uses the strong convexity of $F$. Finally, we have

\begin{equation}
    \E[F(\xsk) - F(w_*)] \leq \left[ \frac{1}{\mu \eta (1 - \eta L (2\alpha + 1))m} + 
    \frac{2 L \eta \alpha}{1 - \eta L (2\alpha + 1)}
    \right]~\E[F(\Tilde{w}_{k-1}) - F(w_*)].
\end{equation}

\noindent which implies,
$$
 \E[F(\xsk) - F(w_*)] \leq \beta^k~\E[F(\Tilde{w}_{0}) - F(w_*)]
$$
\noindent with
\begin{equation*}
    \beta = \left[ \frac{1}{\mu \eta (1 - \eta L (2\alpha + 1))m} + 
    \frac{2 L \eta \alpha}{1 - \eta L (2\alpha + 1)}
    \right].
\end{equation*}
\end{proof}

Next we consider Algorithm~\ref{alg:Algorithm_1} with option-$\rm {1}$.
\begin{theorem} \label{Th:svrg-2_opt1}
Consider Algorithm~\ref{alg:Algorithm_1} with option-$\rm {1}$. Suppose that
Assumptions~\ref{strongly convex} and~\ref{alphaassumtion} hold. Let $w_*$ be
the optimal solution of the problem~\eqref{objectivefun}, and
\begin{equation}\label{Th:svrg2_coef}
   \gamma = \left[(1-2 \eta \mu(1 - \eta L (2\alpha + 1)))^m + \frac{2\alpha \eta L^2}{\mu(1 - \eta L (2\alpha + 1))}
 \right].
\end{equation}
Then for $k$-th epoch, we have 
\begin{equation}
    \E\|\Tilde{w}_k - w_*\|^2 \leq \gamma\ \|\Tilde{w}_{k-1} - w_*\|^2.
\end{equation}
\end{theorem}
\begin{proof}
We bound the distance of $w_t$ to $w_*$ using ~\eqref{linear_x},
\begin{align}
 & \E[\|w_{t} -   w_*\|^2] \nonumber\\
 & = \|w_{t-1} - w_*\|^2 - 2 \eta (1 - \eta L (2\alpha + 1))~[F(w_{t-1}) - 
    F(w_*)] + 4 \alpha \eta^2 L ~[F(\Tilde{w}) - F(w_*)] \nonumber\\
 & \leq  \|w_{t-1} - w_*\|^2 - 2 \eta (1 - \eta L (2\alpha + 1))~[\g F(w_{t-1})^T(w_{t-1} - w_*)]   \nonumber\\
 & \qquad \qquad \qquad \qquad\hspace{6cm}+ 4 \alpha \eta^2 L ~[\g F(\Tilde{w})^T(\Tilde{w} - w_*)] \nonumber\\
 & \leq \|w_{t-1} - w_*\|^2 - 2 \eta \mu(1 - \eta L (2\alpha + 1))\|w_{t-1} - w_*\|^2 +
 4 \alpha \eta^2 L^2 ~\|\Tilde{w} - w_*\|^2 \nonumber\\
 & = (1-2 \eta \mu(1 - \eta L (2\alpha + 1)))\|w_{t-1} - w_*\|^2 + 4 \alpha \eta^2 L^2 ~\|\Tilde{w} - w_*\|^2, \label{eq:linear_w}
\end{align}
where the first inequality follows from the convexity of $F$,
the second inequality follows from the strong convexity of $F$ and the Lipschitz continuity of $\nabla F$.

By applying the above inequality over the $t$ and
$\Tilde{w}_{k-1} = w_0$ and $\Tilde{w}_k = w_m$, we have
 \begin{align}
 & \E[\|\Tilde{w}_{k} -   w_*\|^2] \nonumber\\
 &\leq [1-2 \eta \mu(1 - \eta L (2\alpha + 1))]^m\|\Tilde{w}_{k-1} - w_*\|^2\nonumber\\
 & \qquad \qquad \qquad \qquad \hspace*{2cm}+ 4 \alpha \eta^2 L^2 ~
 \sum_{j=1}^{m}[1-2 \eta \mu(1 - \eta L (2\alpha + 1))]^j
 \|\Tilde{w}_{k-1} - w_*\|^2 \nonumber\\
 & < \left[(1-2 \eta \mu(1 - \eta L (2\alpha + 1)))^m + \frac{2\alpha \eta L^2}{\mu(1 - \eta L (2\alpha + 1))}
 \right] \|\Tilde{w}_{k-1} - w_*\|^2 \nonumber\\
 & =\gamma\  \|\Tilde{w}_{k-1} - w_*\|^2 \nonumber
\end{align}
Hence, the linear convergence follows,
\begin{equation*}
    \E[\|\Tilde{w}_k - w_*\|^2] <  \gamma\  \|\Tilde{w}_{k-1} - w_*\|^2.
\end{equation*}
\end{proof}
Note that $\gamma$ is less than $1$ when the stepsize $\eta $ is sufficiently small. Then $\Tilde{w}$ converges $w_t$ linearly. 

\noindent Next, we show that Assumption~\ref{alphaassumtion} holds for the SVRG-2, SVRG-2D, and the SVRG-2BB variants, that is, these methods converge linearly. 
For the rest of analysis, we consider $\nabla^2 F_i$ is to be Lipschitz continuous with constant $\Tilde{L}_i$ such that $\Tilde{L} = \max_i \Tilde{L}_i$ and $\mu = \min_i \mu_i$ for strong convexity of $F$.
\begin{lemma}
Suppose that there exists $M$ such that $\|w_{t-1} - \Tilde{w}\|^2 \leq M.$ Suppose also that $\nabla^2 F_i$ is Lipschitz continuous with constant $\Tilde{L}_i$. Then Assumption~\ref{alphaassumtion} holds with 
\begin{equation*}
     \alpha = \frac{\Tilde{L}^{2}}{4\mu^2}M
\end{equation*}
for the SVRG-2, where $\Tilde{L} = \max_i \Tilde{L}_i$ and $\mu = \min_i \mu_i$.
\end{lemma}
\begin{proof}
Note that the SVRG-2 uses the original Hessian, i.e., $\iamtk = \g^2 f_{i_t}(\Tilde{w}_{k-1})$. Then from Lemma~\ref{lips-Hess_lemma}, we have
\begin{equation}
        \|{(\isgd{w_{t-1}} - \isgd{\Tilde{w}} -\g^2 f_{i_t}(\Tilde{w}_{k-1}) (w_{t-1}-\Tilde{w}) )}\| \leq 
        \frac{\Tilde{L}_{i_t}}{2}\|w_{t-1} - \Tilde{w}\|^2.
\end{equation}
Moreover the strong convexity of $f_t$ with modules $\mu_t$ implies
\begin{equation*}
    \mu_{i_t} \|w_{t-1} - \Tilde{w}\| \leq \|\isgd{w_{t-1}} - \isgd{\Tilde{w}}\|.
\end{equation*}
From the above inequalities we have
\begin{align}
        \E[\|(\isgd{w_{t-1}} - \isgd{\Tilde{w}} &  -\g^2 f_{i_t}(\Tilde{w}_{k-1} (w_{t-1}-\Tilde{w}) )\|^2] \nonumber \\
       &  \qquad \leq \left[\frac{\Tilde{L}^2_{i_t}}{4\mu^2_{i_t}}\|w_{t-1} - \Tilde{w}\|^2 \right] \E[\|\isgd{w_{t-1}} 
       - \isgd{\Tilde{w}}\|^2] \\
       & \qquad \leq \left[\frac{\Tilde{L}^2_{i_t}}{4\mu^2_{i_t}}M \right] \E[\|\isgd{w_{t-1}} 
       - \isgd{\Tilde{w}}\|^2] \\
       & \qquad \leq \left[\frac{\Tilde{L}^2}{4\mu^2}M \right] \E[\|\isgd{w_{t-1}} 
       - \isgd{\Tilde{w}}\|^2] 
\end{align}
\end{proof}
\noindent Now we proceed to see Assumption~\ref{alphaassumtion} for the SVRG-2BB and the SVRG-2D.
\begin{lemma}
For SVRG-2BB and SVRG-2D, Assumption~\ref{alphaassumtion} holds with $\alpha = \frac{4L^2}{\mu}$, where $\mu = \min_i \mu_i$.
\end{lemma}
\begin{proof}
First, we give the bound of the approximate Hessian obtained by the SVRG-2BB and the SVRG-2D 
separately.
\begin{itemize}
    
\item[i)] Let $\Tilde{A}^k_{i_t}$ obtained by the BB method, i.e.,
 $$
 \Tilde{A}^{k}_{i_t} = \frac{\svct (\g f_{i_t}(\xsk) - \g f_{i_t}(\Tilde{w}_{m-1}))}{\|\svc\|^2},
 $$
where $s^k = \xsk - \Tilde{w}_{k-1}$, $y^k = \g f_{i_t}(\xsk) - \g f_{i_t}(\Tilde{w}_{k-1})$ and $i_t \in \{1, \ldots,n \}$ chosen randomly. From the 
    Lemma~\ref{syl-Lipschitz_inequality}, we have 
\begin{equation}
    |(\g f_i(w) - \g f_i(y))^T(w - y)| \leq \|\g f_i(w) - \g f_i(y)\| \|w-y\| \leq L \|w - y \|^2 
    \label{2.3}
\end{equation}

Therefore, we have $\|\Tilde{A}^k_{i_t}\|\leq L$.

\item[ii)] Let $\Tilde{A}^k_{i_t}$ be obtained by diagonal Hessian. Since $f_i$ has Lipschitz gradient and it is twice 
differentiable, we have
$$ 
\|\g^2 f_i(w)\|_2 \leq L.
$$
Since $\Tilde{A}_{i_t}$ be a diagonal Hessian of the true Hessian, we have $$\|\Tilde{A}_{i_t}\|^2 \leq \| \nabla^2 f_{i_t}(w)\|^2 \leq 
L.$$
\end{itemize}
Consequently, we have $\|\Tilde{A}_{i_t}\|^2 \leq L$ for both the SVRG-2BB and the SVRG-2D. 

Now we proceed to determine the $\alpha$.
\begin{align*}
        & \|{(\isgd{w_{t-1}} - \isgd{\Tilde{w}} -\iamtk (w_{t-1}-\Tilde{w}) )}\|^2 \\
        & \qquad \qquad \leq 2\|\isgd{w_{t-1}} - \isgd{\Tilde{w}}\|^2 +
        2\|A^{k-1}_{i_t} (w_{t-1} - \Tilde{w})\|^2\\
        & \qquad \qquad \leq 2L^2 \|w_{t-1} - \Tilde{w}\|^2 + 2L^2  \|w_{t-1} - \Tilde{w}\|^2 \\
        & \qquad \qquad \leq 4 L^2 \|w_{t-1} - \Tilde{w}\|^2 \\
        & \qquad \qquad \leq \frac{4L^2}{\mu_{i_t}} \|\isgd{w_{t-1}} - \isgd{\Tilde{w}}\|^2\\
        & \qquad \qquad \leq \frac{4L^2}{\mu} \|\isgd{w_{t-1}} - \isgd{\Tilde{w}}\|^2,
\end{align*}
where first inequality uses $\|x + y\|^2 \leq 2 \|x\|^2 + 2 \|y\|^2 $, and the second inequality follows from the Lipschitz continuity of the $\g f_{i_t}$. Therefore, $\alpha = \frac{4L^2}{\mu}$ for the SVRG-2BB and the SVRG-2D.
\end{proof}
We have proved the linear convergence for SVRG-2, for SVRG-2BB, and for SVRG-2D. 
Finally we prove the linear convergence of SVRG-2BBS. Recall that SVRG-2BBS is option-1 of SVRG-2BB  with the BB-step size $\eta^t_k$. Recall also that Theorem~\ref{thm:fun_linear} and \ref{Th:svrg-2_opt1} hold for Algorithm~\ref{alg:Algorithm_1} with the fixed stepsize $\eta$. We can prove the linear convergence of the SVRG-2BBS similar to these theorems.
\begin{theorem}
Suppose that $\xi_t \in (\xi_0,\xi_1)$ where $0<\xi_0 \leq \xi_1$, and assume that $\eta_0$ and $\eta_1$ satisfy $1 - \eta_1 L(2\alpha +1)>0$ and $(1-2\eta_0 \mu (1 - \eta_1 L(2 \alpha +1))) \in [0,1)$.
Then SVRG-2BBS is linearly convergent in expectation:
\begin{equation}
        \E[\|\Tilde{w}_k - w_*\|^2] \leq \Tilde{\gamma}^k \|\Tilde{w_0} - w_* \|^2,
\end{equation}
where 
    \begin{equation*}
    \Tilde{\gamma} = \left[(1-2 \eta_0 \mu(1 - \eta_1 L (2\alpha + 1)))^m + \frac{2\alpha \eta_1^2 L^2}{\eta_0\mu(1 - \eta_1 L (2\alpha + 1))}\right].
    \end{equation*}
\end{theorem}
\begin{proof}
Let $\eta^t_k$ be a BB-step size in SVRG-2BBS, that is
\begin{equation*}
    \eta^t_k = \vartheta_k\ \xi_t=\frac{\xi_t}{m_1}\frac{\|\Tilde{w}_{k-1} - \Tilde{w}_{k-2}\|^2}{((\Tilde{w}_{k-1} - \Tilde{w}_{k-2})^T (\Tilde{g}_{k-1} - \Tilde{g}_{k-2}))},
\end{equation*}
where $g_k = \nabla F(\Tilde{w}_k)$, and $\xi$ is small positive constant. Note that $\eta^t_k$ varies at each iteration.

We first obtain the lower and upper bounds of $\eta^t_k$ using the $L$-Lipschitz continuity of $\nabla F(w)$ and the strong convexity of $F(w)$.
\begin{equation*}
\mu \|\Tilde{w}_{k-1} - \Tilde{w}_{k-2}\|^2 \leq (g_{k-1} - g_{k-2})^T(\Tilde{w}_{k-1} - \Tilde{w}_{k-2}) \leq L \|\Tilde{w}_{k-1} - \Tilde{w}_{k-2}\|^2.
\end{equation*}
From the above inequalities, it is easy to see that $\eta^t_k$ has lower and upper bound as follows:
\begin{equation*}
    \frac{\xi_0}{m_1L} \leq \eta^t_k \leq \frac{\xi_1}{m_1 \mu}.
\end{equation*}
Let $\eta_0 = \frac{\xi_0}{m_1L}$ and $\eta_1 = \frac{\xi_1}{m_1\mu}$. In a way similar to show~\eqref{linear_x},  we can show that 
\begin{align*}
    \E[\|w_{t} -   w_*\|^2] & \leq \|w_{t-1} - w_*\|^2 - 2 \eta_0 (1 - \eta_1 L (2\alpha + 1))~[F(w_{t-1}) - F(w_*)] \nonumber\\
   & \qquad \qquad \qquad \qquad \qquad \qquad \qquad \qquad \qquad + 4 \alpha \eta_1^2 L ~[F(\Tilde{w}) - F(w_*)].
\end{align*}

In a way similar to the arguments in the proof of Theorem~\ref{Th:svrg-2_opt1}, summing the inequality over $t=0,\ldots, m-1$ and taking the expectation of the history, yield
\begin{align}
 \E[\|\Tilde{w}_{k} -   w_*\|^2] & \leq [1-2 \eta_0 \mu(1 - \eta_1 L (2\alpha + 1))]^m\|\Tilde{w}_{k-1} - w_*\|^2\nonumber\\
 &  \quad \qquad \qquad + 4 \alpha \eta_1^2 L^2 ~
 \sum_{j=1}^{m}[1-2 \eta_0 \mu(1 - \eta_1 L (2\alpha + 1))]^j
 \|\Tilde{w}_{k-1} - w_*\|^2 \label{eq:eta_1_0}\\
 & = \left[(1-2 \eta_0 \mu(1 - \eta_1 L (2\alpha + 1)))^m + \frac{2\alpha \eta_1^2 L^2}{\eta_0\mu(1 - \eta_1 L (2\alpha + 1))}\right] \|\Tilde{w}_{k-1} - w_*\|^2 \nonumber\\
 & = \Tilde{\gamma}  \|\Tilde{w}_{k-1} - w_*\|^2 \nonumber
\end{align}

Hence, we have the linear convergence, that is
\begin{equation*}
    \E[\|\Tilde{w}_k - w_*\|^2] <  \Tilde{\gamma}\  \|\Tilde{w}_{k-1} - w_*\|^2.
\end{equation*}



\end{proof}
\section{Numerical results}
 
In this section we present the numerical results for the proposed algorithms.
 
First we discuss the experiment setup for the numerical experiments. We performed all experiments on MATLAB R2018a on Intel(R) Xeon(R) CPU E7-8890 v4 @ 2.20GHz with 96 cores and we used SGDLibrary~\cite{kasai2017sgdlibrary} to perform the existing algorithms. We solve on standard learning problems, that is, the $\ell_2$-logistic regression:
\begin{equation}
    \min_w F(w) = \frac{1}{n} \sum_{i=1}^{n}\log\left[1+\exp(-b_ia_i^Tw) \right] + \frac{\lambda}{2} \|w\|^2,
\end{equation}
and the $\ell_2$-squared SVM:
 \begin{equation}
     \min_w F(w) = \frac{1}{2n} \sum_{i=1}^{n}\left( \left[1 - b_ia_i^T w\right]_+\right)^2 + \frac{\lambda}{2} \|w\|^2,
 \end{equation}
where $a_i \in \R^d$ is feature vector and $b_i \in \{\pm 1 \}$ is target label of the $i$-th sample, and $\lambda$ is a $\ell_2$ regularizer. We performed numerical experiments on benchmark datasets given in Table~\ref{tab:datasets}. The datasets are binary classification problems available on LIBSVM~\cite{LIBSVM_CC01a}. 
\begin{table}[!h]
    \small
    \centering
    \caption{Details of the datasets used in the experiments}
    \begin{tabular}{|l|r|r|r|}
        \hline
        \textbf{Dataset} & \textbf{Dimension} & \textbf{Training} & \textbf{Model}\\ 
        \hline
        \textit{ijcnn1}     & 22 & 49990 & SVM \\
        \textit{mnist-38}    & 784 & 11,982 & SVM\\ 
        \textit{adult}  & 123 & 32,561 & LR \\
        \textit{gisette}    & 5000 & 6000 & LR \\
        \textit{covtype}     & 54 & 464808 & LR \\
        \hline
    \end{tabular}
    \label{tab:datasets}
\end{table}\\
\subsection{Comparison of $\xi_t$ in SVRG-2BBS}

We first conduct the numerical experiments on the three different possibility of $\xi_t$ for the generalized BB step size in SVRG-2BBS. Recall that the step size $\eta^t_k$ in SVRG-2BBS is given as, 
\begin{equation}
    \eta^t_k = \frac{\xi_t}{m_1}\ \frac{\|\Tilde{w}_{k-1} - \Tilde{w}_{k-2}\|^2}{((\Tilde{w}_{k-1} - \Tilde{w}_{k-2})^T (\Tilde{g}_{k-1} - \Tilde{g}_{k-2}))},
\end{equation}
and $\xi_t$ is computed differently as given in Table~\ref{tab:sub_methods}, where $\xi_t$ is of the form of decay $\xi_t = \frac{c_1}{1+c_2T}$, $T$ is
current iterate $T = km + t, c_1$ is initial stepsize and $c_2= \eta_0 \lambda$ and in the case of fix $\xi_t$, we use grid search with $\xi_t =  c_1$. 
\begin{table}[!h]
    \small
    \centering
    \caption{generalized variants of SVRG-2BBS} 
    \begin{tabular}{|l|r|r|r|}
        \hline
        \textbf{Methods} & \textbf{$ m_1 $} & \textbf{$ \xi_t $ type} \\ 
        \hline
        \textbf{\quad M1}      & $2n$ & fix $\xi_1$ \\
        \textbf{\quad M2}      & $n$ & decay $\xi_t$ \\
        \textbf{\quad M3}      & $1$ & decay $\xi_t$ \\
        \hline
    \end{tabular}
    \label{tab:sub_methods}
\end{table}
Figure~\ref{fig:M1to4} shows the comparison of M1, M2, and M3 on the \textit{adult, covtype,} and \textit{gisette} datasets for the \textit{$\ell_2$-regularized logistic regression} with $\lambda = 10^{-3}$. We tuned step size for all methods via grid
 search $c_1 \in \{10^1,10^0,10^{-1},10^{-2},10^{-3},10^{-4},10^{-5}\}$ and we set $m = 2n$ (where $n$ is sample size) as in SVRG~\cite{SVRG_NIPS2013}. From Figure~\ref{fig:M1to4} we see that M1 and and M3 performs better among all three methods. Therefore, in the further experiments we use M1 and M3 for comparison with SVRG-2BB and the other existing methods.

\begin{figure}[!h]
    \centering
    \subfigure{\includegraphics[width = 3cm, height = 0.5cm]{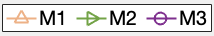}}\\
    \subfigure[\textit{adult}]{\includegraphics[width = 4.7cm, height = 3.4cm]{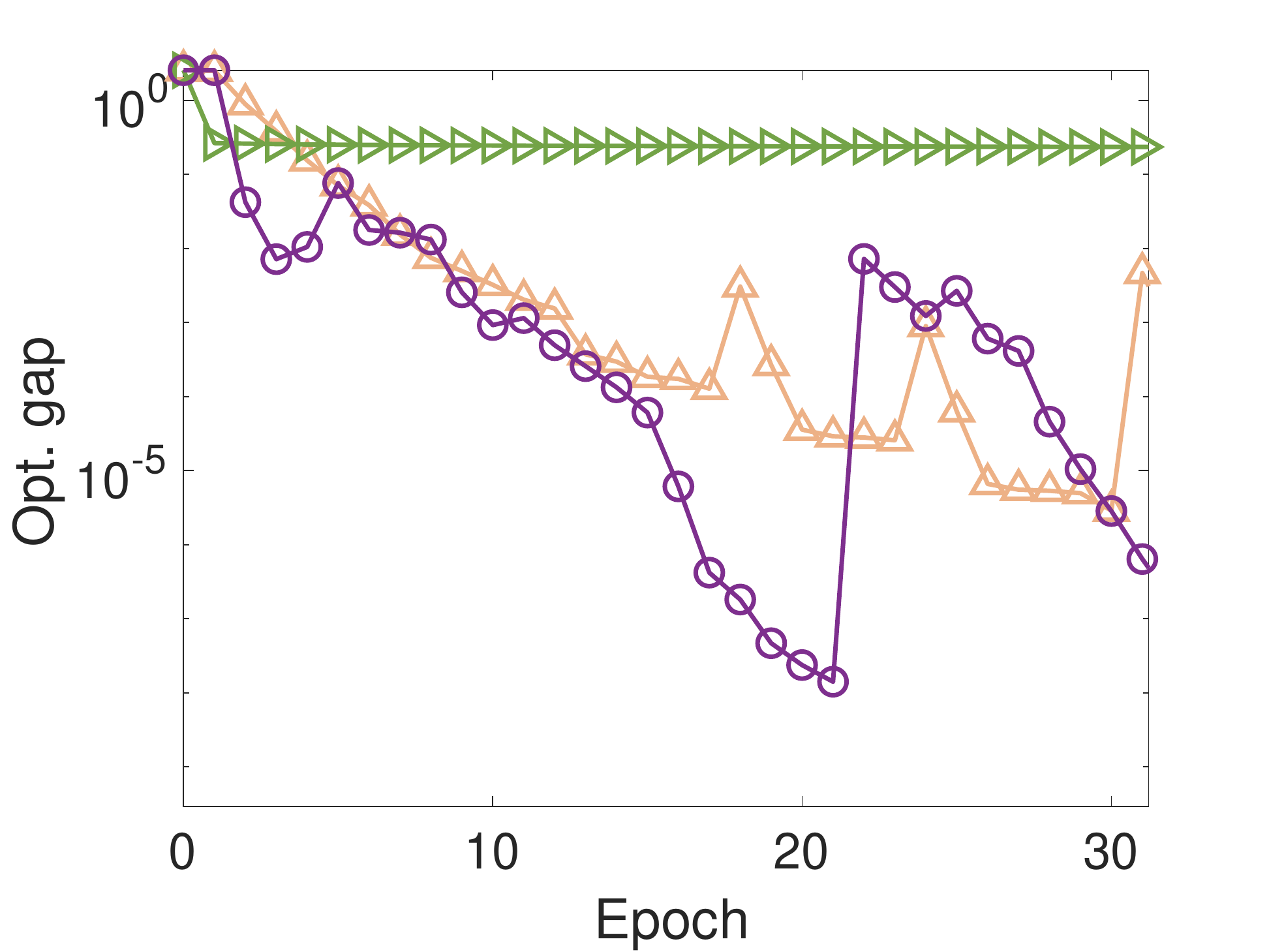}}
    \subfigure[\textit{covtype}]{\includegraphics[width = 4.7cm, height = 3.4cm]{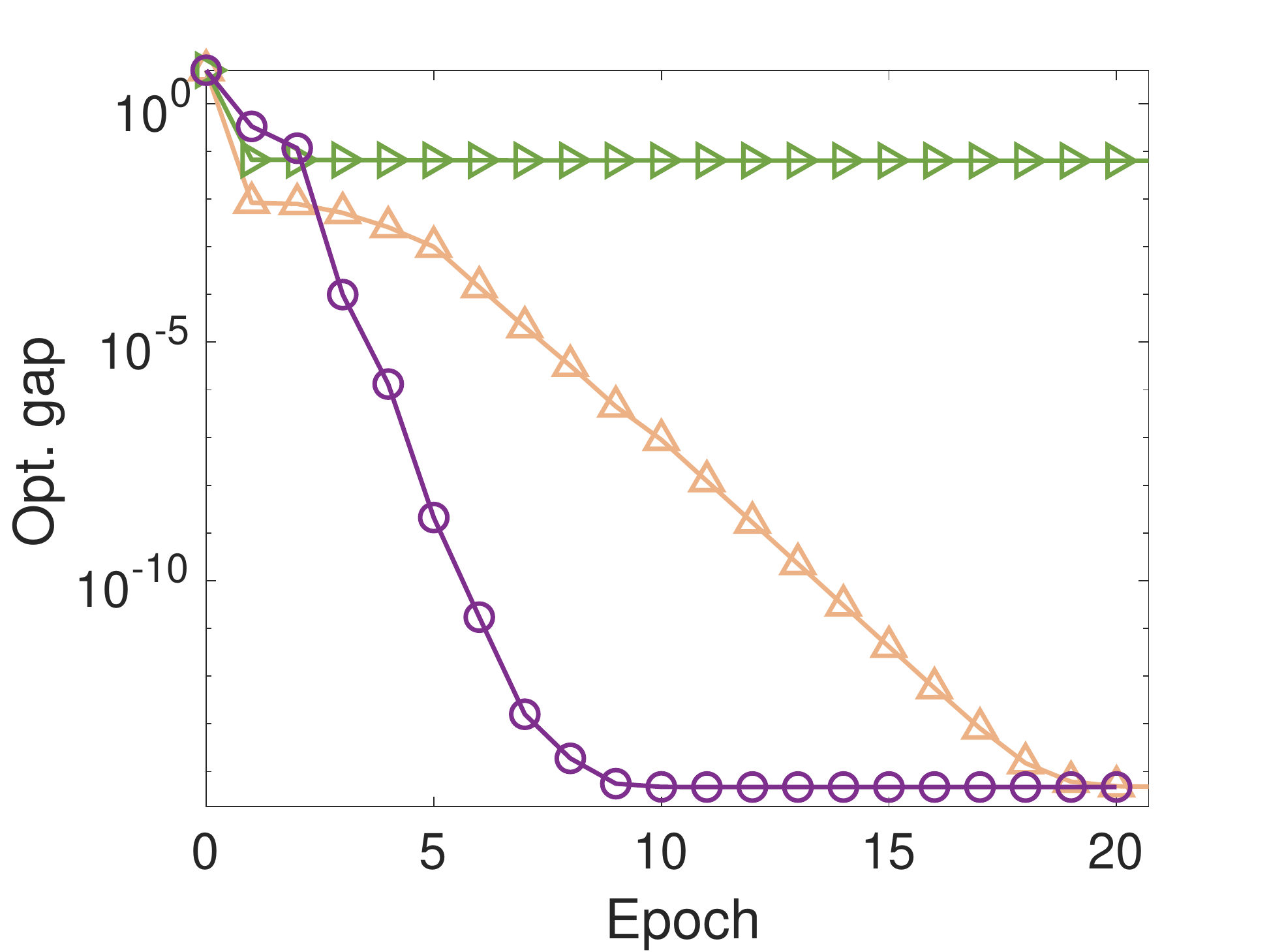}}
    \subfigure[\textit{gisette}]{\includegraphics[width = 4.7cm, height = 3.4cm]{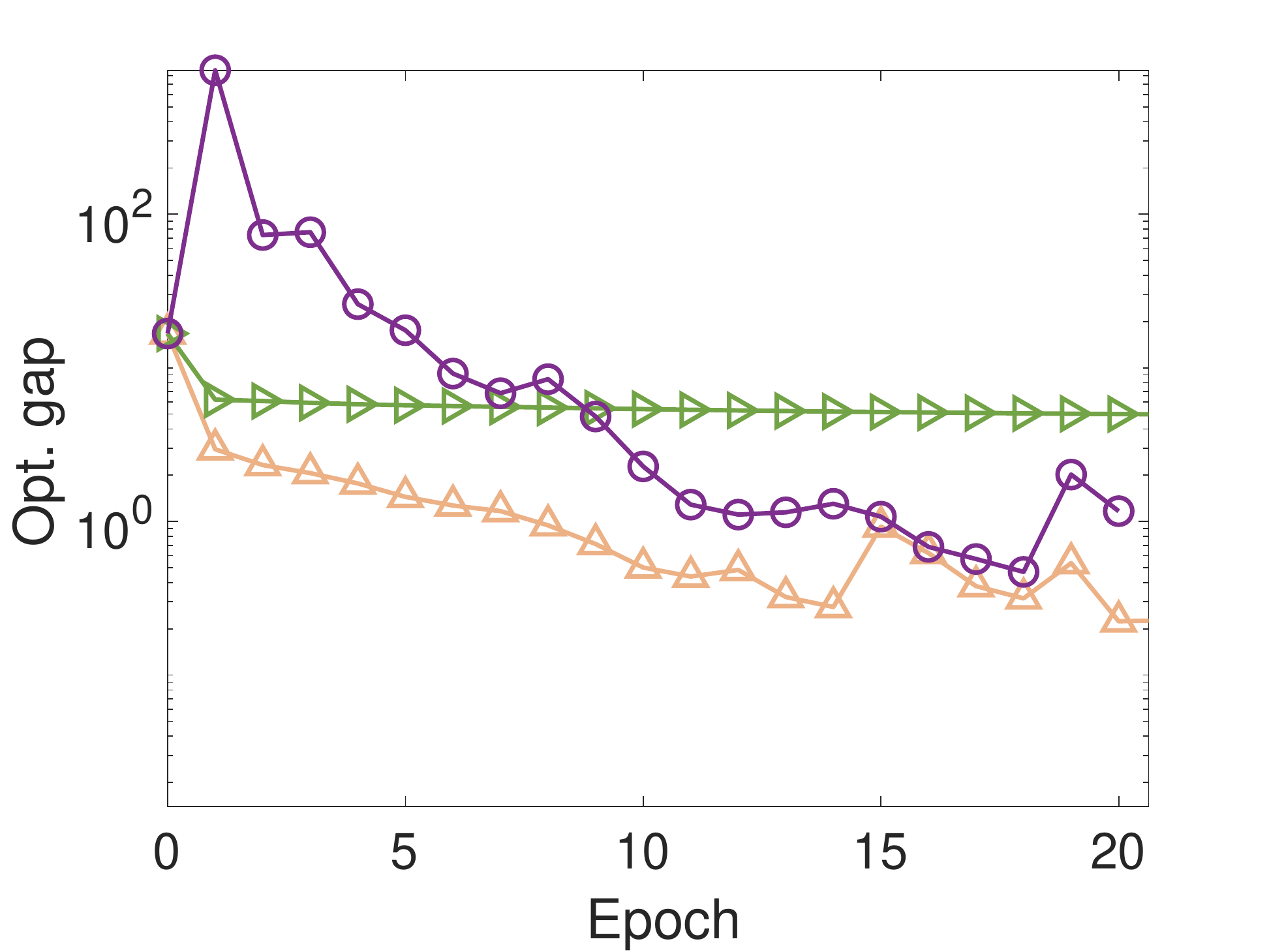}}
    \caption{comparison of M1, M2, and M3 with $\lambda = 10^{-3}$ on \textit{adult, covtype,} and \textit{gisette}.}
    \label{fig:M1to4}        
\end{figure}
\subsection{Comparison of SVRG, SVRG-BB, SVRG-2D, and the proposed methods}

Next, we compare SVRG, SVRG-BB, SVRG-2D with the proposed methods SVRG-2BB and SVRG-2BBS (M1 and M3).

We explain the experimental setup. We tuned the stepsize $\eta$ for all methods via grid
search $\eta \in \{10^0,10^{-1},10^{-2},10^{-3},10^{-4},10^{-5}\}$ for SVRG, SVRG-2D, SVRG-2BB, SVRG-BB, and $\xi_T \in \{10^0,10^{-1},10^{-2},10^{-3},10^{-4},10^{-5}\}$ for SVRG-2BBS (M1 and M3). It is essential to note that SVRG-BB highly depends on the initial step size $\eta_0$ as it can be seen in the numerical results of ~\cite{SVRGBB_NIPS2016}. This implies, that even though we are using adaptive step size, we still need to hunt for the best initial (parameter) step size. Therefore, we are not benefiting from the SVRG-BB in terms of reducing the hunt of the number of initial parameters which does not outperform the SVRG-2BB, SVRG, and SVRG-2D. 

For all methods, we set $m = 2n$. We present the optimiality gap with epoch and CPU time. Also, we present the variance information with epoch for each datasets. The variance is calculated as $\| v_k - \nabla f(w_k)\|^2$, where $v_k$ is search direction at $t=0$ of the $k$-th epoch.

We present numerical results in Figures~\ref{fig:covtype}-\ref{fig:ijcnn1}. Each column of the figure represents the results for each dataset. In the figures of the first row of each column, the $x$-axis denotes epochs and the $y$-axis denotes the optimality gap $F(\Tilde{w}_k) - F(w_*)$, where $w_*$ is obtained by limited memory BFGS method. In the figures of the second rows of each column, the $x$-axis denotes CPU time and the $y$-axis denotes the optimality gap. In the figures of the third rows of each column, the $x$-axis denotes epochs and the $y$-axis denotes the variance $\| v_k - \nabla f(w_k)\|^2$.
\begin{figure}[!h]
    \centering
    {\includegraphics[scale=0.35]{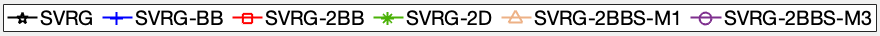}} \\
    {\includegraphics[width = 4.3cm, height = 3.3cm]{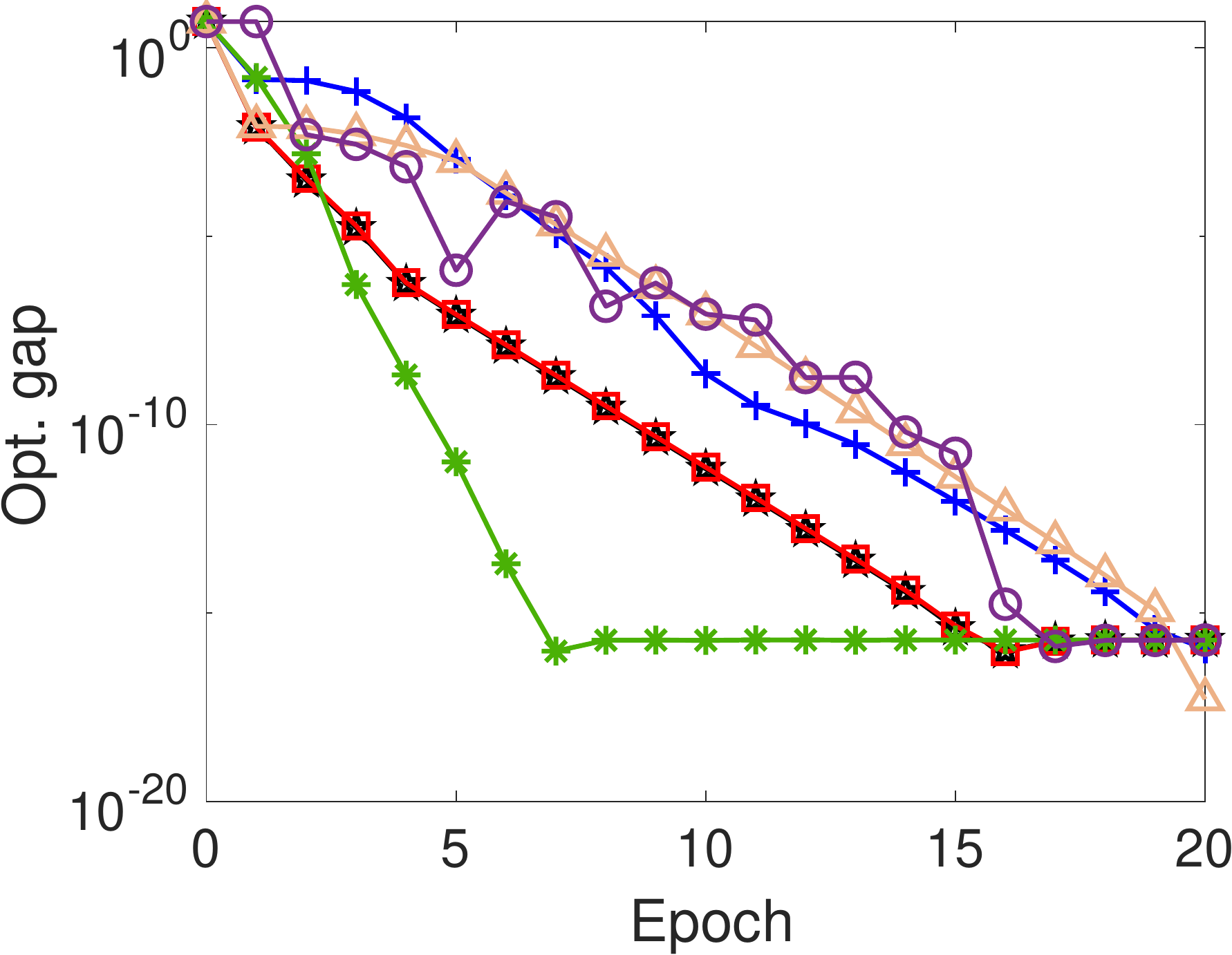}}
    {\includegraphics[width = 4.3cm, height = 3.3cm]{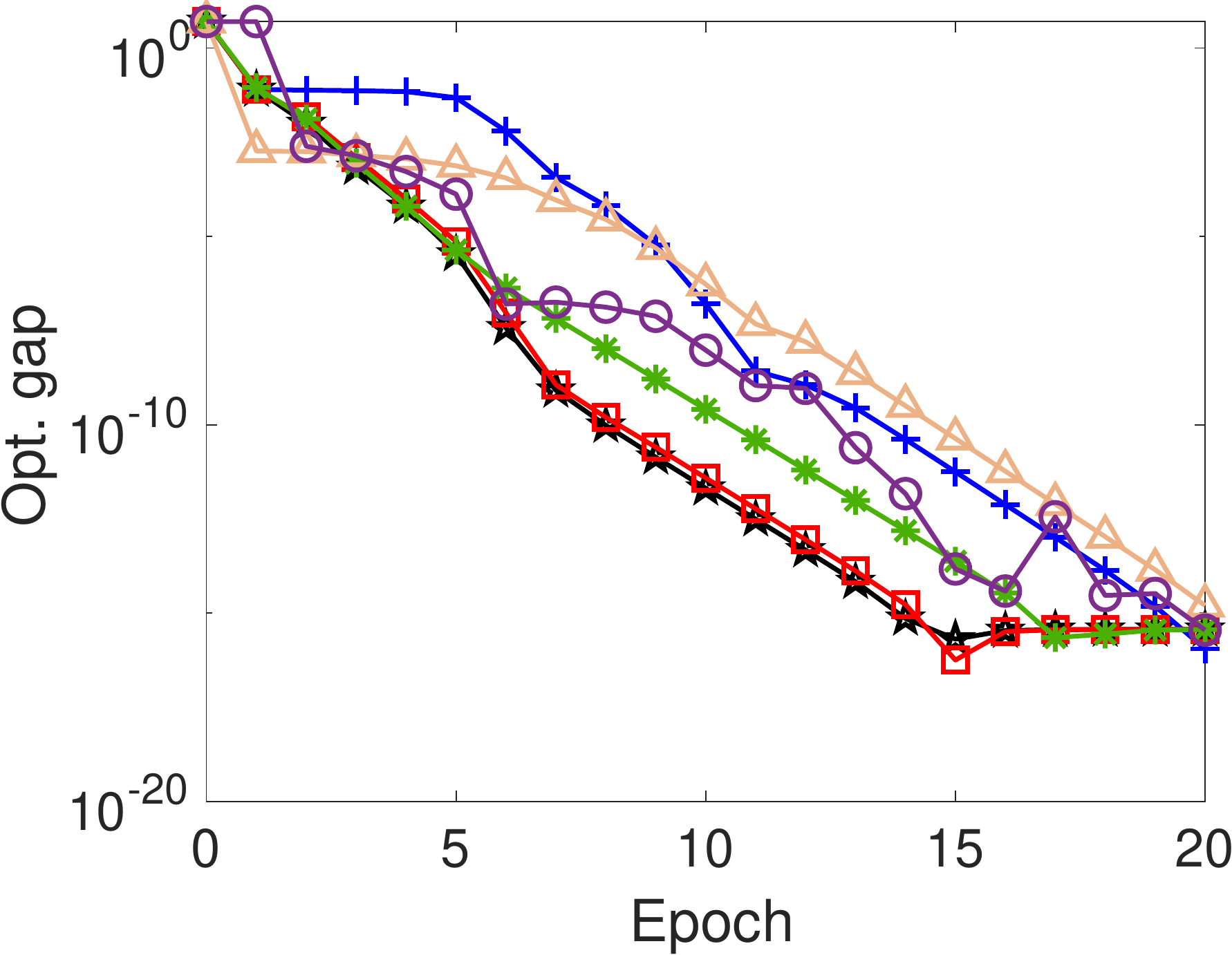}}
    {\includegraphics[width = 4.3cm, height = 3.3cm]{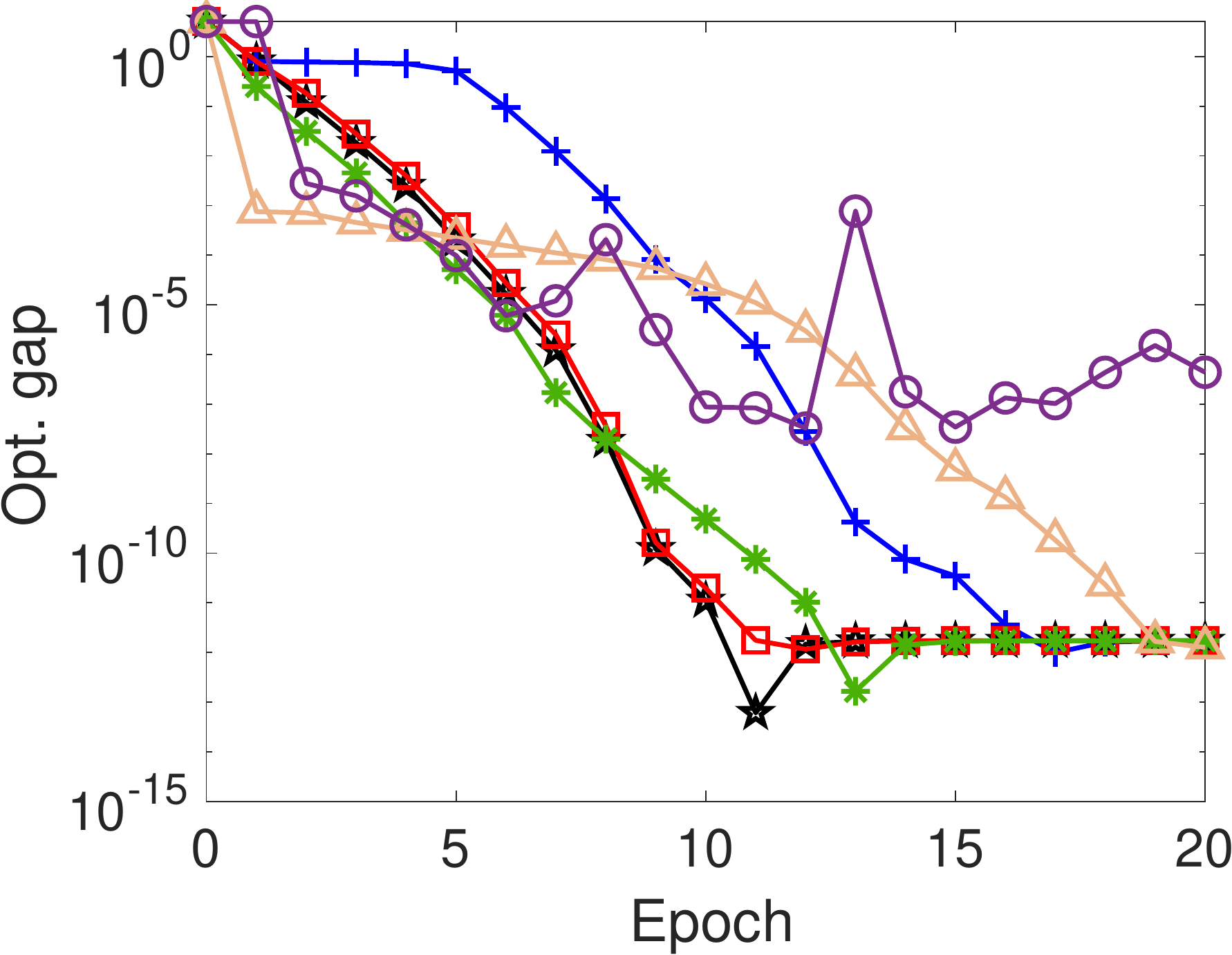}}\\
    {\includegraphics[width = 4.3cm, height = 3.3cm]{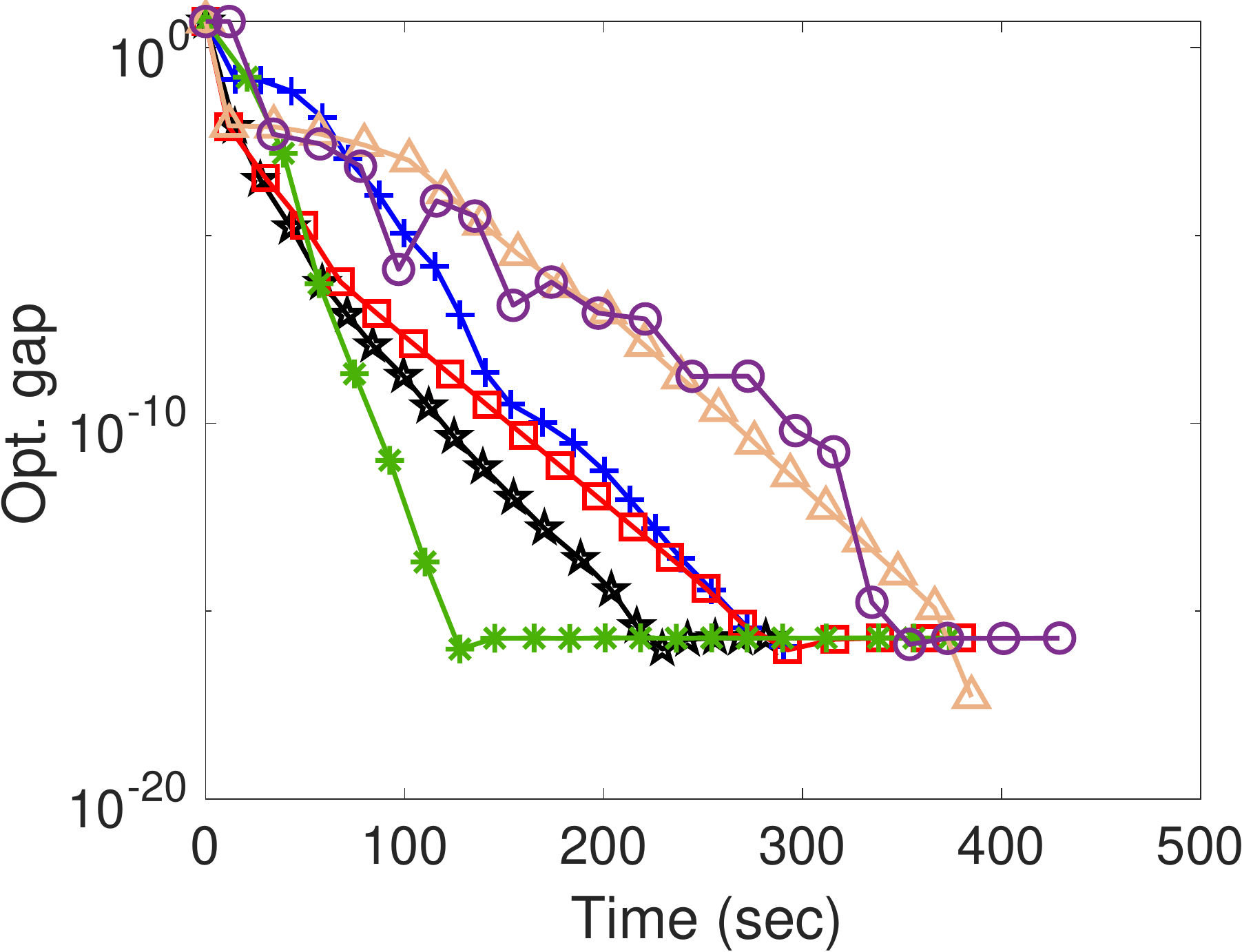}}
    {\includegraphics[width = 4.3cm, height = 3.3cm]{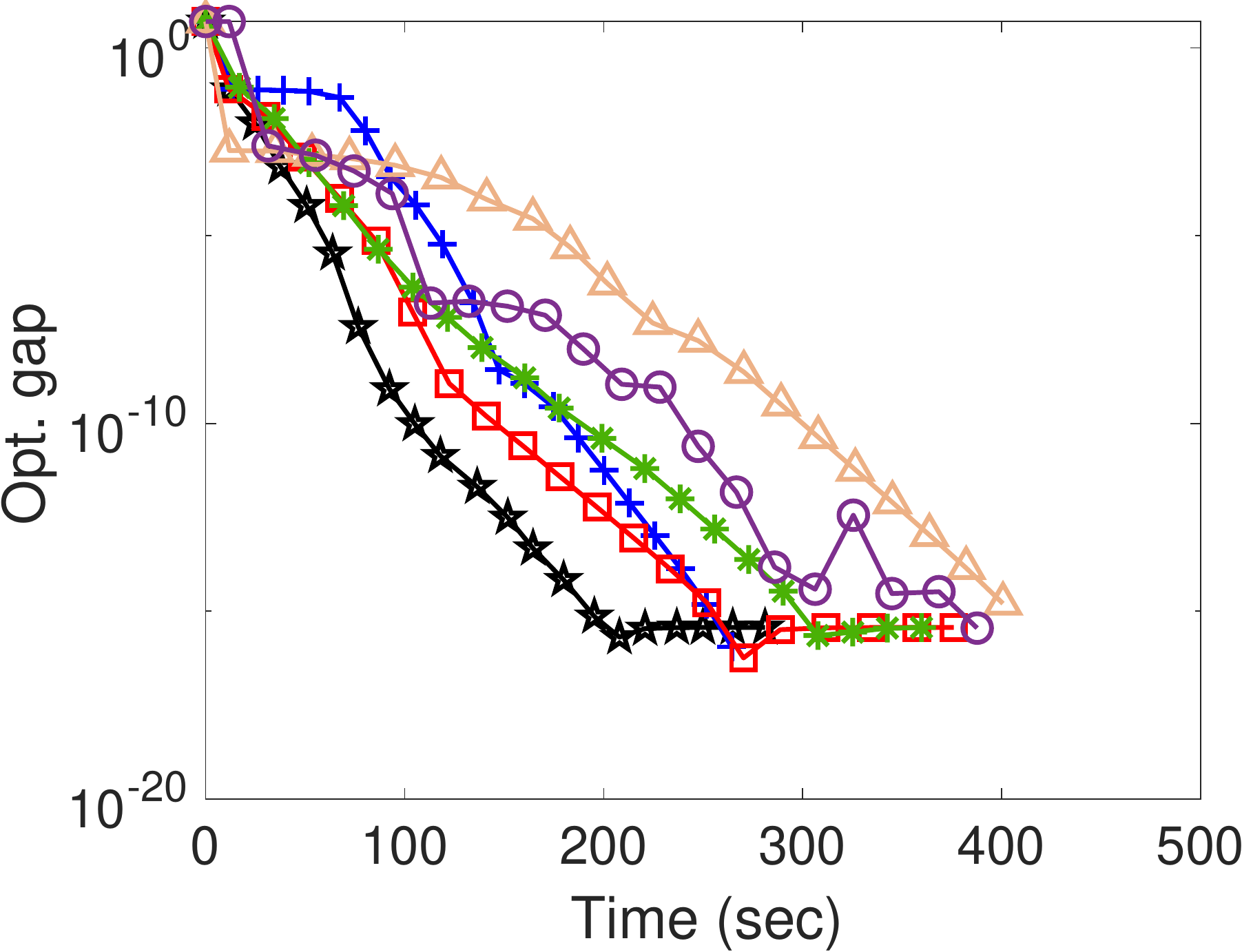}}
    {\includegraphics[width = 4.3cm, height = 3.3cm]{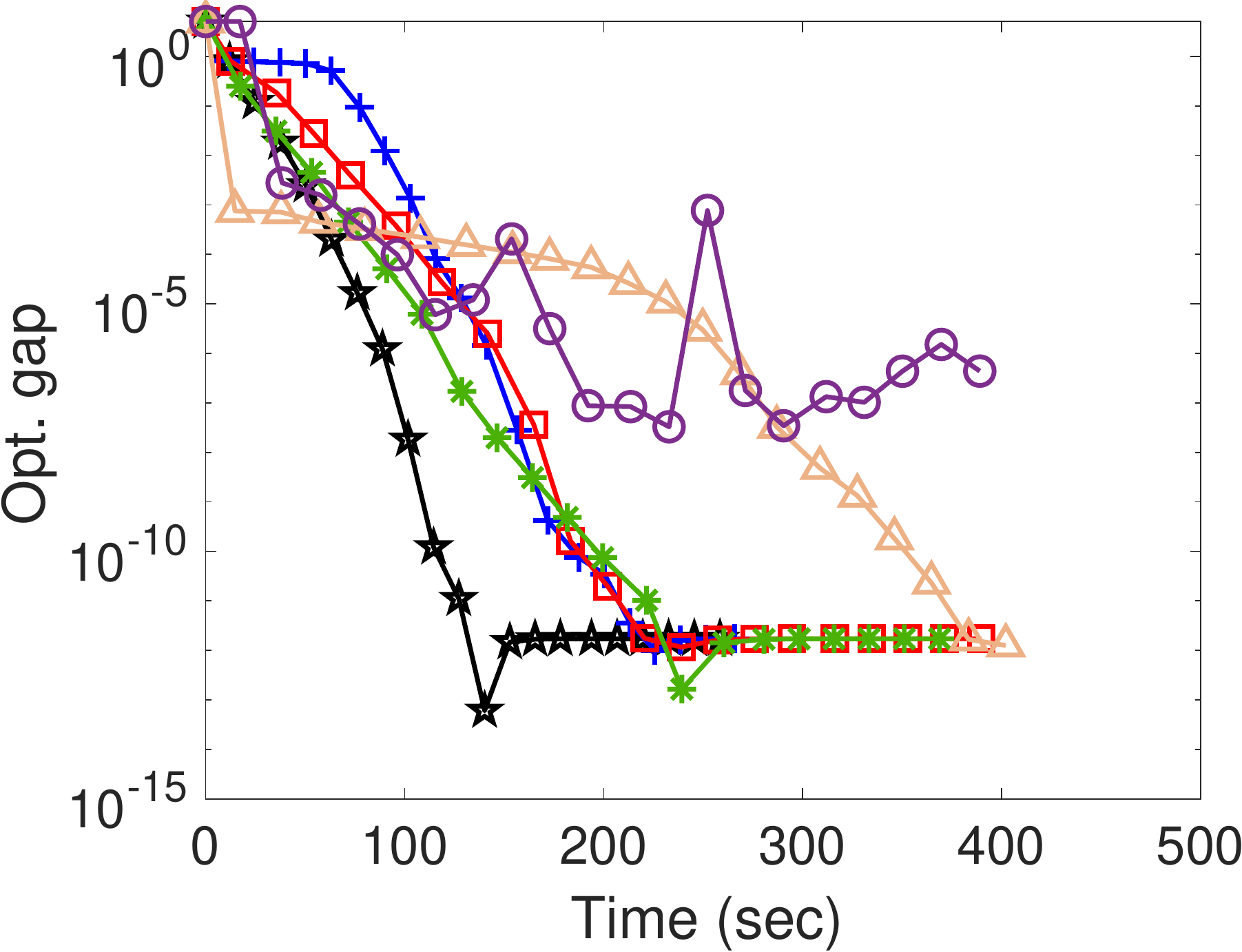}}\\
    \subfigure[\textit{covtype} ($\lambda=10^{-3}$)]{\includegraphics[width = 4.3cm, height = 3.3cm]{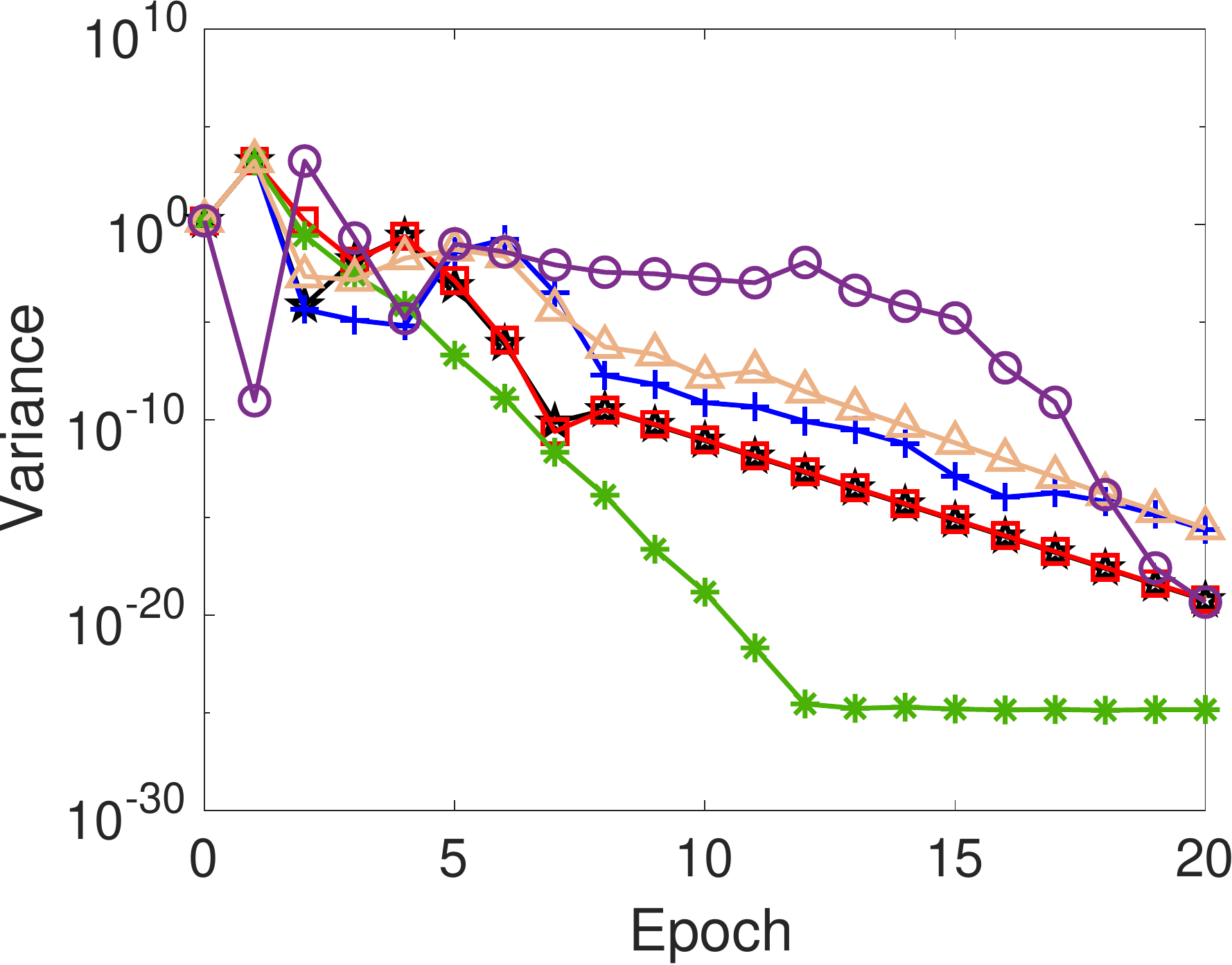}}
    \subfigure[\textit{covtype} ($\lambda=10^{-4}$)]{\includegraphics[width = 4.3cm, height = 3.3cm]{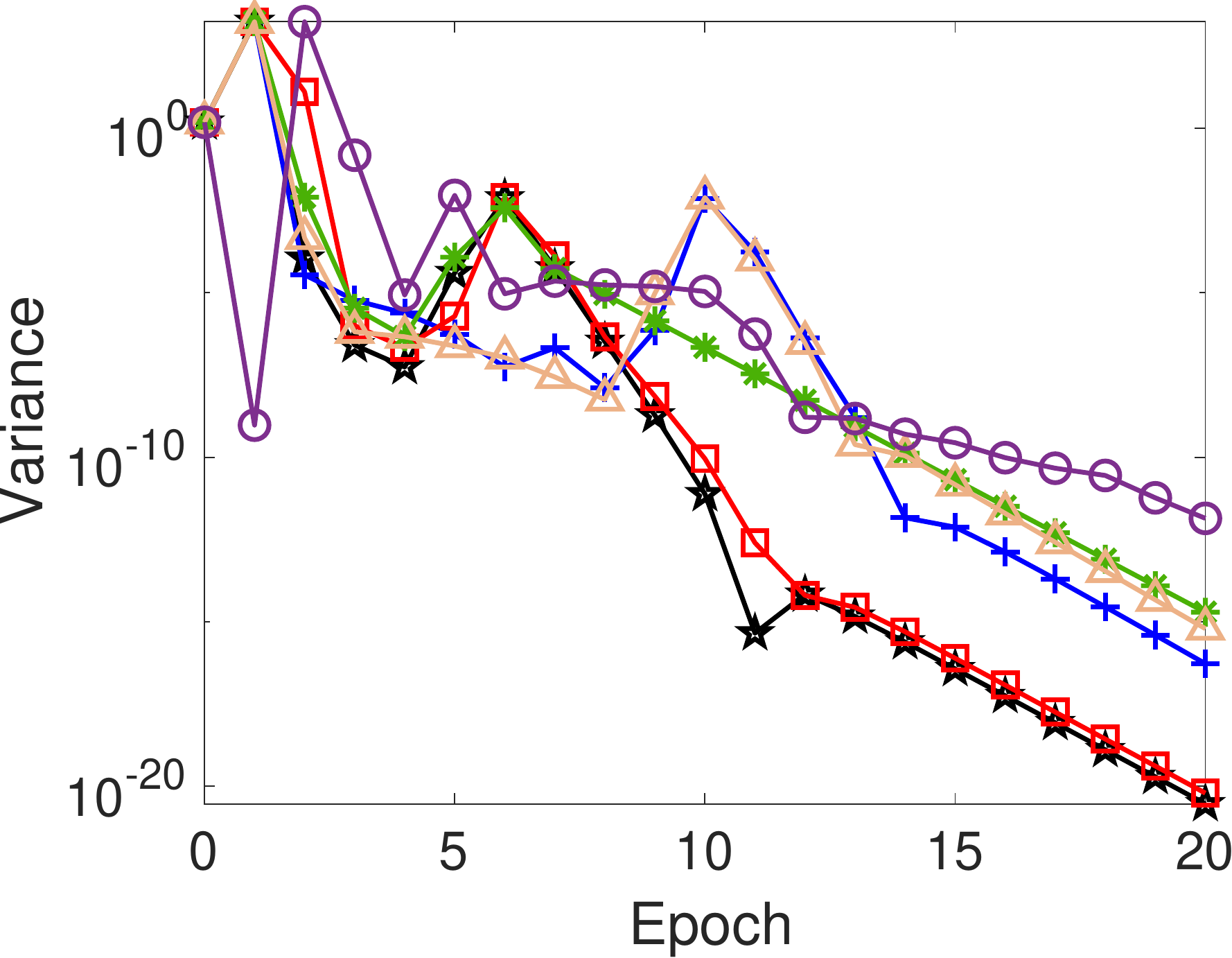}}
    \subfigure[\textit{covtype} ($\lambda=10^{-5}$)]{\includegraphics[width = 4.3cm, height = 3.3cm]{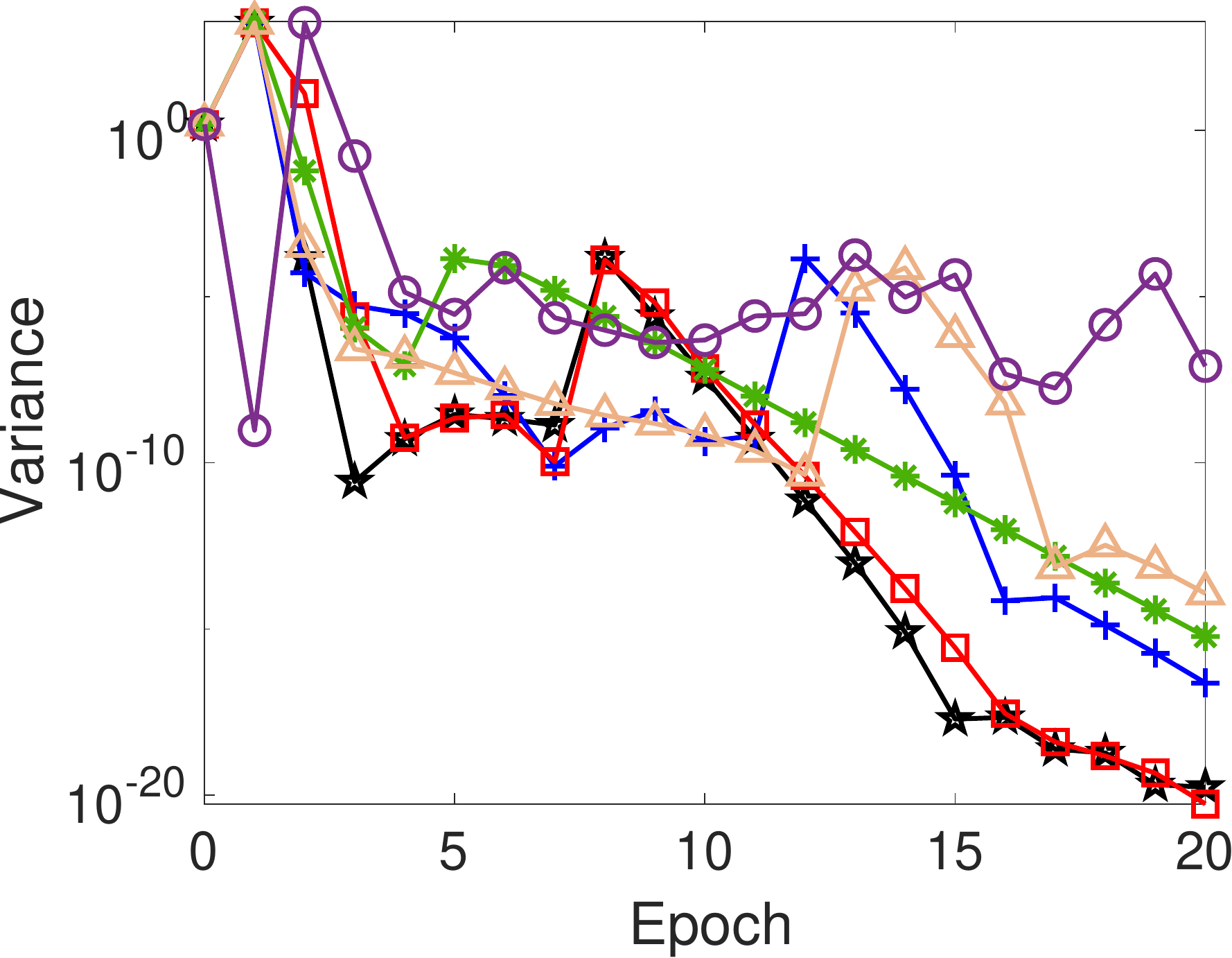}}
    \caption{Comparison on \textit{covtype} dataset}
    \label{fig:covtype}
\end{figure}

We conduct numerical experiments on \textit{covtype, gisette} and \textit{adult} for the \textit{$\ell_2$-regularized logistic regression} with $\lambda = \{10^{-3}, 10^{-4}, 10^{-5}\}$ in Figure~\ref{fig:covtype}, Figure~\ref{fig:gisette} and Figure~\ref{fig:adult}, respectively. Figure~\ref{fig:mnist} shows the numerical experiments on \textit{mnist38} for the \textit{$\ell_2$-regularized squared SVM} with $\lambda = \{10^{-3}, 10^{-4}, 10^{-5}\}$, where we used binary class using those samples whose target class were 3 and 8. Figure~\ref{fig:ijcnn1} shows the numerical experiments on \textit{ijcnn1} for the \textit{$\ell_2$-regularized squared SVM} with $\lambda = \{10^{-3}, 10^{-4}\}$.

The figures for the \textit{$\ell_2$-regularized logistic regression} shows that SVRG with the second order information outperforms the pure SVRG. SVRG-2BB outperforms all methods when $d$ is large (\textit{i.e.,} where ratio of $n/d$ is close to 1) and $\lambda$ is small. SVRG-2D performs better only on small $\lambda = 10^{-3}$ for \textit{covtype} dataset where $n$ is large and $d$ is very small. However, on the remaining datasets, SVRG-2D performs similar to SVRG. As the variance of SVRG-BB and SVRG-2BBS is fluctuating, hence the BB-stepsize can be quite unstable.

The fugues for the \textit{$\ell_2$-regularized squares SVM} show that SVRG unables to outperform other methods. Figure~\ref{fig:mnist} shows SVRG-2D performs well in the initial epochs, however SVRG-2BB, SVRG-2BBS (M1 and M3) and SVRG-BB performs similar on the later epochs. Figure~\ref{fig:ijcnn1} shows that SVRG-2BB, SVRG-2BBS (M1 and M3), and SVRG-BB outperform SVRG and SVRG-2D for $\lambda = 10^{-3}.$
\begin{figure}[!h]
    \centering
    {\includegraphics[scale=0.35]{Legend_Main_aug22.png}} \\
    {\includegraphics[width = 4.3cm, height = 3.3cm]{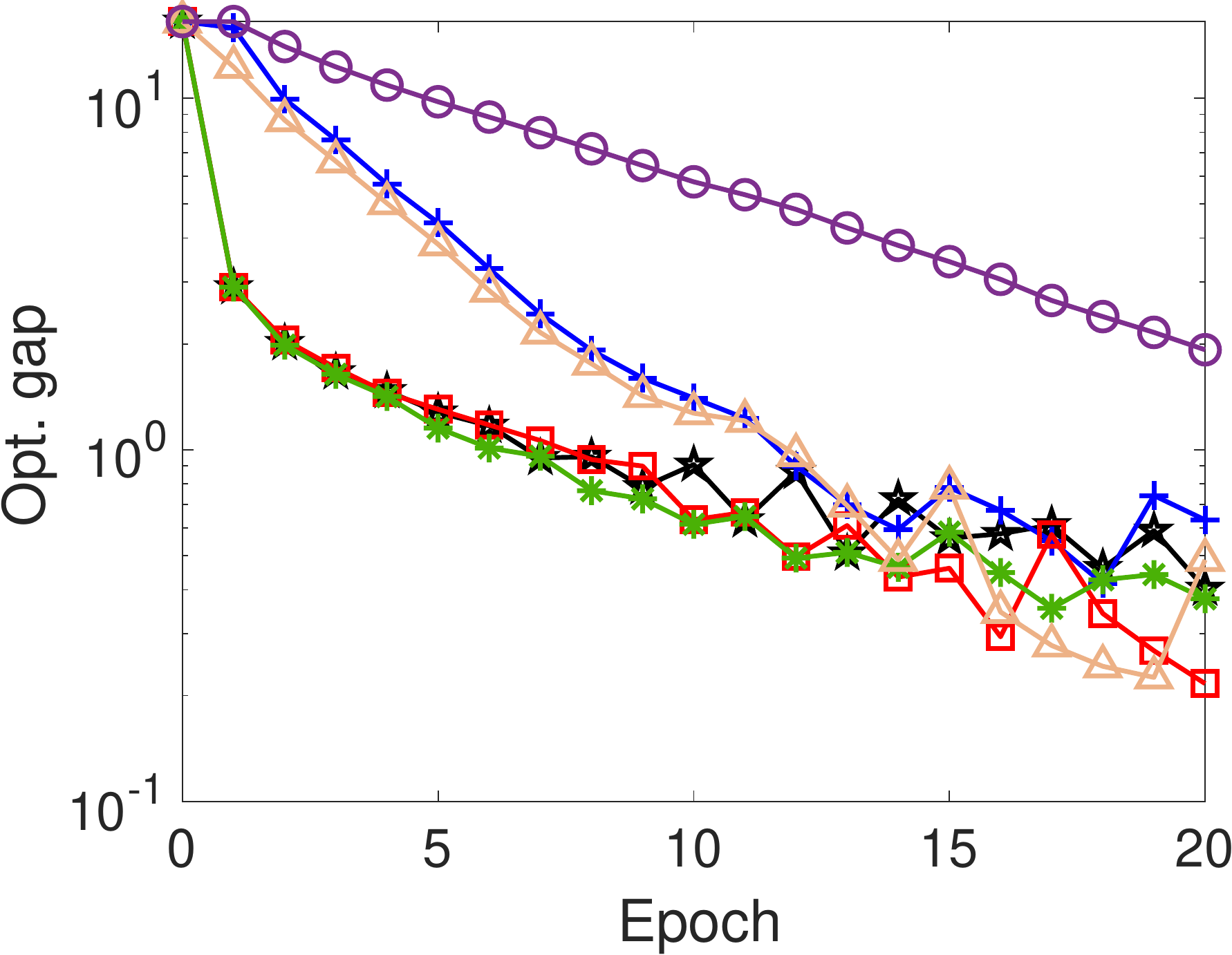}}
    {\includegraphics[width = 4.3cm, height = 3.3cm]{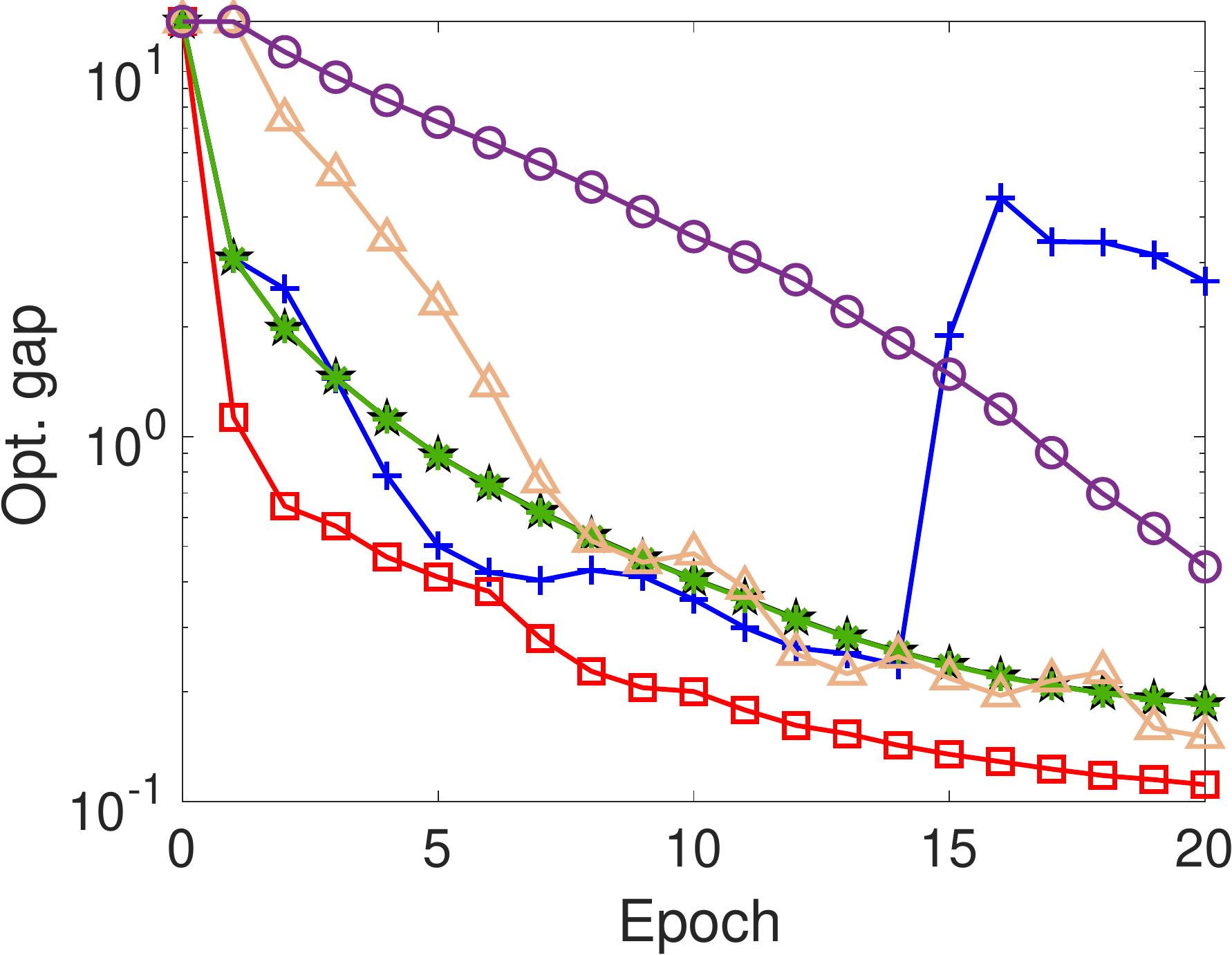}}
    {\includegraphics[width = 4.3cm, height = 3.3cm]{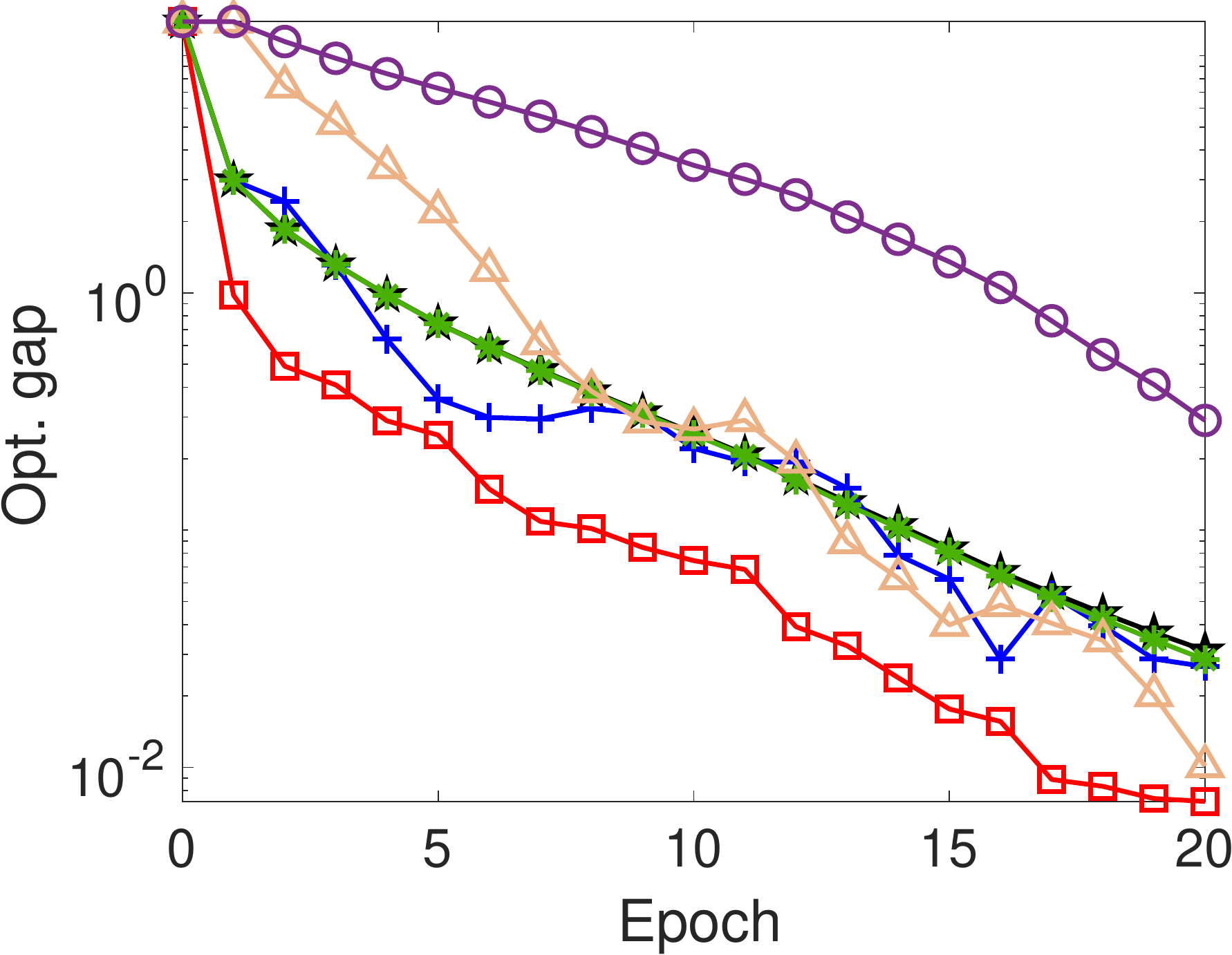}}\\
    {\includegraphics[width = 4.3cm, height = 3.3cm]{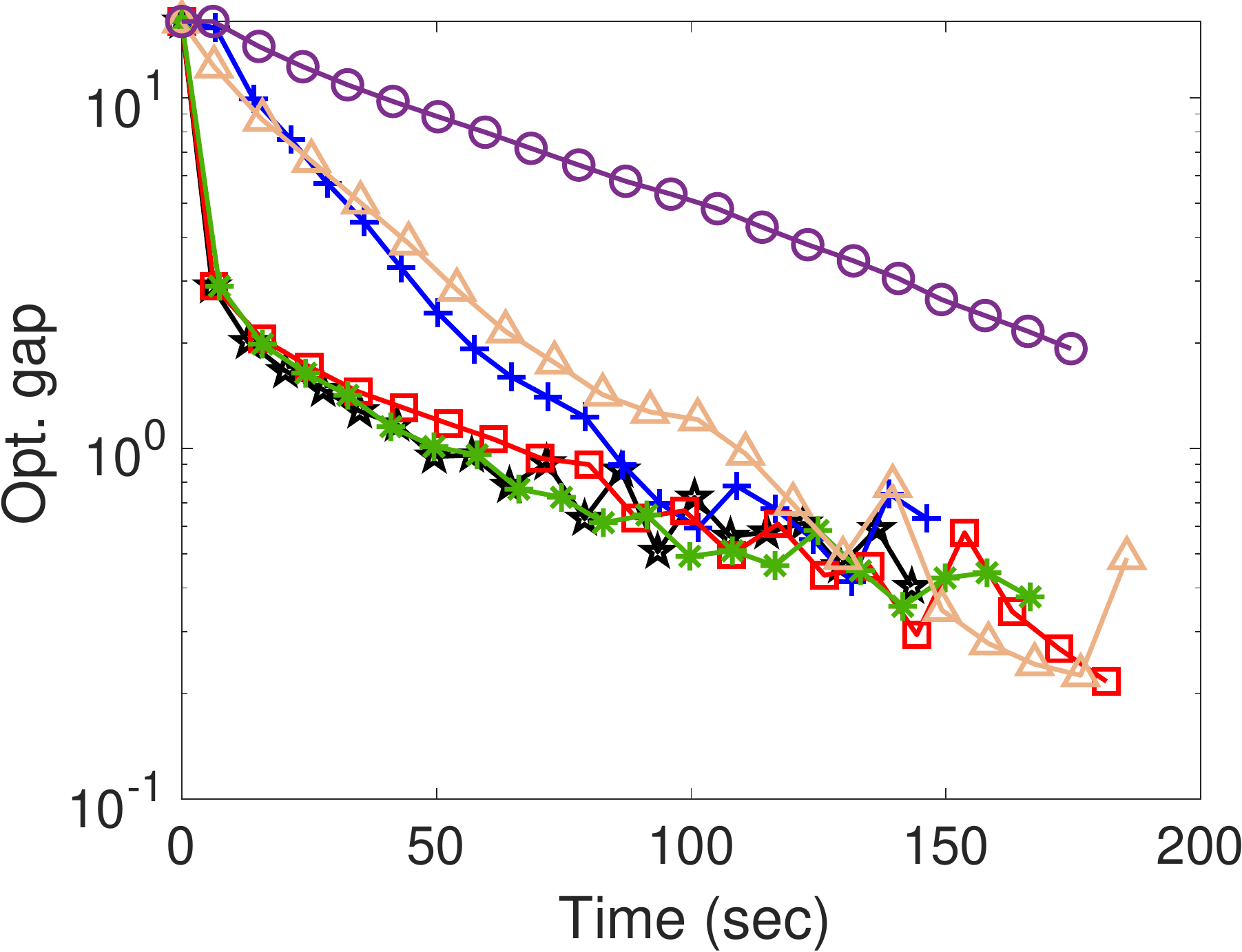}}
    {\includegraphics[width = 4.3cm, height = 3.3cm]{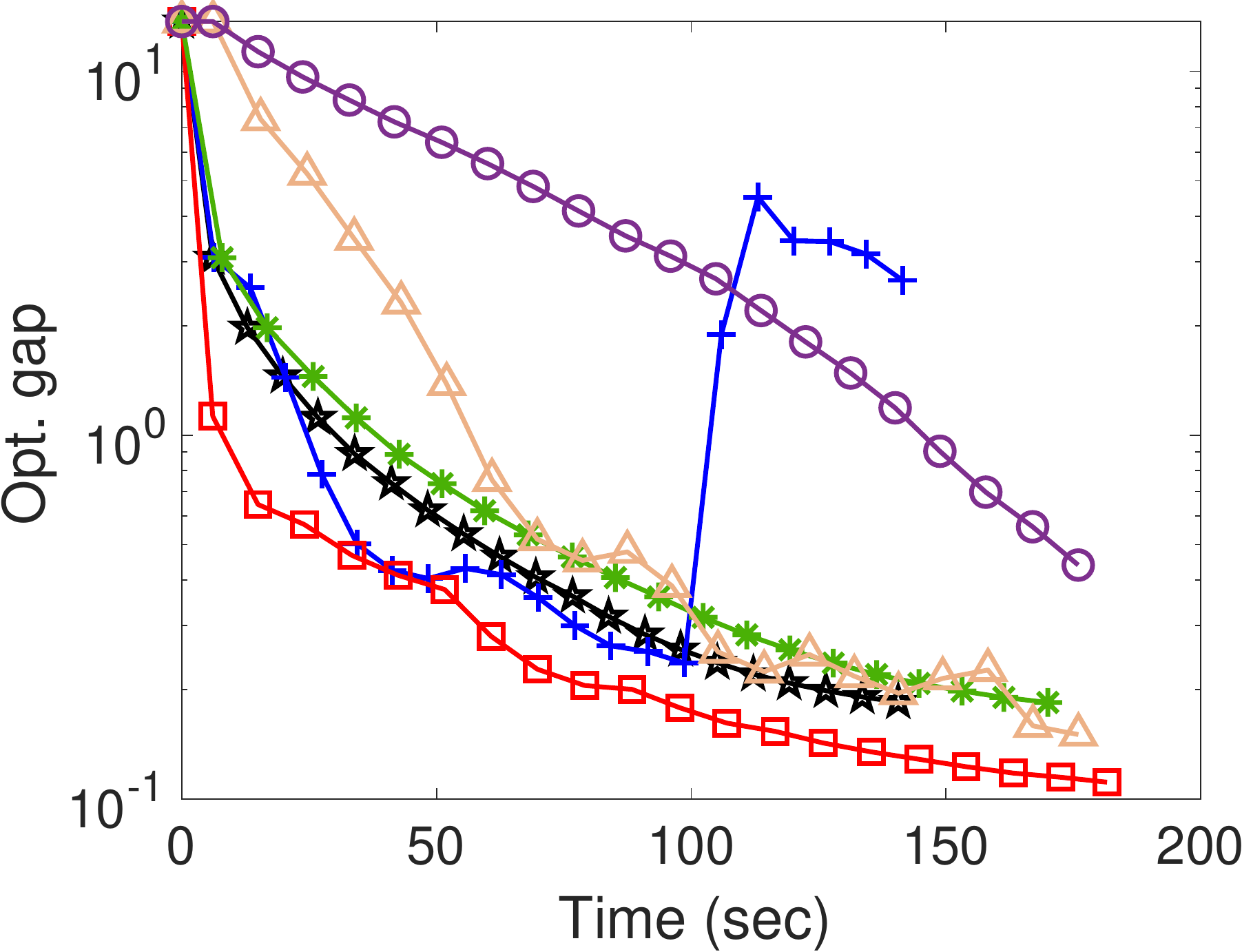}}
    {\includegraphics[width = 4.3cm, height = 3.3cm]{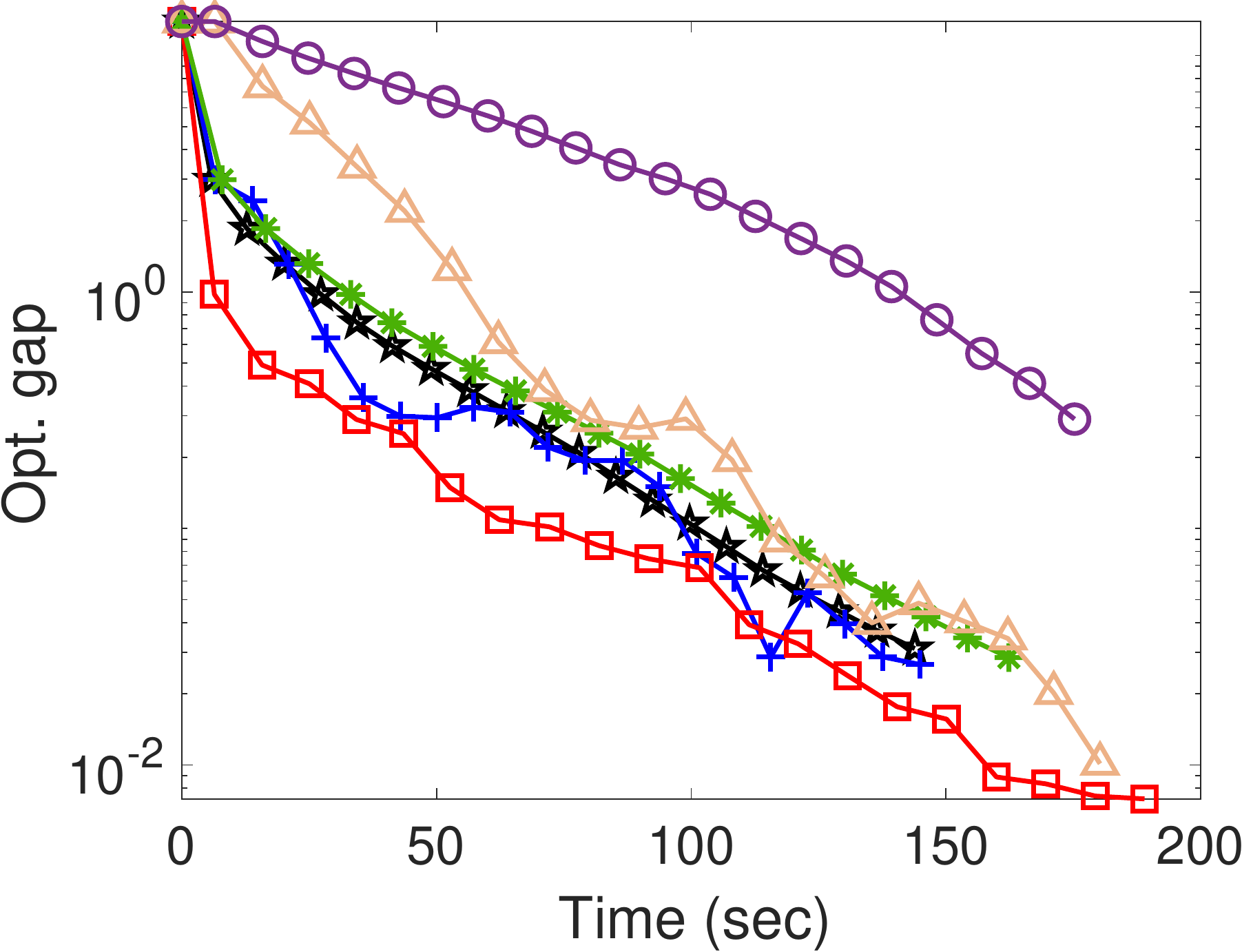}}\\
    \subfigure[\textit{gisette} ($\lambda=10^{-3}$)]{\includegraphics[width = 4.3cm, height = 3.3cm]{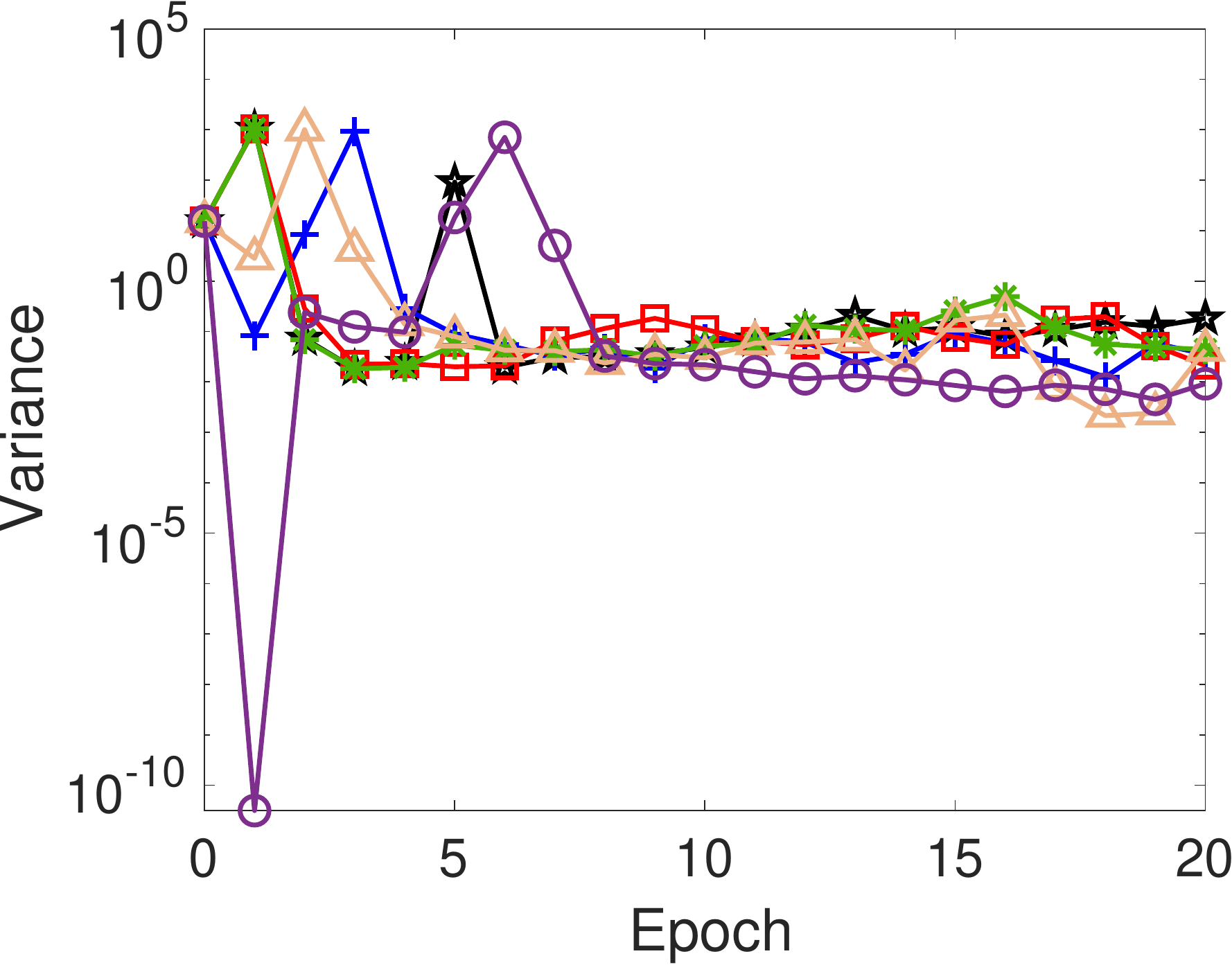}}
    \subfigure[\textit{gisette}($\lambda=10^{-4}$)]{\includegraphics[width = 4.3cm, height = 3.3cm]{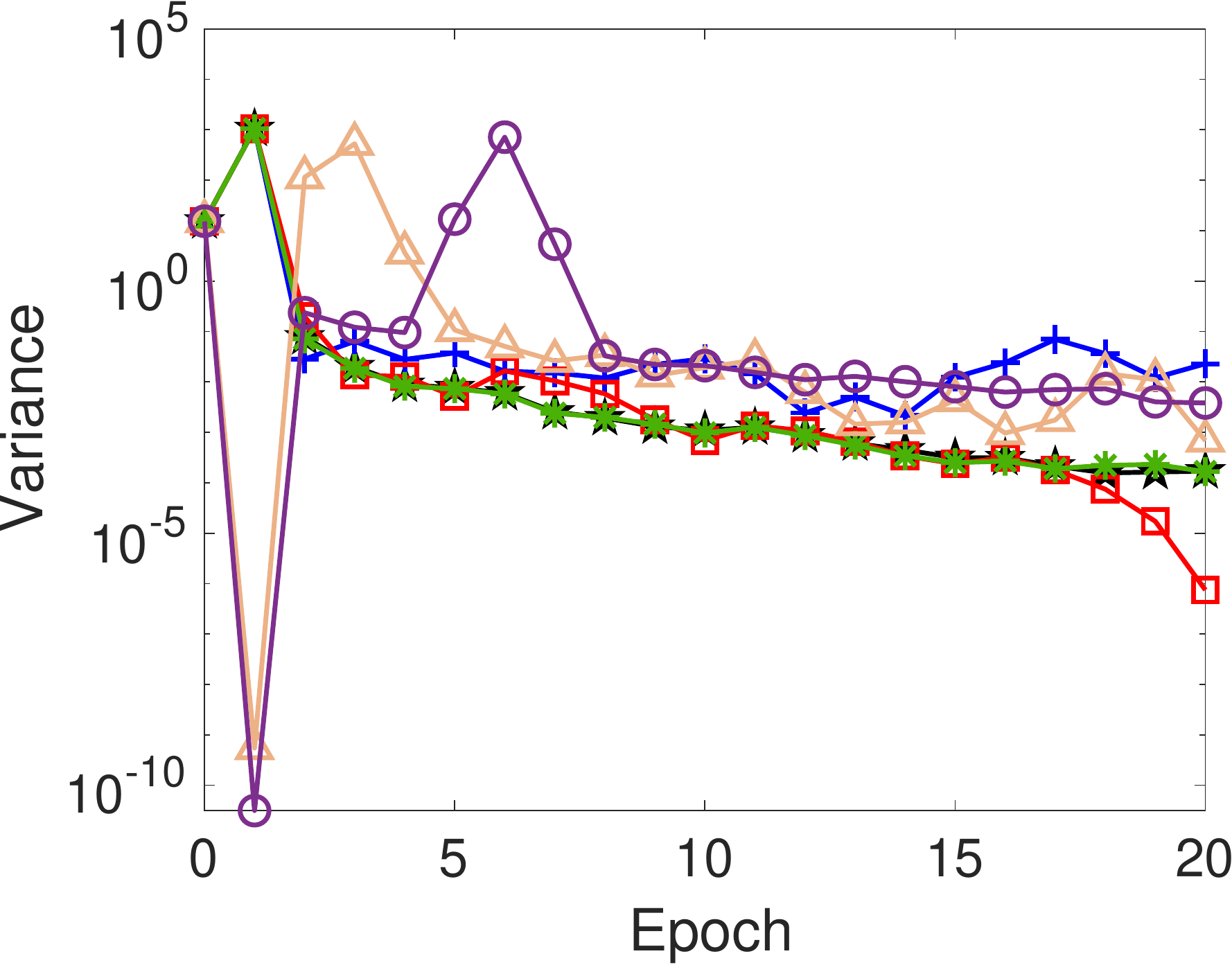}}
    \subfigure[\textit{gisette} ($\lambda=10^{-5}$)]{\includegraphics[width = 4.3cm, height = 3.3cm]{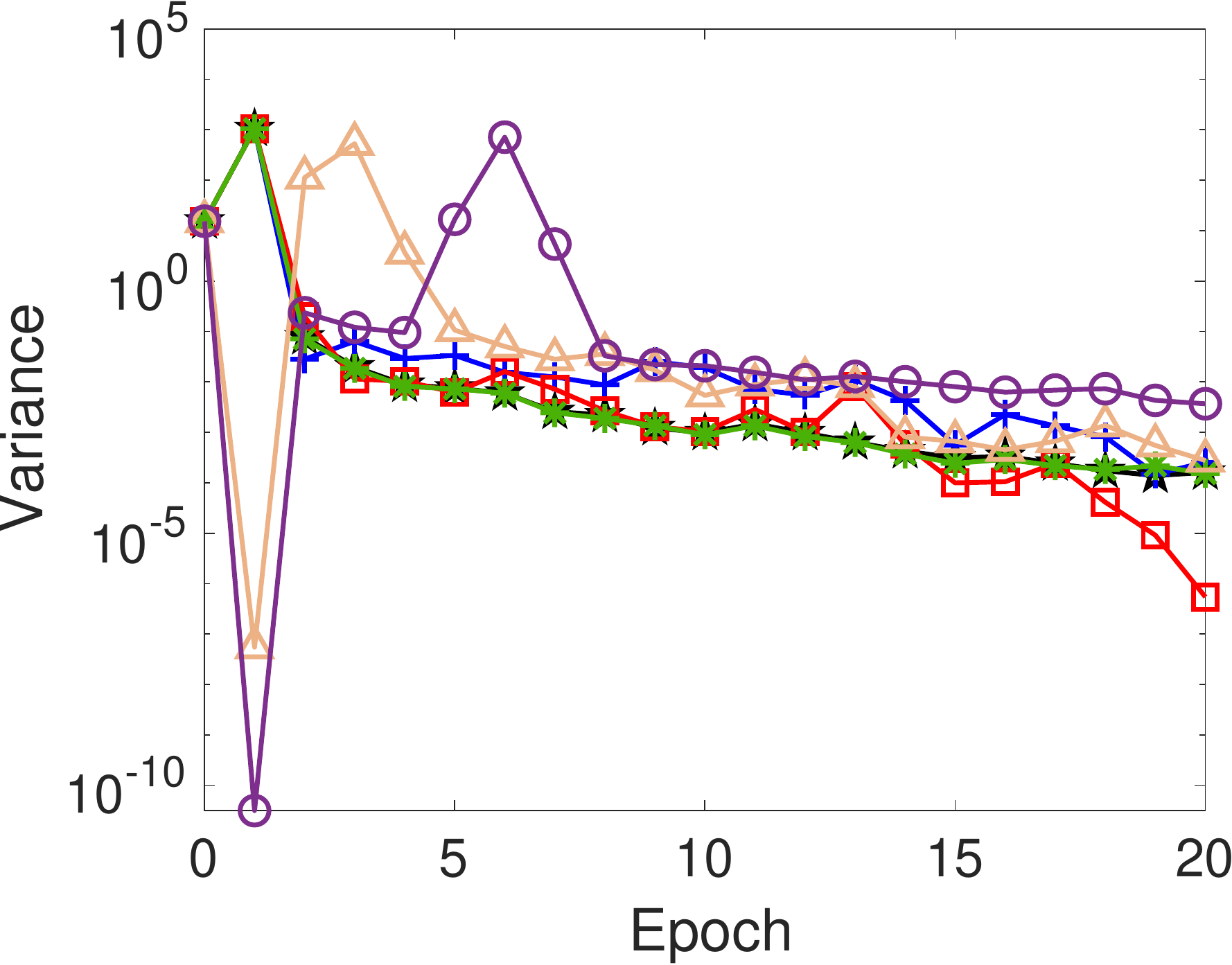}}
    \caption{Comparison on \textit{gisette} dataset}
    \label{fig:gisette}
\end{figure}
\begin{figure}[!t]
    \centering
    {\includegraphics[scale=0.35]{Legend_Main_aug22.png}} \\
    {\includegraphics[width = 4.3cm, height = 3.3cm]{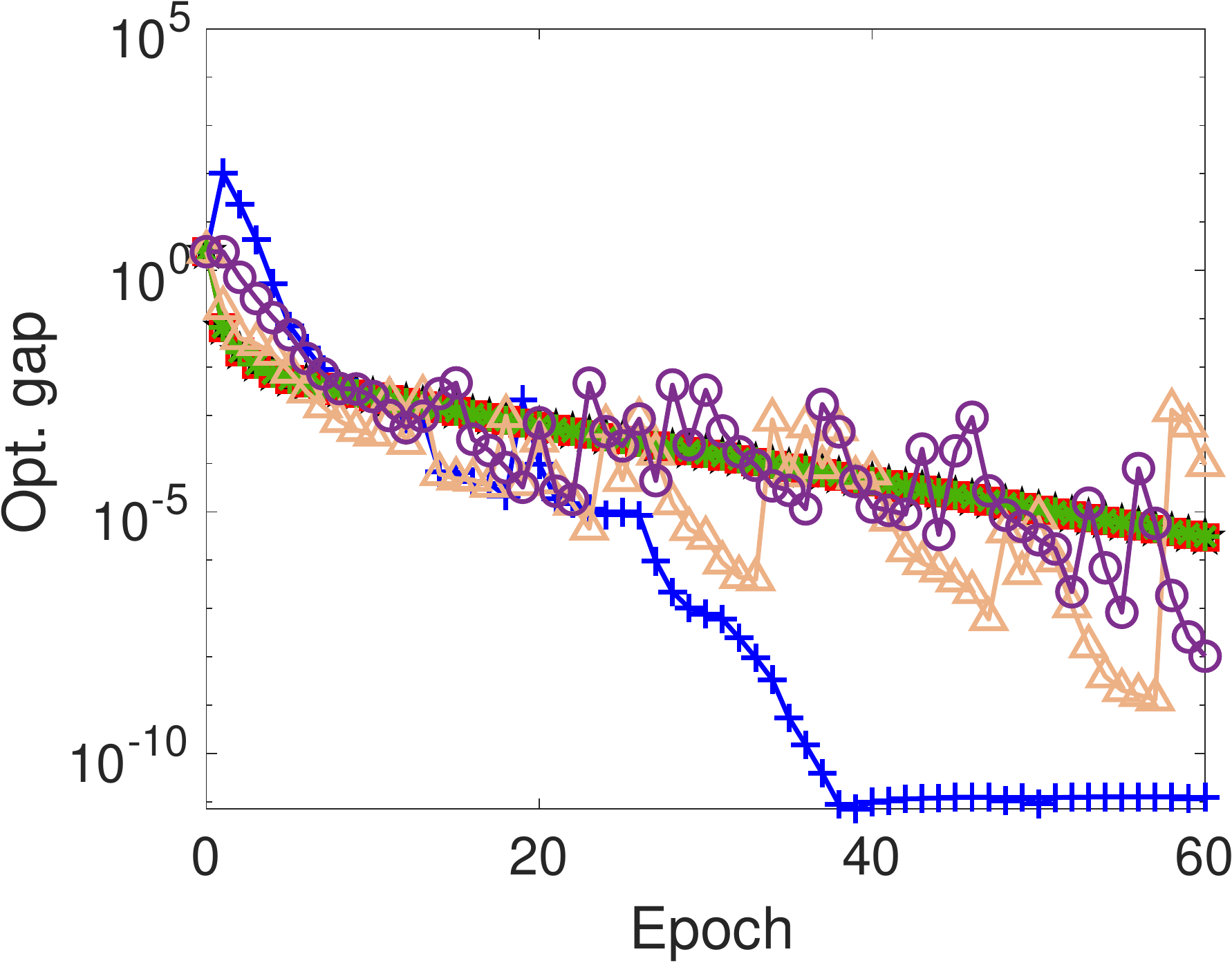}}
    {\includegraphics[width = 4.3cm, height = 3.3cm]{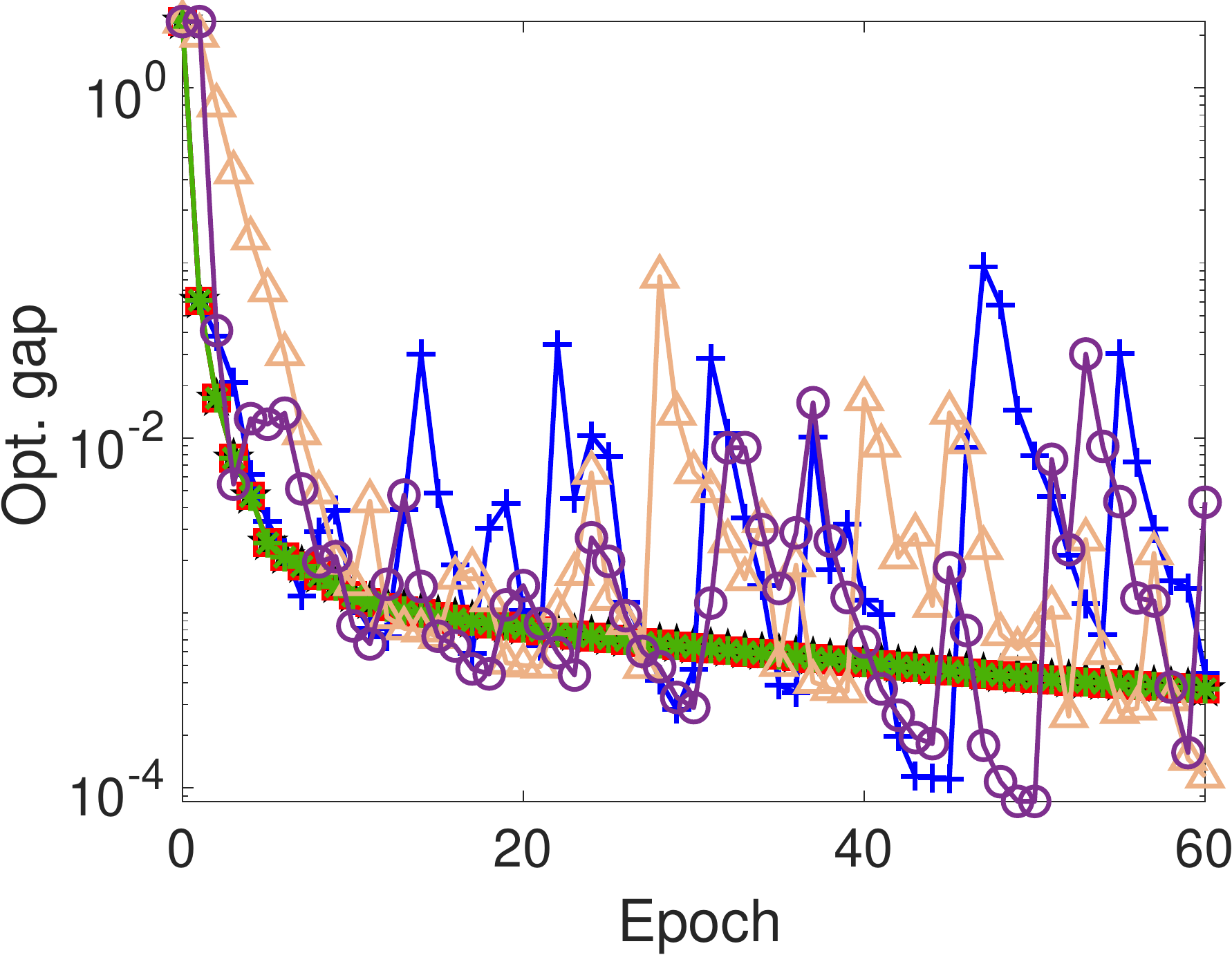}}
    {\includegraphics[width = 4.3cm, height = 3.3cm]{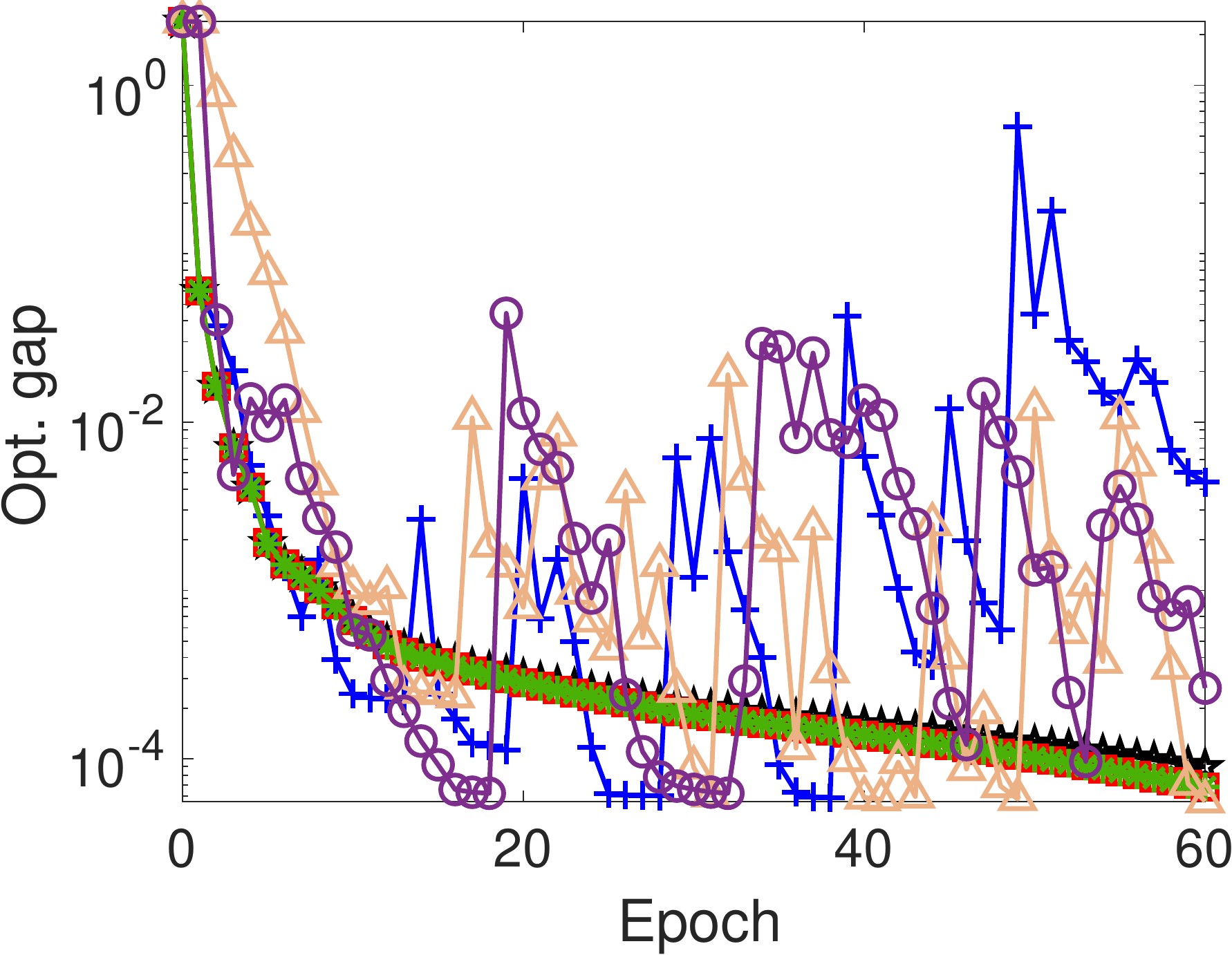}}\\
    {\includegraphics[width = 4.3cm, height = 3.3cm]{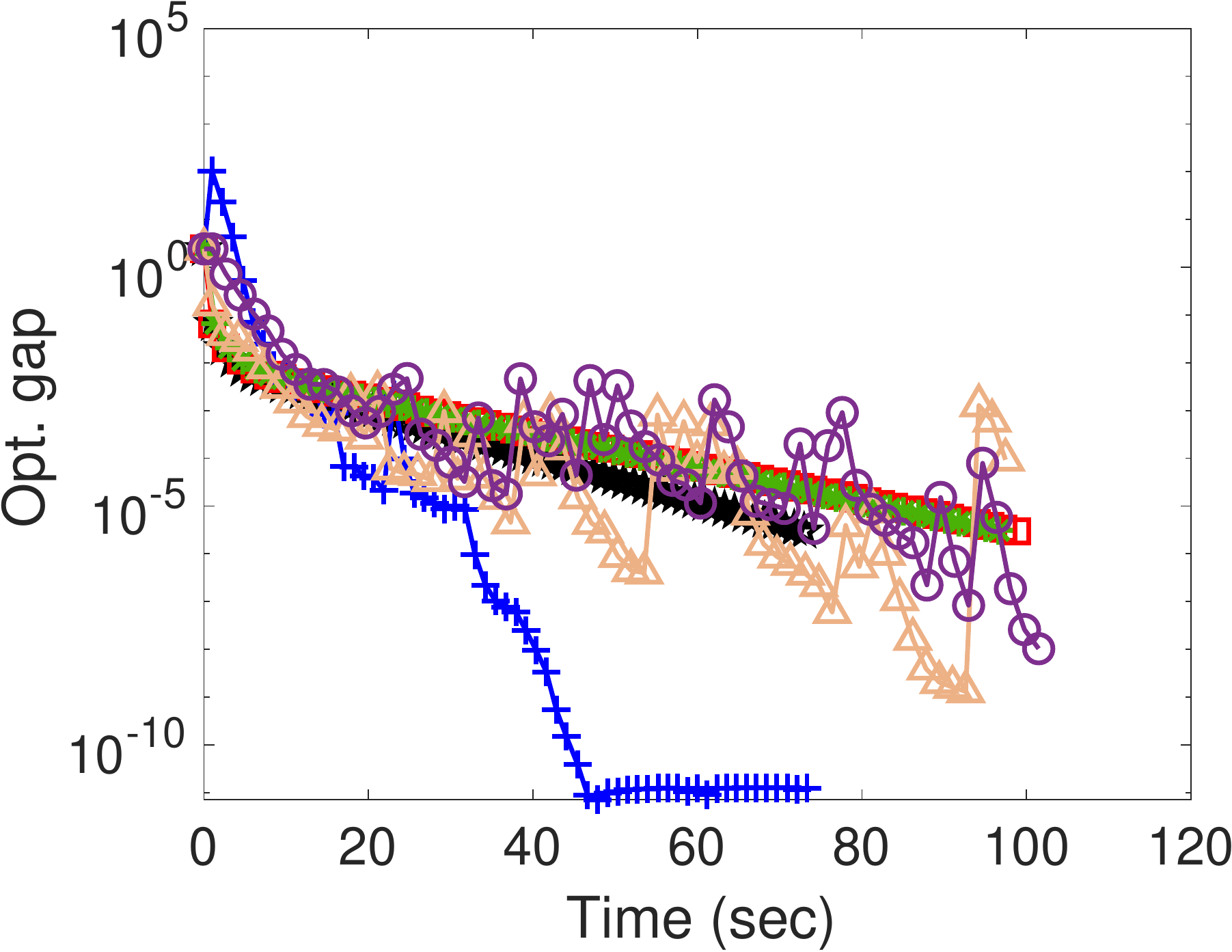}}
    {\includegraphics[width = 4.3cm, height = 3.3cm]{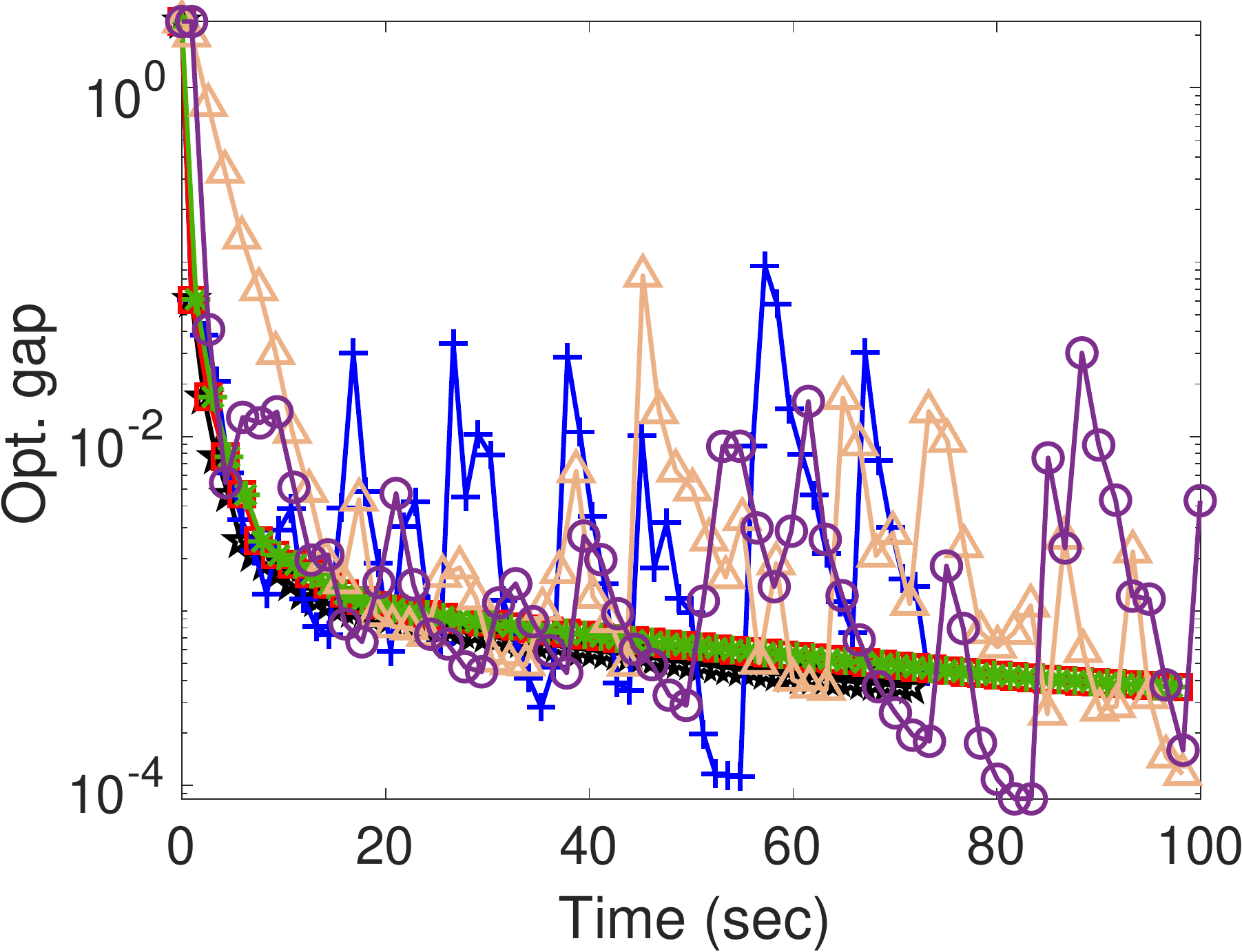}}
    {\includegraphics[width = 4.3cm, height = 3.3cm]{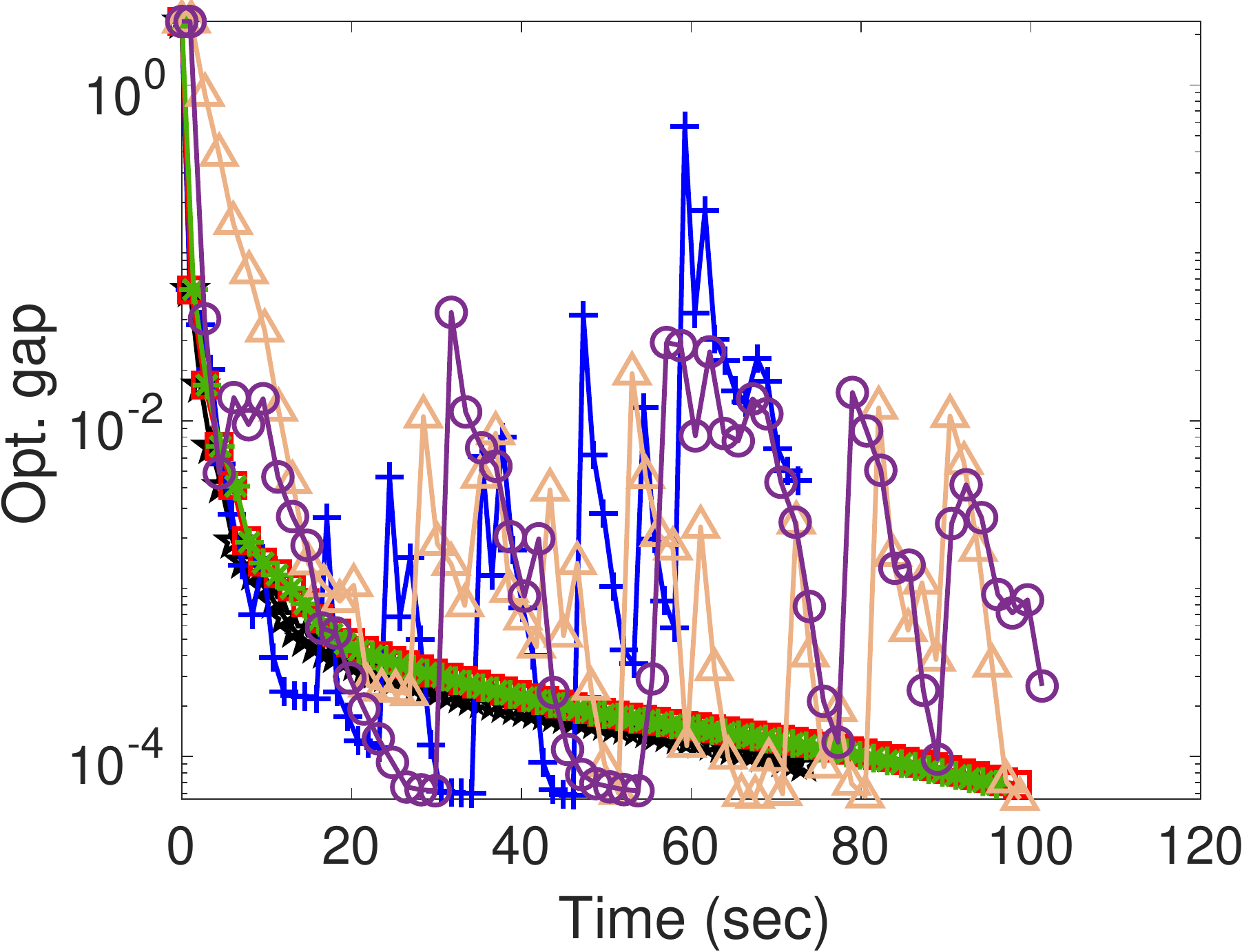}}\\
    \subfigure[\textit{adult} ($\lambda=10^{-3}$)]{\includegraphics[width = 4.3cm, height = 3.3cm]{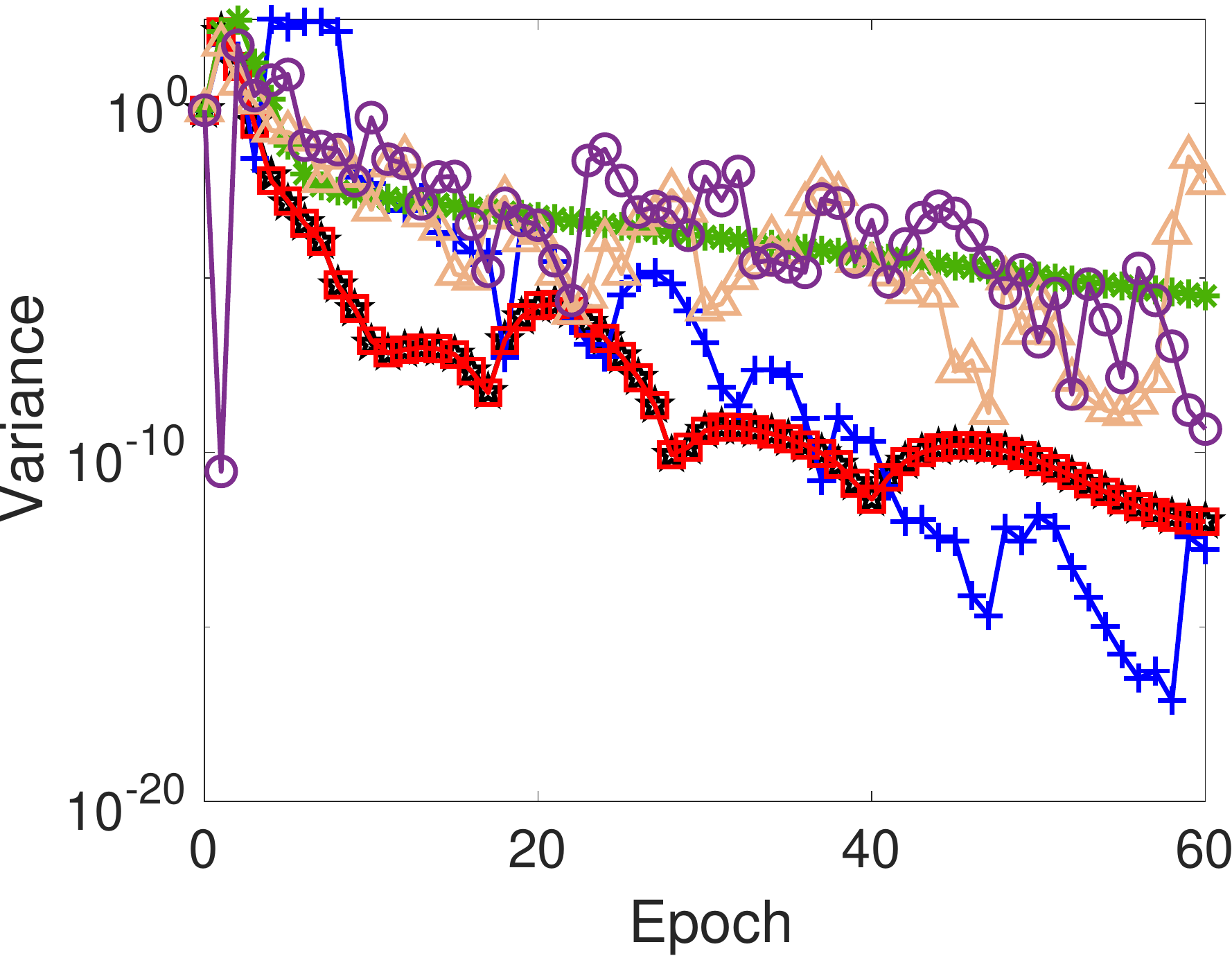}}
    \subfigure[\textit{adult} ($\lambda=10^{-4}$)]{\includegraphics[width = 4.3cm, height = 3.3cm]{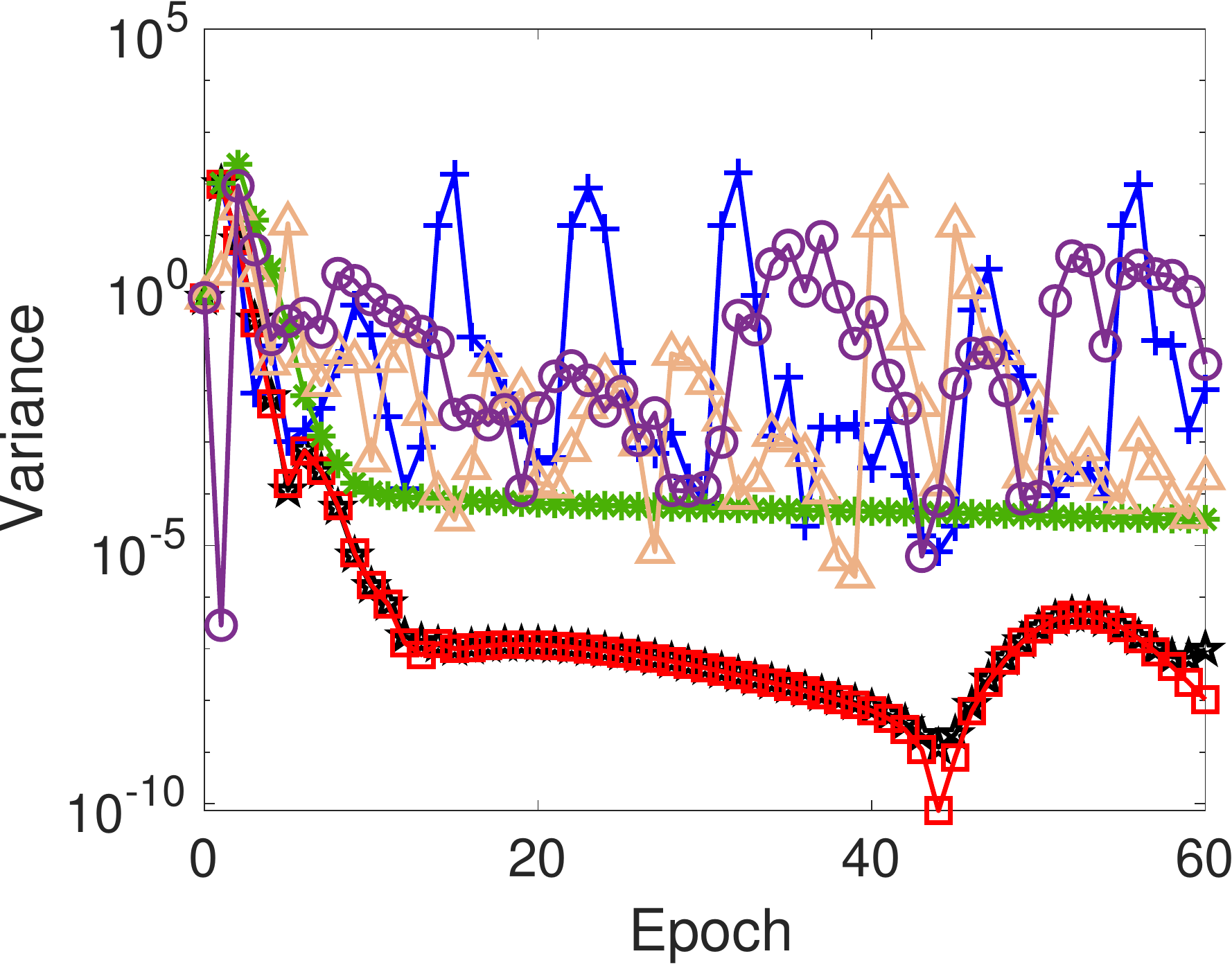}}
    \subfigure[\textit{adult} ($\lambda=10^{-5}$)]{\includegraphics[width = 4.3cm, height = 3.3cm]{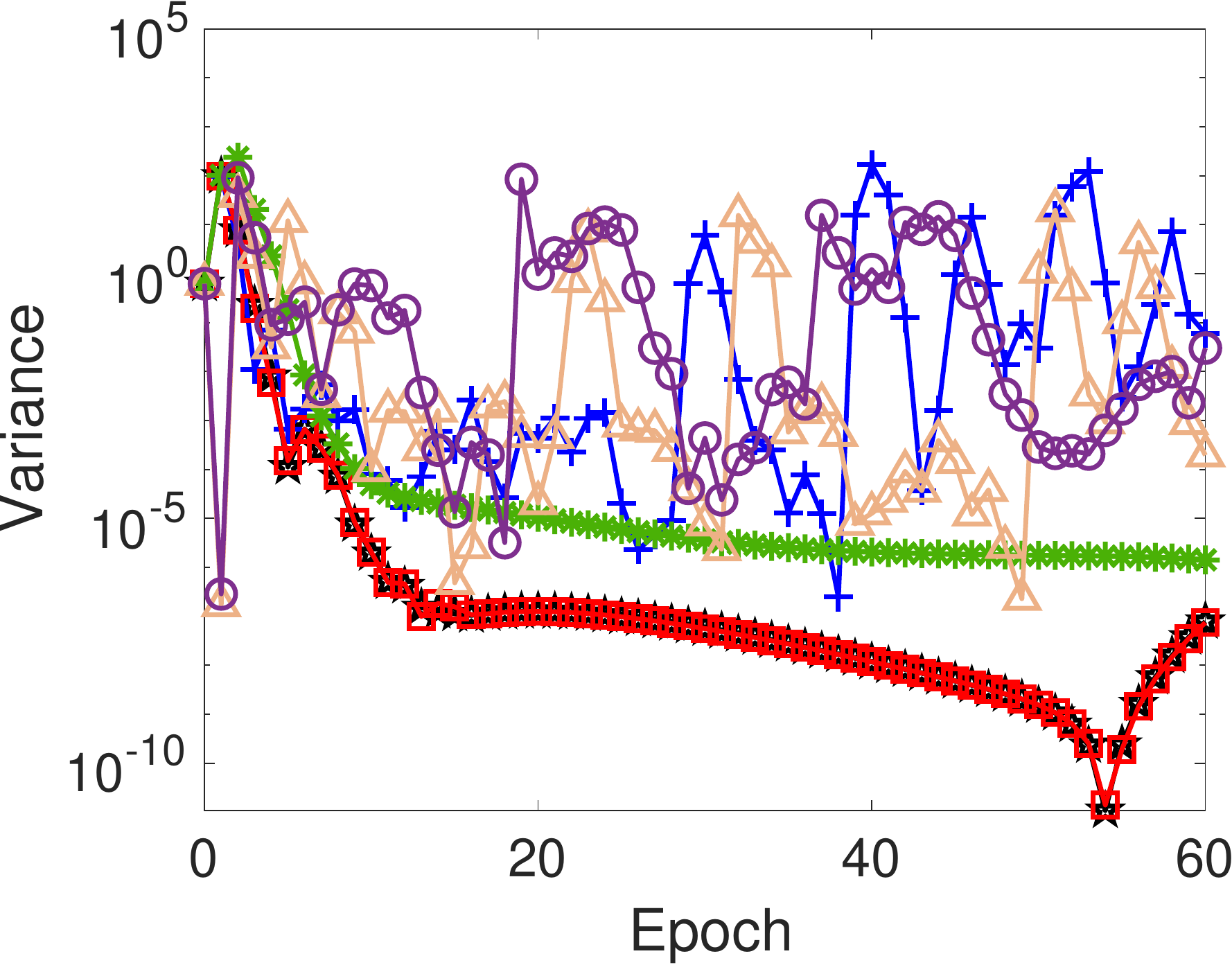}}
    \caption{Comparison on \textit{adult} dataset}
    \label{fig:adult}
\end{figure}
\begin{figure}[!h]
    \centering
    {\includegraphics[scale=0.35]{Legend_Main_aug22.png}} \\
    {\includegraphics[width = 4.3cm, height = 3.3cm]{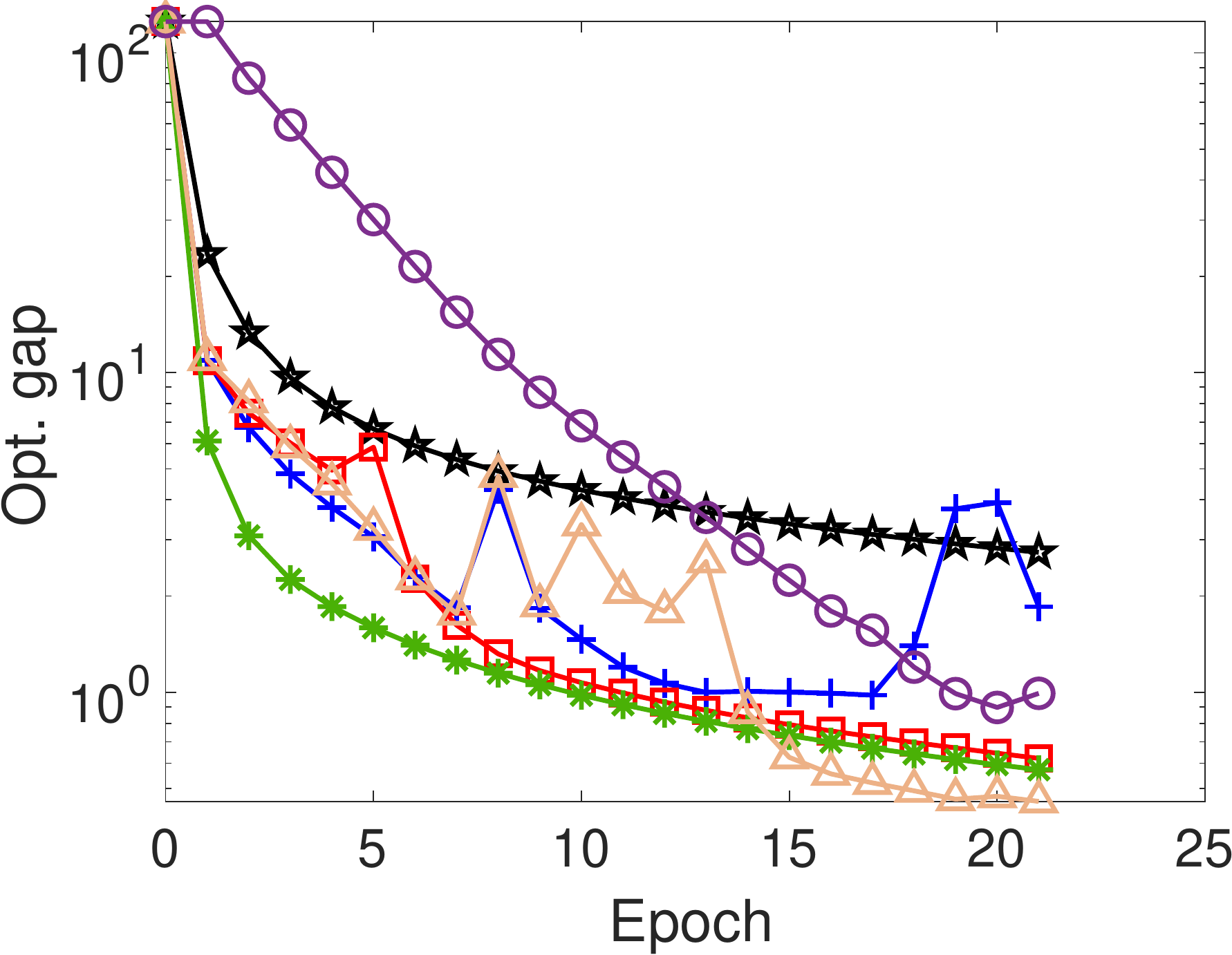}}
    {\includegraphics[width = 4.3cm, height = 3.3cm]{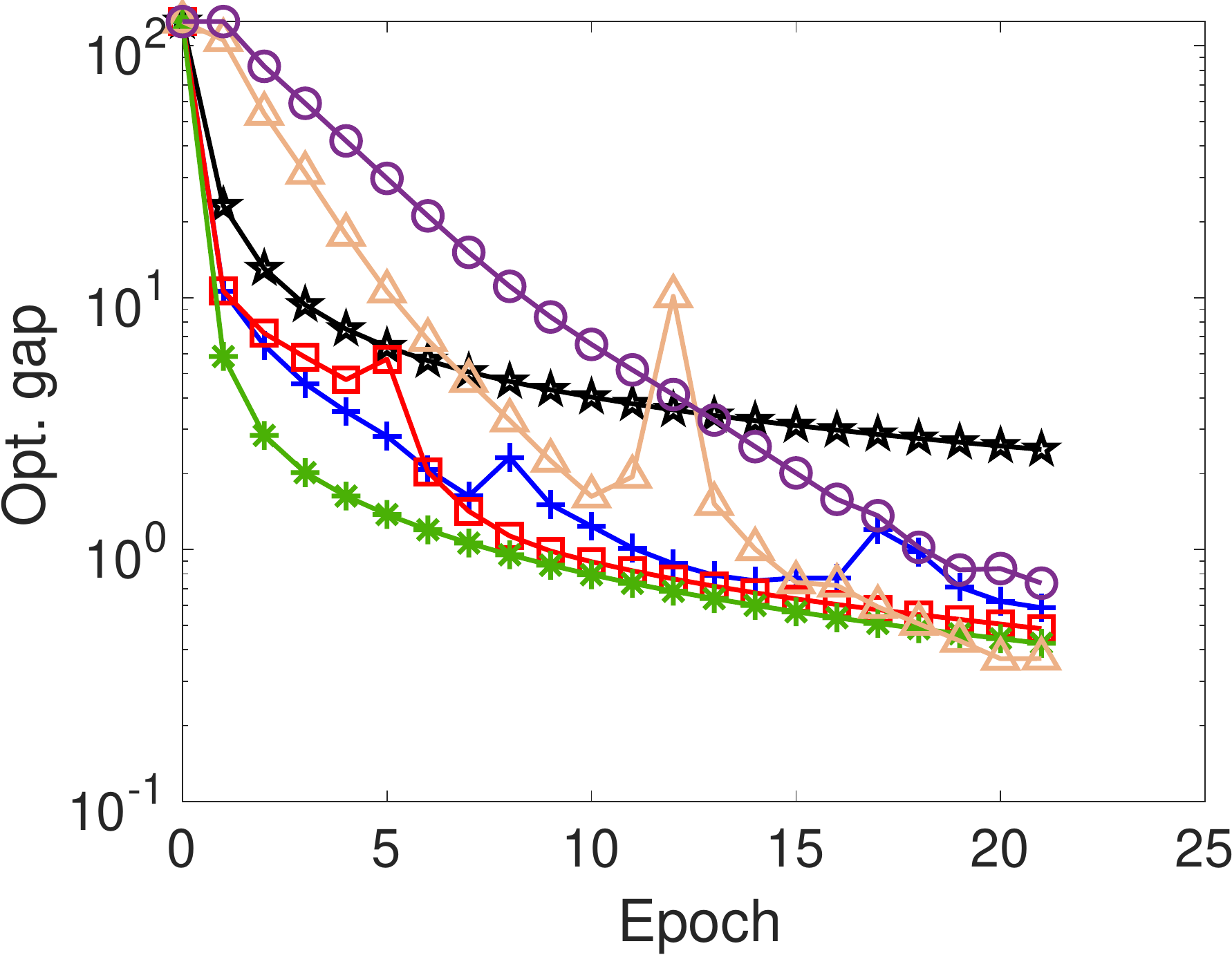}}
    {\includegraphics[width = 4.3cm, height = 3.3cm]{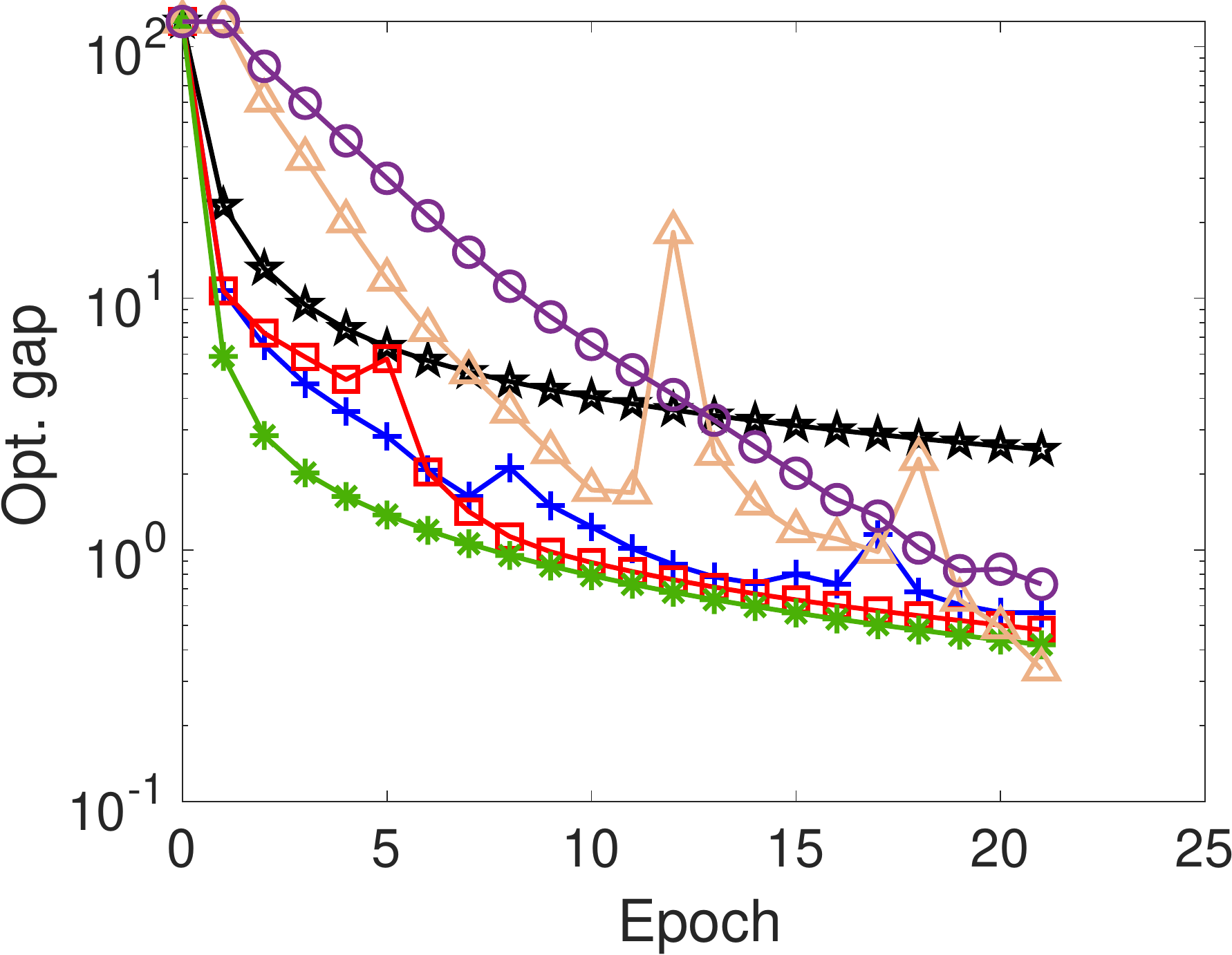}}\\
    {\includegraphics[width = 4.3cm, height = 3.3cm]{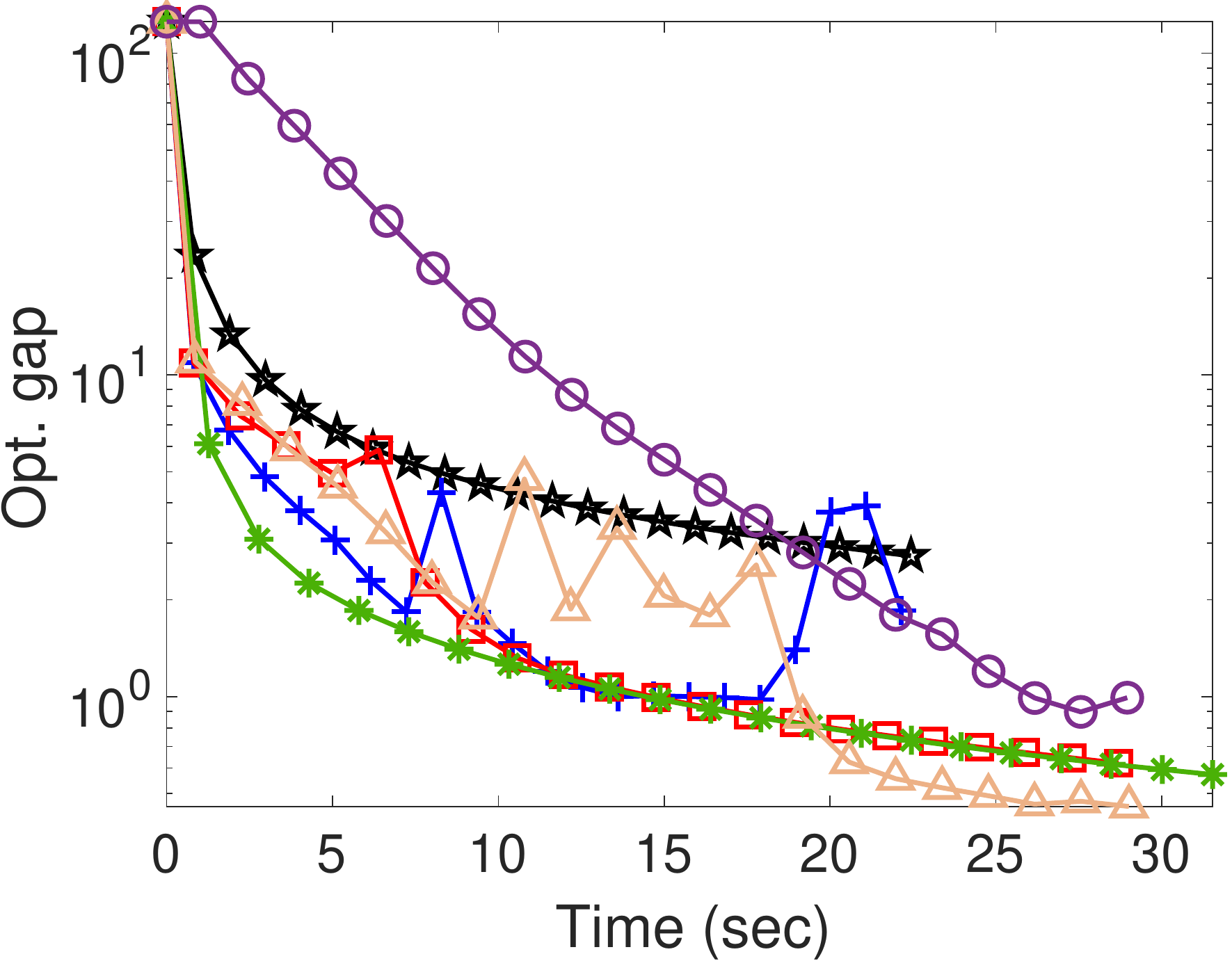}}
    {\includegraphics[width = 4.3cm, height = 3.3cm]{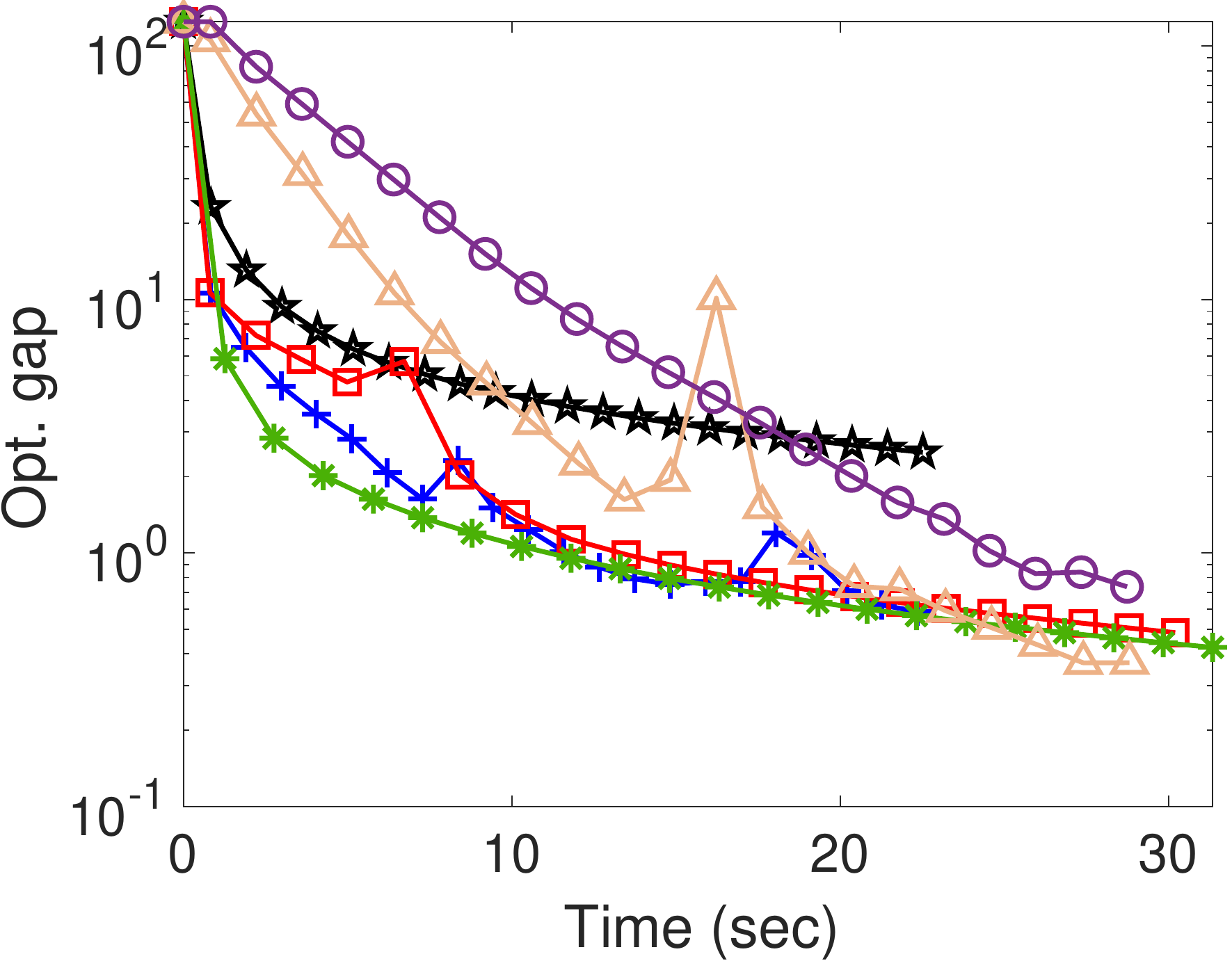}}
    {\includegraphics[width = 4.3cm, height = 3.3cm]{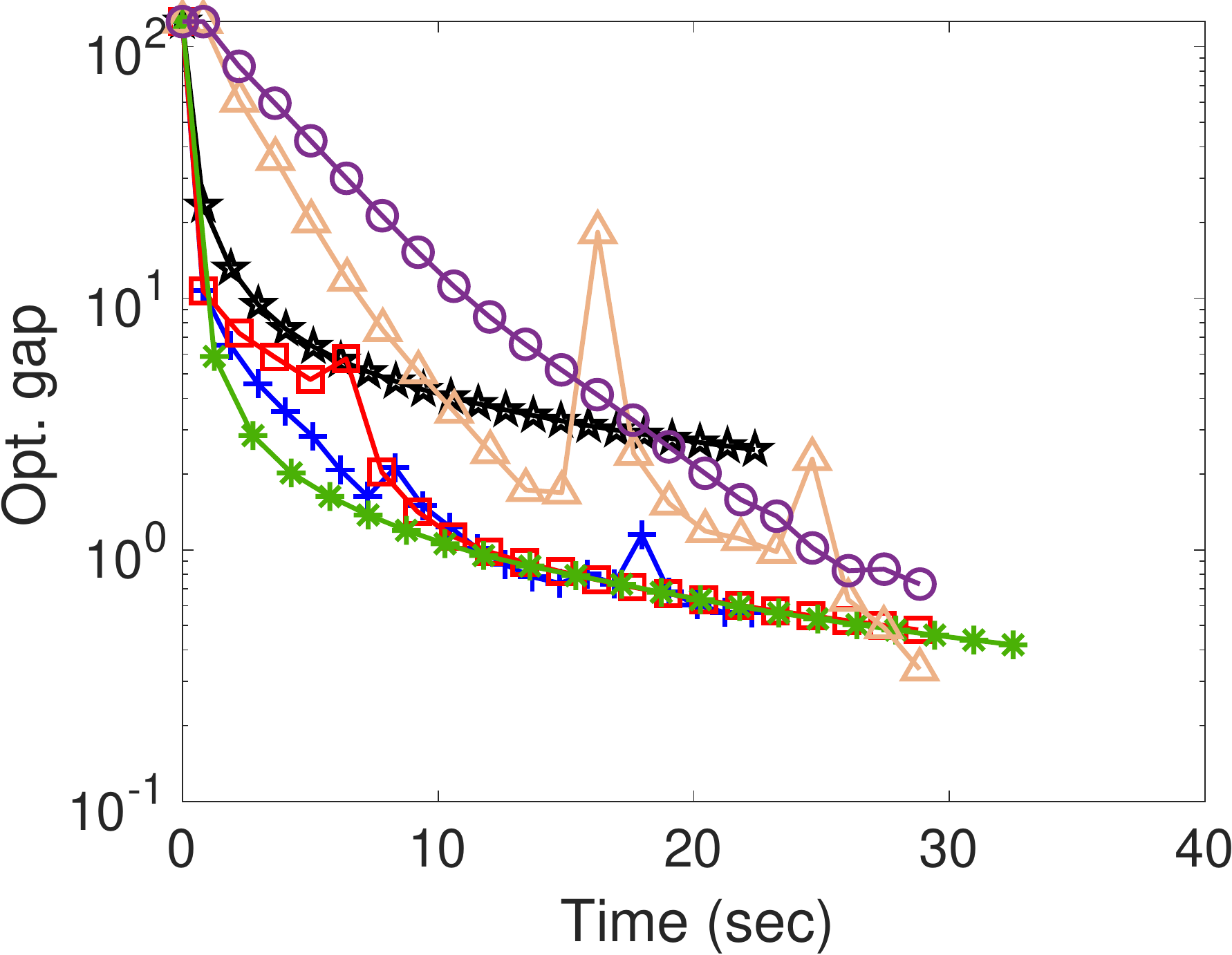}}\\
    \subfigure[\textit{mnist38} ($\lambda=10^{-3}$)]{\includegraphics[width = 4.2cm, height = 3.3cm]{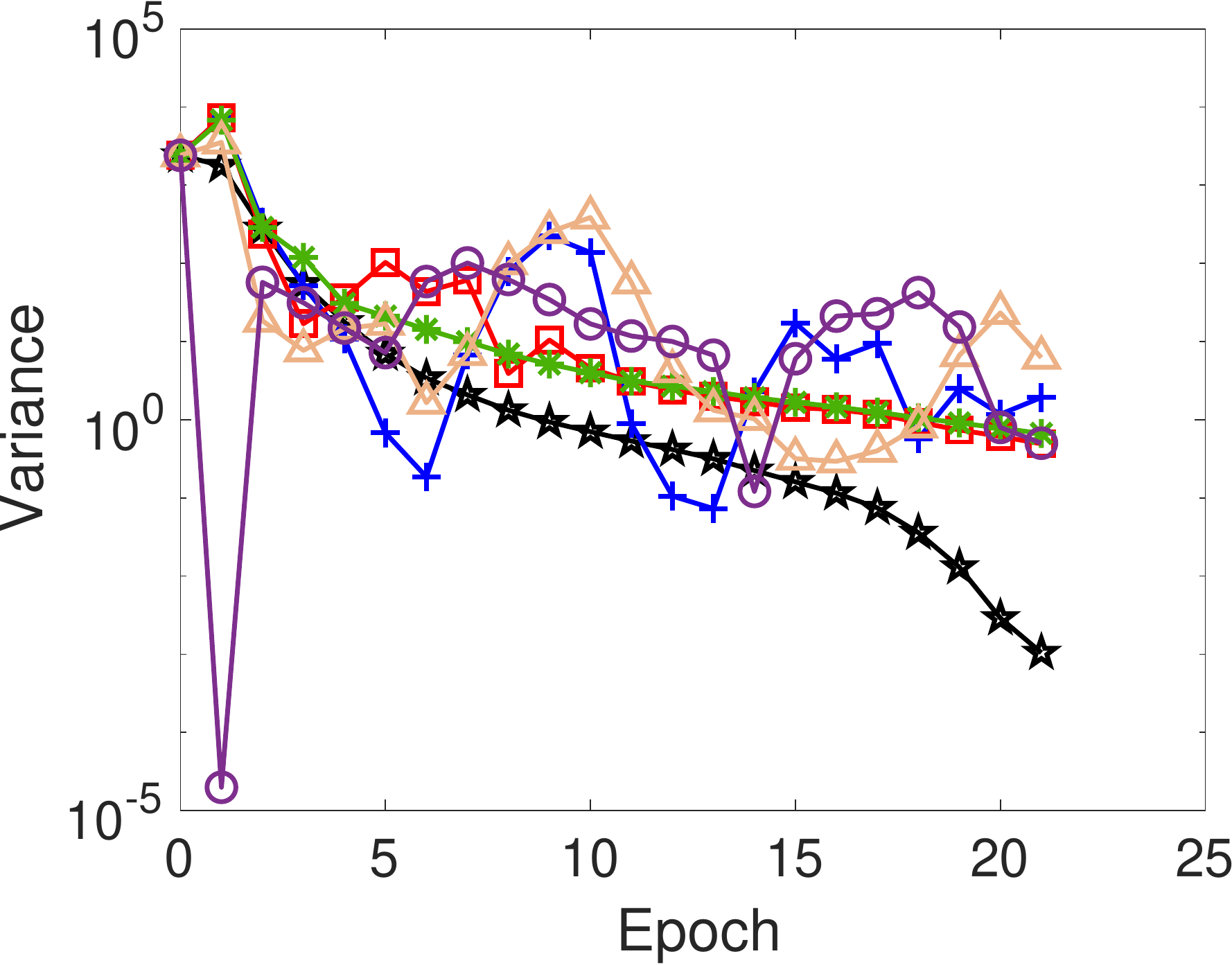}}
    \subfigure[\textit{mnist38} ($\lambda=10^{-4}$)]{\includegraphics[width = 4.2cm, height = 3.3cm]{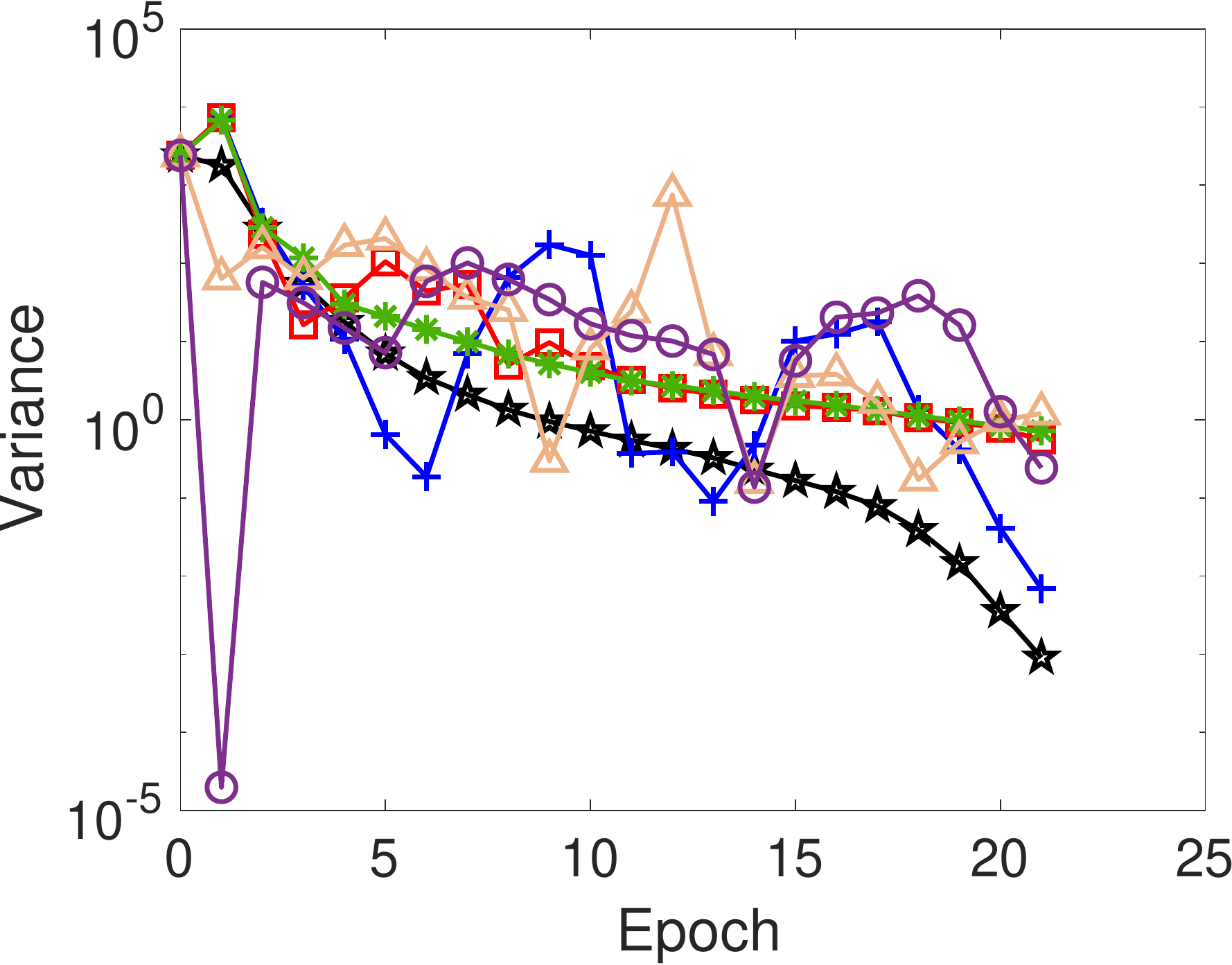}}
    \subfigure[\textit{mnist38} ($\lambda=10^{-5}$)]{\includegraphics[width = 4cm, height = 3.3cm]{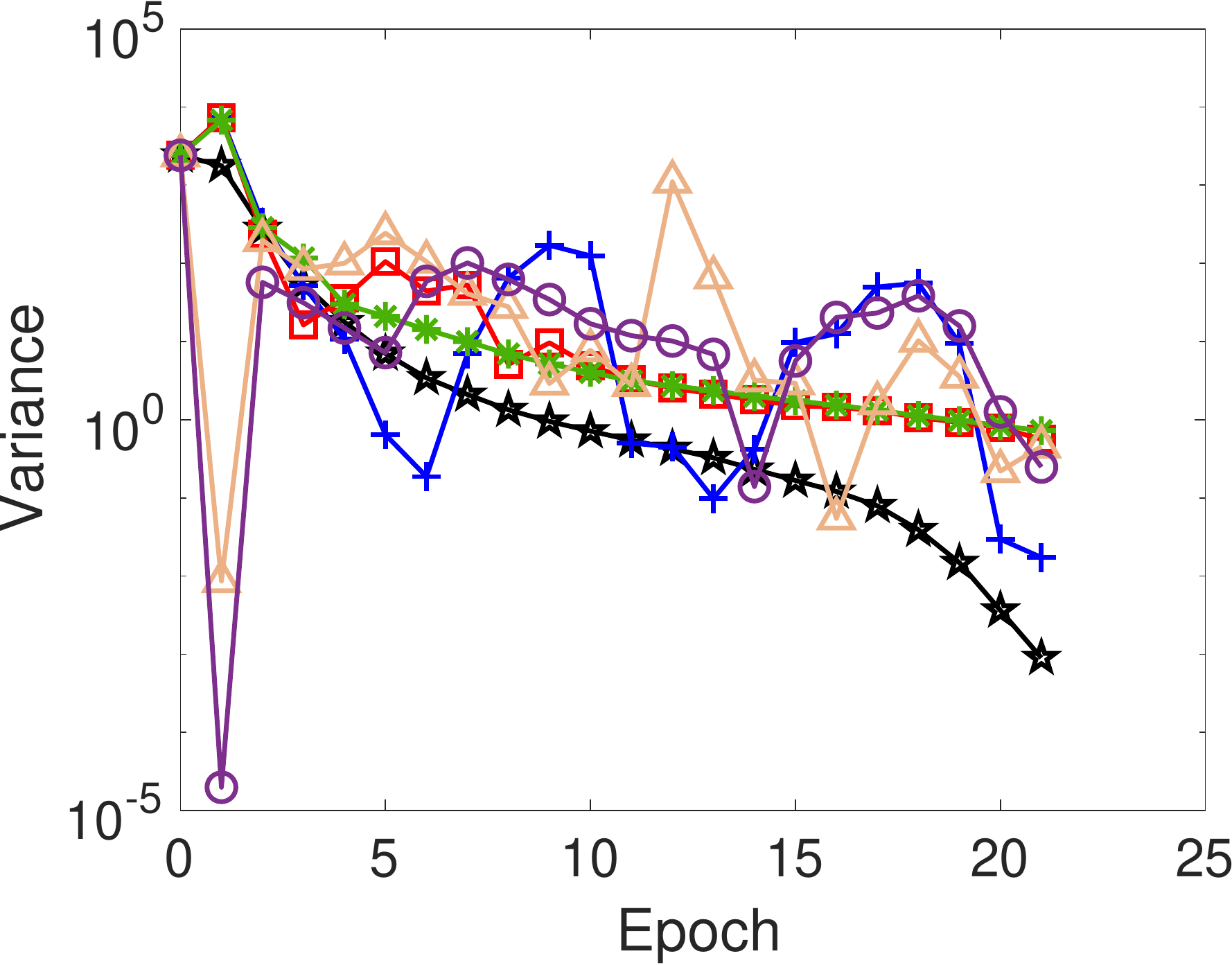}}
    \caption{Comparison on \textit{mnist38} dataset}
    \label{fig:mnist}
\end{figure}
\begin{figure}[!h]
    \centering
    {\includegraphics[scale=0.35]{Legend_Main_aug22.png}} \\
    {\includegraphics[width = 4.5cm, height = 3.5cm]{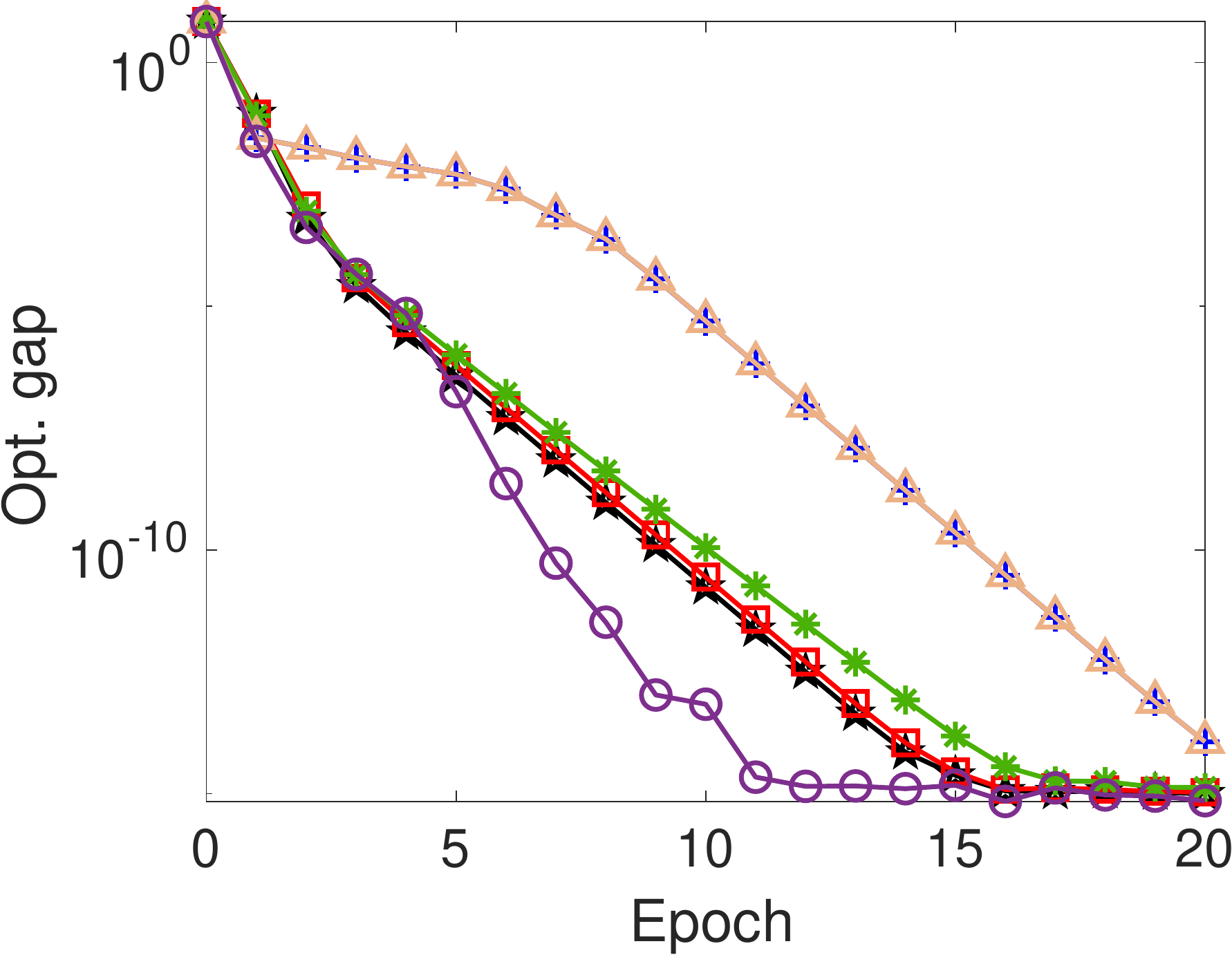}}\hspace{2cm}
    {\includegraphics[width = 4.5cm, height = 3.5cm]{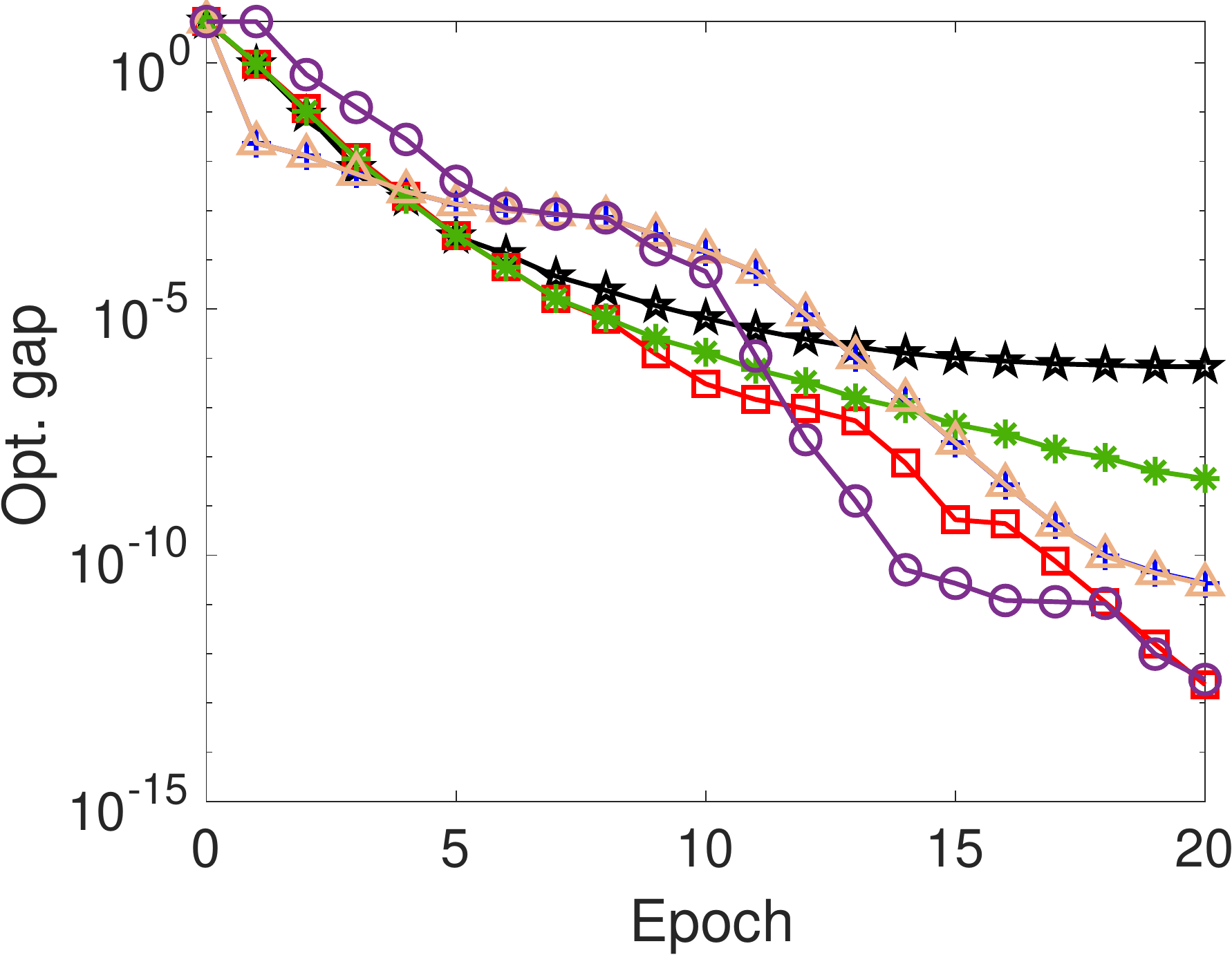}}\\
    {\includegraphics[width = 4.5cm, height = 3.5cm]{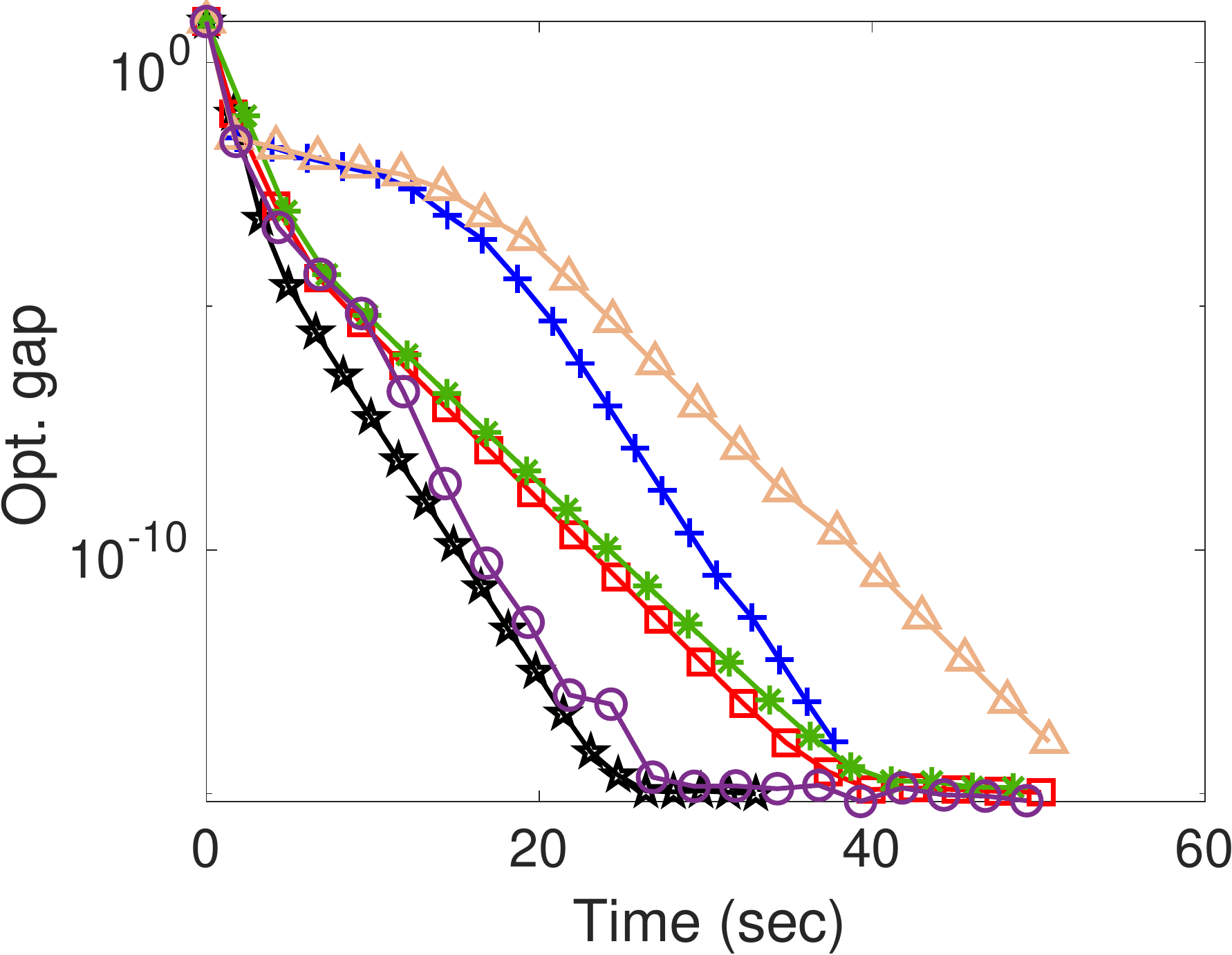}}\hspace{2cm}
    {\includegraphics[width = 4.5cm, height = 3.5cm]{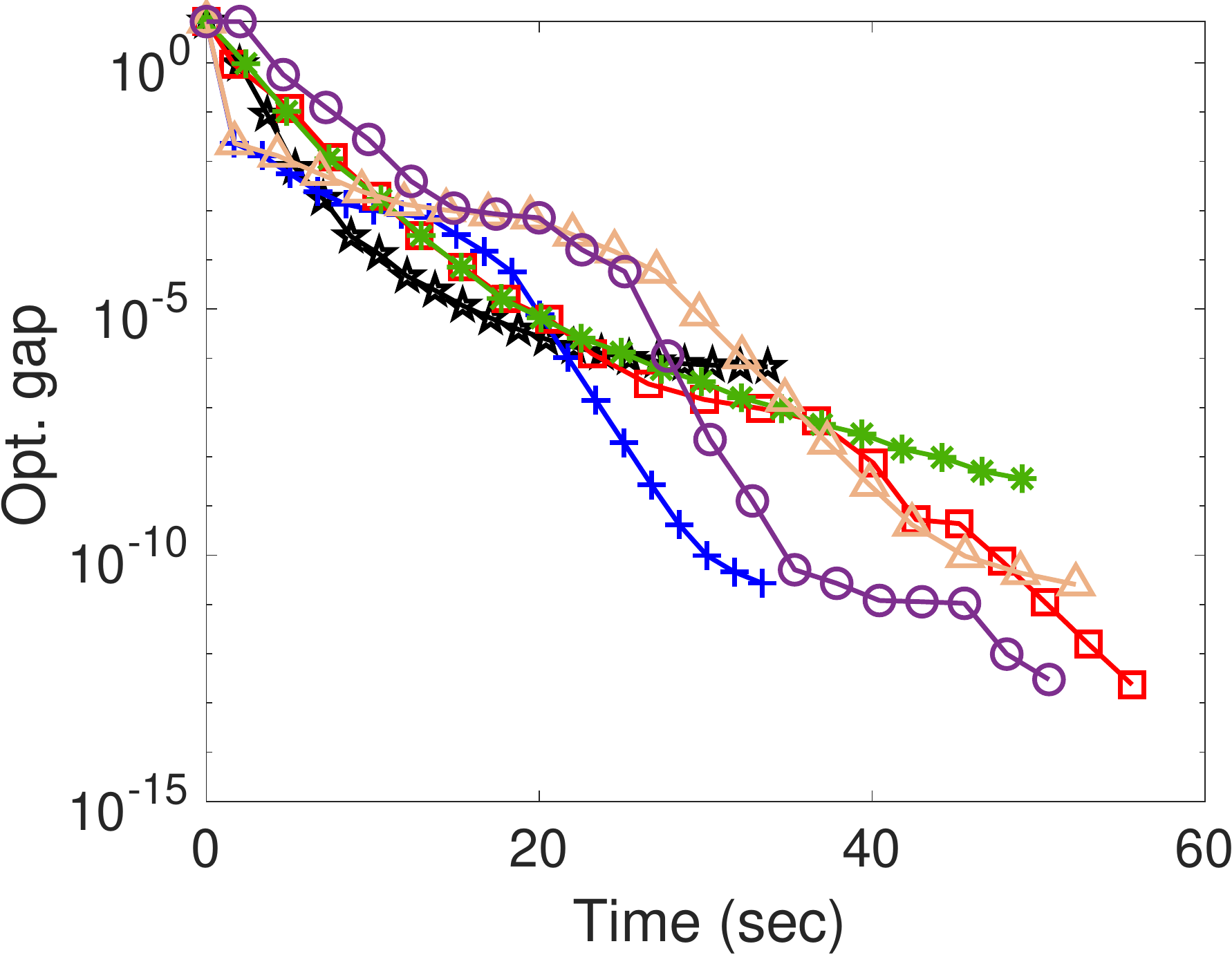}}\\
    \subfigure[\textit{ijcnn1} ($\lambda=10^{-2}$)]{\includegraphics[width = 4.5cm, height = 3.5cm]{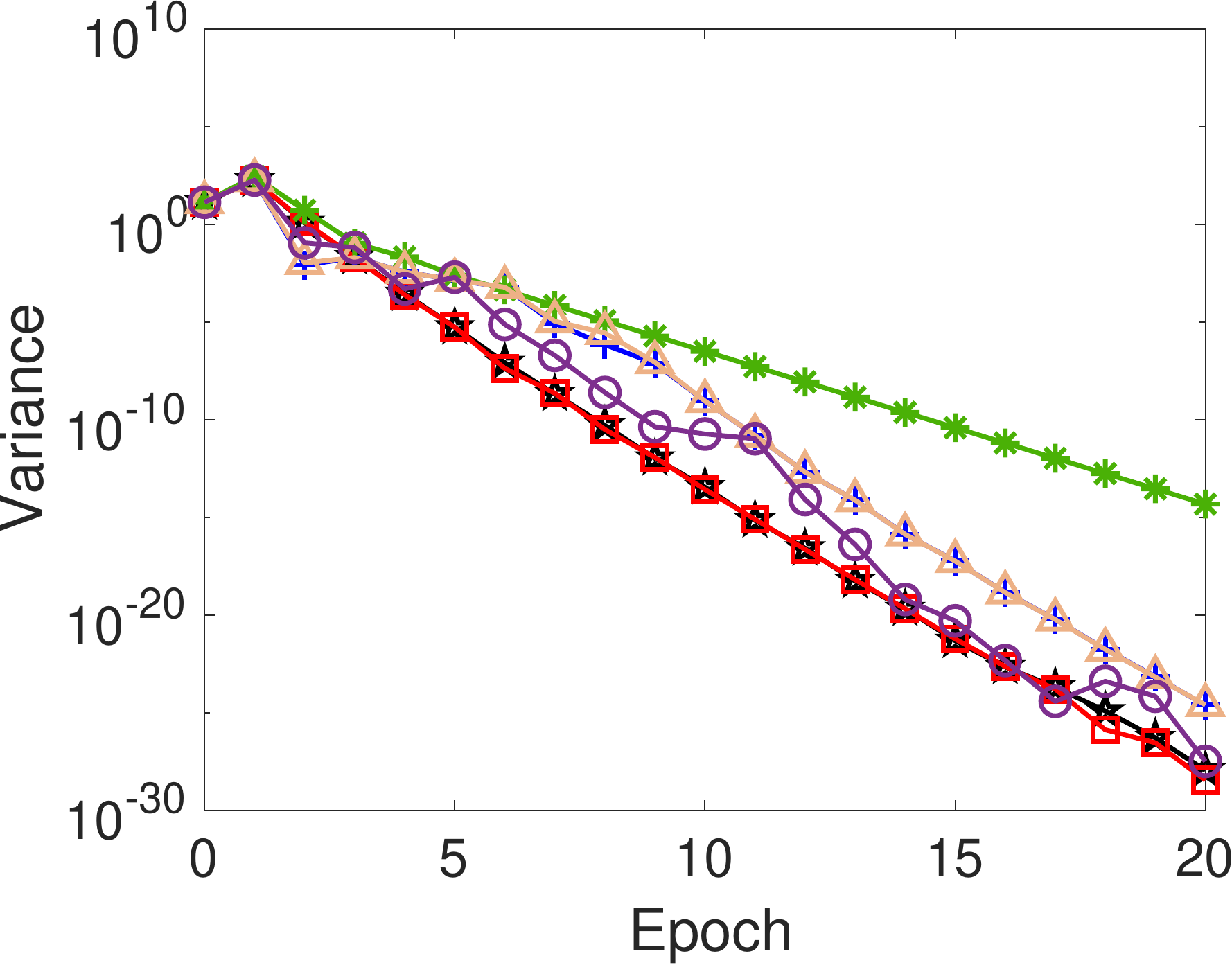}}\hspace{2cm}
    \subfigure[\textit{ijcnn1} ($\lambda=10^{-3}$)]{\includegraphics[width = 4.5cm, height = 3.5cm]{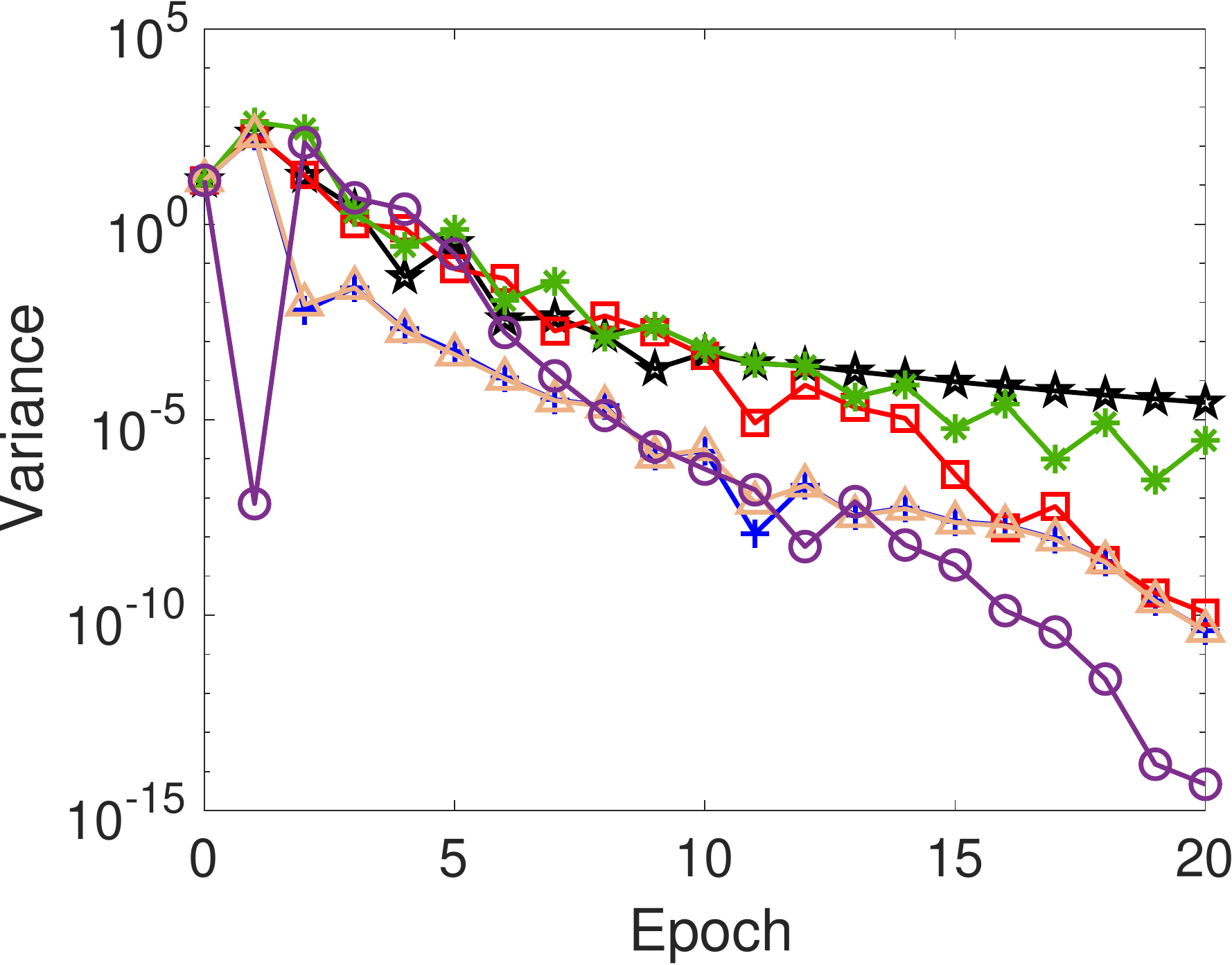}}
    \caption{Comparison on \textit{ijcnn1} dataset}
    \label{fig:ijcnn1}
\end{figure}
 \clearpage
\section{Conclusion}
 In this paper we proposed the method to reducing the variance of stochastic gradients using the Barzilai-Borwein approximation. We then incorporate the BB-step size of SVRG-BB~\cite{SVRGBB_NIPS2016} with SVRG-2BB. When the objective function is strongly convex, we proved the linear convergence of Algorithm~\ref{alg:Algorithm_1} which includes SVRG-2BB, and SVRG-2~\cite{GOWER18a}, diagonal Hessian approximation SVRG-2D~\cite{GOWER18a}, SVRG-2BBS. We conducted the numerical experiments on the benchmark datasets and show that the proposed method with constant step size performs better than the existing variance reduced methods for some test problems.  In terms of optimal cost with respect to both CPU time and epoch, experimental results on benchmark datasets show that the SVRG-2BB performs stable.

\bibliographystyle{plain}
\bibliography{arxiv}
\end{document}